\documentclass[11pt]{article}
\usepackage{amsmath,amssymb,amsthm,mathrsfs}
\renewcommand{\mathcal}{\mathscr}

\newcommand{\cL}{\mathcal{L}}
\newcommand{\cT}{\mathcal{T}}
\newcommand{\cX}{\mathcal{X}}
\newcommand {\IN}{\mathbb{N}}
\newcommand {\IR}{\mathbb{R}}

\newcommand{\s}{\sigma}

\renewcommand{\P}{\mathrm{P}}
\newcommand {\E}{{\mathrm E}}

\newcommand{\dimh}{\dim_{_{\rm H}}}

\newcommand{\R}{\mathbb{R}}
\newcommand{\e}{\epsilon}

\newtheorem{stat}{Statement}[section]
\newtheorem{prop}[stat]{Proposition}
\newtheorem{cor}[stat]{Corollary}
\newtheorem{thm}[stat]{Theorem}
\newtheorem{lem}[stat]{Lemma}

\theoremstyle{definition} \newtheorem{rema}[stat]{Remark}
\newtheorem{defi}[stat]{Definition}

\newtheorem{nota}[stat]{Notation}

\numberwithin{equation}{section}

\begin{document}
\title{\bf Hitting probabilities for systems of 
	non-linear stochastic heat equations with multiplicative noise}
\author{Robert C. Dalang$^{1,4}$, Davar Khoshnevisan$^{2,5}$, and Eulalia Nualart$^3$}
\date{}

\footnotetext[1]{Institut de Math\'ematiques, Ecole Polytechnique
	F\'ed\'erale de Lausanne, Station 8, CH-1015 
	Lausanne, Switzerland. \texttt{robert.dalang@epfl.ch}}

\footnotetext[2]{Department of Mathematics, The 
	University of Utah, 155 S. 1400 E. Salt Lake City, 
	UT 84112-0090, USA. \texttt{davar@math.utah.edu}}

\footnotetext[3]{Institut Galil\'ee, Universit\'e
	Paris 13, 93430 Villetaneuse, France.
	\texttt{nualart@math.univ-paris13.fr}}
	
\footnotetext[4]{Supported in part by the Swiss National Foundation for Scientific Research.}%

\footnotetext[5]{Research supported in part by a grant
	from the US National Science Foundation.}%

	\maketitle

\begin{abstract}
We consider a system of $d$ non-linear stochastic heat equations in spatial
dimension $1$ driven by $d$-dimensional space-time white noise. The non-linearities appear
both as additive drift terms and as multipliers of the noise. Using techniques of
Malliavin calculus, we establish upper and lower bounds on the one-point density of the
solution $u(t,x)$, and upper bounds of Gaussian-type on the two-point density of
$(u(s,y),u(t,x))$. In particular, this estimate quantifies how this density degenerates as
$(s,y) \to (t,x)$. From these results, we deduce upper and lower bounds on hitting
probabilities of the process $\{u(t\,,x)\}_{t \in \mathbb{R}_+, x \in [0,1]}$, in terms of respectively
Hausdorff measure and Newtonian capacity. These estimates make it possible to show that
points are polar when $d \geq 7$ and are not polar when $d \leq 5$. We also show that the
Hausdorff dimension of the range of the process is $6$ when $d>6$, and give analogous results
for the processes $t \mapsto u(t,x)$ and $x \mapsto u(t,x)$. Finally, we obtain the values of
the Hausdorff dimensions of the level sets of these processes.
\end{abstract}

\vskip 3cm {\it \noindent AMS 2000 subject classifications:}
Primary: 60H15, 60J45; Secondary: 60H07, 60G60. \\[1cm] \noindent {\it Key words
and phrases}. Hitting probabilities, stochastic heat equation, space-time white
noise, Malliavin calculus. \vfill\pagebreak

\section{Introduction and main results}

Consider the following system of non-linear stochastic partial differential equations ({\em spde}'s) 
\begin{equation}\label{e1}
    \frac{\partial u_i}{\partial t}(t\,,x) =
    \frac{\partial^2 u_i}{\partial x^2}(t\,,x)
    + \sum_{j=1}^d \s_{i,j}(u(t\,,x)) \dot{W}^j(t\,,x) +
    b_i(u(t\,,x)),
\end{equation}
for $1\le i\le d$, $t\in[0\,,T]$, and $x\in[0\,,1]$, where ${u}:=(u_1\,,\ldots,u_d)$, with initial conditions ${u}(0\,,x)={0}$ for all $x\in[0\,,1]$, and Neumann boundary conditions
\begin{equation}\label{neumann}
    \frac{\partial u_i}{\partial x}(t\,,0)=
    \frac{\partial u_i}{\partial x}(t\,,1) =0,
    \qquad 0\le t\le T.
\end{equation}
Here,
$\dot{W}:=(\dot{W}^1\,,\ldots,\dot{W}^d)$ is a vector of $d$ independent space-time white noises
on $[0\,,T]\times[0\,,1]$. For all $1\le i, j\le d$, $b_i, \sigma_{ij} : \mathbb{R}^d \rightarrow \mathbb{R}$ are globally Lipschitz functions. We set $b=(b_i)$, $\sigma = (\sigma_{ij})$. Equation (\ref{e1}) is formal: the rigorous formulation of \textsc{Walsh} \cite{Walsh:86} will be recalled in Section \ref{sec2}.

The objective of this paper is to develop a potential theory for the $\IR^d$-valued process $u=(u(t,x), \ t \geq 0, \ x \in (0,1))$.
In particular, given $A \subset \IR^d$, we want to determine whether the process $u$ visits (or hits) $A$ with positive probability.

The only potential-theoretic result that we are aware of for
systems of {\it non-linear} spde's with multiplicative noise
($\sigma$ non-constant) is \textsc{Dalang and Nualart} \cite{Dalang:04}, who study the case of
the reduced hyperbolic spde~on $\IR_+^2$ (essentially
equivalent to the wave equation in spatial dimension $1$):
\begin{equation*}
\frac{\partial^2 X^i_t}{\partial t_1 \partial t_2} =\sum_{j=1}^d
\s_{i,j}(X_{t})\frac{\partial^2 W_t^j}{\partial t_1 \partial
t_2} +b_i(X_{t}),
\end{equation*}
where $t=(t_1 \,,t_2) \in \mathbb{R}_+^2$, and $X_{t}^i=0$ if $t_1
t_2=0$, for all $1 \leq i \leq d$. There, Dalang and Nualart used
Malliavin calculus to show that the solution $(X_t)$ of this
spde~satisfies
\begin{equation*}
K^{-1} \text{Cap}_{d-4} (A) \leq \P \{\exists t \in [a,b]^2 : X_t \in A\} \leq K \text{Cap}_{d-4} (A),
\end{equation*}
where $\text{Cap}_\beta$ denotes the capacity with respect to the Newtonian $\beta$-kernel ${\rm K}_\beta(\cdot)$ (see (\ref{k})). This result, particularly the upper bound, relies heavily on properties of the underlying two-parameter filtration and uses Cairoli's maximal inequality for two-parameter processes.

    Hitting probabilities for systems of linear heat equations have been obtained in \textsc{Mueller and Tribe} \cite{Mueller:03}. For systems of non-linear stochastic heat equations with {\em additive noise,} that is, $\s$ in \eqref{e1} is a constant matrix, so \eqref{e1} becomes
\begin{equation}\label{eadd}
   \frac{\partial u_i}{\partial t}(t\,,x) =
    \frac{\partial^2 u_i}{\partial x^2}(t\,,x)
    + \sum_{j=1}^d \s_{i,j}\, \dot{W}^j(t\,,x) +
    b_i(u(t\,,x)),
\end{equation}
estimates on hitting probabilities have been obtained in \textsc{Dalang, Khoshnevisan and Nualart} \cite{Dalang:05}. That paper develops some general results that lead to upper and lower bounds on hitting probabilities for continuous two-parameter random fields, and then uses these, together with a careful analysis of the linear equation ($b \equiv 0$, $\s \equiv I_d$, where $I_d$ denotes the $d\times d$ identity matrix) and Girsanov's theorem, to deduce bounds on hitting probabilities for the solution to \eqref{eadd}.

   In this paper, we make use of the general results of \cite{Dalang:05}, but then, in order to handle the solution of \eqref{e1}, we use a very different approach. Indeed, the results of \cite{Dalang:05} require in particular information about the probability density function $p_{t,x}$ of the random vector $u(t,x)$. In the case of multiplicative noise, estimates on $p_{t,x}$ can be obtained via Malliavin calculus.


   We refer in particular on the results of \textsc{Bally and Pardoux} \cite{Bally:98}, who used Malliavin calculus in the case $d=1$ to prove that for any $t > 0$, $k \in \IN$ and $0 \leq x_1 < \cdots < x_k \leq 1,$ the law of $(u(t \,, x_1), \ldots, u(t \,, x_k))$ is absolutely continuous with respect to Lebesgue measure, with a smooth and strictly positive density on $\{\sigma \not= 0\}^k$, provided $\sigma$ and $b$ are infinitely differentiable functions which are bounded together with their derivatives of all orders. A Gaussian-type lower bound for this density is established by \textsc{Kohatsu-Higa} \cite{Kohatsu:03} under a uniform ellipticity condition. \textsc{Morien} \cite{Morien:99} showed that the density function is also H\"{o}lder-continuous as a function of $(t \,,x).$

In this paper, we shall use techniques of Malliavin calculus to establish the following theorem. Let $p_{t,x}(z)$ denote the probability density function of the $\IR^d$-valued random vector $u(t \,,x) = (u_1(t \,,x), \ldots, u_d(t\,,x))$ and for $(s,y) \not= (t,x)$, let $p_{s,y; \, t,x}(z_1,z_2)$ denote the joint density of the $\IR^{2d}$-valued random vector
\begin{equation}\label{vu}
(u(s \, ,y), u(t \, ,x)) = (u_1(s \, ,y), \ldots, u_d(s\, ,y), u_1(t \, ,x), \ldots, u_d(t \, ,x))
\end{equation}
(the existence of $p_{t,x}(\cdot)$ is essentially a consequence of the result of \textsc{Bally and Pardoux} \cite{Bally:98}, see our Corollary \ref{cor4.3}; the existence of $p_{s,y; \, t,x}(\cdot, \cdot)$ is a consequence of Theorems \ref{3t1} and \ref{ga}).
\vskip 12pt
Consider the following two hypotheses on the coefficients of the system (\ref{e1}):
\begin{itemize}
\item[{\bf P1}] The functions $\sigma_{ij}$ and $b_i$
are bounded and infinitely differentiable with bounded
partial derivatives of all orders, for $1 \leq i,j \leq d$.
\item[{\bf P2}] The matrix $\sigma$ is uniformly elliptic, that is,
$\Vert \sigma(x) \xi \Vert^2 \geq \rho^2 >0$ for some $\rho>0$,
for all $x \in \mathbb{R}^d$, $\xi \in \mathbb{R}^d$, $\Vert \xi \Vert=1$ ($\Vert \cdot \Vert$ denotes the Euclidean norm on $\IR^d$).
\end{itemize}

\begin{thm} \label{t1}
Assume {\bf P1} and {\bf P2}. Fix $T>0$ and let $I \subset (0,T]$ and $J \subset (0,1)$ be two compact nonrandom intervals.
\begin{itemize}
\item[\textnormal{(a)}] The density $p_{t,x}(z)$ is uniformly bounded
over $z \in \IR^d$, $t\in I$ and $x \in J$.

\item[\textnormal{(b)}] There exists $c>0$ such that for any $t\in I$, $x \in J$ and $z  \in \mathbb{R}^d$,
\begin{equation*}
p_{t,x}(z) \geq c t^{-d/4} \exp\biggl(-\frac{\Vert z \Vert^2}{c t^{1/2}} \biggr).
\end{equation*}
\item[\textnormal{(c)}] For all $\eta>0$,
there exists $c>0$ such that for any $s,t \in I$, $x,y \in J$, $(s,y) \neq (t,x)$ and $z_1,z_2 \in \mathbb{R}^d$,
\begin{equation}\label{psytx}
p_{s,y;\,t,x}(z_1,z_2) \leq c (|t-s|^{1/2}+|x-y|)^{-(d+\eta)/2} \exp
\biggl( -\frac{\Vert z_1-z_2 \Vert^2}{c
(|t-s|^{1/2}+|x-y|)}\biggr).
\end{equation}
\item[\textnormal{(d)}] There exists $c>0$ such that for any $t\in I$, $x,y \in J$, $x \neq y$ and $z_1,z_2 \in \mathbb{R}^d$,
\begin{equation*}
p_{t,y;\,t,x}(z_1,z_2) \leq c (|x-y|)^{-d/2} \exp
\biggl( -\frac{\Vert z_1-z_2 \Vert^2}{c |x-y|}\biggr).
\end{equation*}
\end{itemize}
\end{thm}


   The main technical effort in this paper is to obtain the upper bound in (c). Indeed, it is not difficult to check that for fixed $(s,y;\,t,x)$, $(z_1, z_2) \mapsto p_{s,y;\,t,x} (z_1, z_2)$ behaves like a Gaussian density function. However, for $(s, y) = (t,x)$, the $\IR^{2d}$-valued random vector $(u(s \,,y), u(t\,,x))$  is concentrated on a $d$-dimensional subspace in $\IR^{2d}$ and therefore does not have a density with respect to Lebesgue measure in $\IR^{2d}$. So the main effort is to estimate how this density blows up as $(s,y) \to (t,x)$. This is achieved by a detailed analysis of the behavior of the Malliavin matrix of $(u(s\,,y), u(t\,,x))$ as a function of $(s,y; t, x)$, using a perturbation argument. The presence of $\eta$ in statement (c) may be due to the method of proof. When $t=s$, it is possible to set $\eta=0$ as in Theorem \ref{t1}(d).

   This paper is organized as follows. After introducing some notation and stating our main results on hitting probabilities (Theorems \ref{t2} and \ref{t4}), we assume Theorem \ref{t1} and use the theorems of \cite{Dalang:05} to prove these results in Section \ref{sec2}. In Section \ref{sec3}, we recall some basic facts of Malliavin calculus and state and prove two results that are tailored to our needs (Propositions \ref{normHc} and \ref{deter}). In Section \ref{sec4}, we establish the existence, smoothness and uniform boundedness of the one-point density function $p_{t,x}$, proving Theorem \ref{t1}(a). In Section \ref{sec5}, we establish a lower bound on $p_{t,x}$, which proves Theorem \ref{t1}(b). This upper (respectively lower) bound is a fairly direct extension to $d \geq 1$ of a result of \textsc{Bally and Pardoux} \cite{Bally:98} (respectively \textsc{Kohatsu-Higa} \cite{Kohatsu:03}) when $d=1$. In Section \ref{sec6}, we establish Theorem \ref{t1}(c) and (d). The main steps are as follows.

   The upper bound on the two-point density function $p_{s,y;\,t,x}$ involves a bloc-decomposition of the Malliavin matrix of the $\R^{2d}$-valued random vector $(u(s,y),\, u(s,y) - u(t,x))$. The entries of this matrix are of different orders of magnitude, depending on which bloc they are in: see Theorem \ref{ga}. Assuming Theorem \ref{ga}, we prove Theorem \ref{t1}(c) and (d) in Section \ref{sec63}. The exponential factor in \eqref{psytx} is obtained from an exponential martingale inequality, while the factor $(|t-s|^{1/2}+|x-y|)^{-(d+\eta)/2}$ comes from an estimate of the iterated Skorohod integrals that appear in Corollary \ref{no} and from the block structure of the Malliavin matrix. 

   The proof of Theorem \ref{ga} is presented in Section \ref{p63}: this is the main technical effort in this paper. We need bounds on the inverse of the Malliavin matrix. Bounds on its cofactors are given in Proposition \ref{3p2}, while bounds on negative moments of its determinant are given in Proposition \ref{3p3}. The determinant is equal to the product of the $2d$ eigenvalues of the Malliavin matrix. It turns out that at least $d$ of these eigenvalues are of order $1$ (``large eigenvalues") and do not contribute to the upper bound in \eqref{psytx}, and at most $d$ are of the same order as the smallest eigenvalue (``small eigenvalues"), that is, of order $|t-s|^{1/2}+|x-y|$. If we did not distinguish between these two types of eigenvalues, but estimated all of them by the smallest eigenvalue, we would obtain a factor of $(|t-s|^{1/2}+|x-y|)^{-d+\eta/2}$ in \eqref{psytx}, which would {\em not} be the correct order. The estimates on the smallest eigenvalue are obtained by refining a technique that appears in \cite{Bally:98}; indeed, we obtain a precise estimate on the density whereas they only showed existence. The study of the large eigenvalues does not seem to appear elsewhere in the literature.


Coming back to potential theory, let us introduce some notation. For all Borel sets $F\subset \R^d$,
we define $\mathcal{P}(F)$ to be the set of all probability measures
with compact support contained in $F$.
For all integers $k\ge 1$ and
$\mu\in \mathcal{P} (\R^k)$, we let $I_\beta(\mu)$ denote the
\emph{$\beta$-dimensional energy} of $\mu$, that is,
\begin{equation*}
    I_\beta(\mu)  := \iint {\rm K}_\beta(\|x-y\|)\, \mu(dx)\,\mu(dy),
\end{equation*}
where $\|x\|$ denotes the Euclidian norm of $x \in \R^k$,
\begin{equation} \label{k}
	{\rm K}_\beta(r) := 
	\begin{cases}
		r^{-\beta}&\text{if $\beta >0$},\\
		\log  ( N_0/r ) &\text{if $\beta =0$},\\
		1&\text{if $\beta<0$},
    \end{cases}
\end{equation}
and $N_0$ is a sufficiently large constant (see \textsc{Dalang, Khoshnevisan, and Nualart} \cite[(1.5)]{Dalang:05}.

For all $\beta\in\R$, integers $k\ge 1$, and Borel sets $F\subset\R^k$, $\text{Cap}_\beta(F)$ denotes the \emph{$\beta$-dimensional capacity of
$F$}, that is,
\begin{equation*}
    \text{Cap}_\beta(F) := \left[ \inf_{\mu\in\mathcal{P}(F)}
    I_\beta(\mu) \right]^{-1},
\end{equation*}
where $1/\infty:=0$. Note that if $\beta <0$, then $\text{Cap}_\beta(\cdot) \equiv 1$.

Given $\beta\geq 0$, the $\beta $-dimensional \emph{Hausdorff measure}
of $F$ is defined by
\begin{equation}\label{eq:HausdorfMeasure}
	{\mathcal{H}}_ \beta (F)= \lim_{\epsilon \rightarrow 0^+} \inf
	\left\{ \sum_{i=1}^{\infty} (2r_i)^ \beta : F \subseteq
	\bigcup_{i=1}^{\infty} B(x_i\,, r_i), \ \sup_{i\ge 1} r_i \leq
	\epsilon \right\},
\end{equation}
where $B(x\,,r)$ denotes the open (Euclidean) ball of radius
$r>0$ centered at $x\in \R^d$. When $\beta <0$, we define 
$\mathcal{H}_\beta (F)$ to be infinite.

Throughout, we consider the following \emph{parabolic metric}:
For all $s,t\in[0\,,T]$ and $x,y\in[0\,,1]$,
\begin{equation}\label{eq:Delta}
    {\bf \Delta}((t\,,x)\,; (s\,,y)) :=
    |t-s|^{1/2} + |x-y|.
\end{equation}
Clearly, this is a metric on $\R^2$ which
generates the usual Euclidean topology on $\R^2$.
Then we obtain an energy form
\begin{equation*}
    I_\beta^{{\bf \Delta}} (\mu) := \iint {\rm K}_\beta(
    {\bf \Delta}((t\,,x)\,; (s\,,y)))
   \,\mu(dt\,dx)\,\mu(ds\,dy),
\end{equation*}
and a corresponding capacity
\begin{equation*}
    \text{Cap}^{\bf \Delta}_\beta (F) := \left[
    \inf_{\mu\in\mathcal{P}(F)} I_\beta^{ \bf \Delta}(\mu)
    \right]^{-1}.
\end{equation*}
For the Hausdorff measure, we write
\begin{equation*}
	{\mathcal{H}}^{\bf \Delta}_{\beta} (F)= \lim_{\epsilon \rightarrow 0^+} \inf
	\biggl\{ \sum_{i=1}^{\infty} (2r_i)^{\beta} : F \subseteq
	\bigcup_{i=1}^{\infty} B^{\bf \Delta}((t_i\,,x_i)\,, r_i), \ 
	\sup_{i\ge 1} r_i \leq
	\epsilon \biggr\},
\end{equation*}
where $B^{\bf\Delta}((t\,,x)\,,r)$ denotes the open
$\bf\Delta$-ball of radius $r>0$ centered at $(t\,,x)\in[0\,,T]\times[0\,,1].$
\vskip 12pt

Using Theorem \ref{t1} together with results from \textsc{Dalang, Khoshnevisan, and Nualart} \cite{Dalang:05}, we shall prove the following result. Let $u(E)$ denote the (random) range of $E$ under the map $(t,x) \mapsto u(t,x)$, where $E$ is some Borel-measurable subset of $\R^2$.

\begin{thm} \label{t2}
Assume {\bf P1} and {\bf P2}. Fix $T>0$, $M>0$, and $\eta>0$. Let $I \subset (0,T]$ and $J \subset (0,1)$ be two fixed non-trivial compact intervals.
\begin{itemize}
\item[\textnormal{(a)}] There exists $c>0$ depending on $M, I, J$ and $\eta$ such that for all compact sets $A \subseteq [-M, M]^d$,
\begin{equation*}
c^{-1} \,  \textnormal{Cap}_{d-6+\eta}(A) \leq \P \{ u(I \times J) \cap A \neq
\emptyset \}  \leq c \, \mathcal{H}_{d-6-\eta}(A).
\end{equation*}
\item[\textnormal{(b)}]  For all $t\in (0,T]$, 
there exists $c_1>0$ depending on $T$, $M$ and $J$, and $c_2>0$ depending on $T$, $M$, $J$ and $\eta>0$ such that for all compact sets $A \subseteq [-M, M]^d$,
\begin{equation*}
c_1 \, \textnormal{Cap}_{d-2}(A) \leq \P \{ u(\{t\} \times J) \cap A \neq \emptyset \} \leq c_2 \, \mathcal{H}_{d-2-\eta}(A).
\end{equation*}
\item[\textnormal{(c)}]  For all $x\in (0,1)$, there exists $c>0$ depending on $M, I$ and $\eta$ such that for all compact sets $A \subseteq [-M, M]^d$,
\begin{equation*}
c^{-1} \, \textnormal{Cap}_{d-4+\eta}(A) \leq \P \{ u(I \times\{x\}) \cap A \neq \emptyset \} \leq c \,  \mathcal{H}_{d-4-\eta}(A).
\end{equation*}
\end{itemize}
\end{thm}

\begin{rema} 

\begin{enumerate}

\item[\textnormal{(i)}] Because of the inequalities between capacity and Hausdorff measure, the right-hand sides of Theorem \ref{t2} can be replaced by $c \, \textnormal{Cap}_{d-6-\eta}(A)$, $c \, \textnormal{Cap}_{d-2-\eta}(A)$
and $c \, \textnormal{Cap}_{d-4-\eta}(A)$ in \textnormal{(a)}, \textnormal{(b)} and \textnormal{(c)}, respectively (cf. \textsc{Kahane} \cite[p. 133]{Kahane:85}).

\item[\textnormal{(ii)}] Theorem~\ref{t2} also holds if we consider Dirichlet boundary conditions
(i.e. $u_i(t,0)=u_i(t,1)=0$, for $t \in [0,T]$) instead of Neumann boundary conditions.

\item[\textnormal{(iii)}] In the upper bounds of Theorem~\ref{t2}, the condition in {\bf P1}  that $\sigma$ and $b$
are bounded can be removed, but their derivatives of all orders must exist and be bounded.
\end{enumerate}
\end{rema}

As a consequence of Theorem~\ref{t2}, we deduce the following result on the polarity of points. Recall that a Borel set $A\subseteq\R^d$
is called \emph{polar} for $u$ if $\P\{u((0,T] \times (0,1))\cap A\neq\varnothing\}=0$;
otherwise, $A$ is called \emph{nonpolar}.

\begin{cor} \label{c3}
Assume {\bf P1} and {\bf P2}.
\begin{itemize}
\item[\textnormal{(a)}] Singletons are nonpolar for $(t,x) \mapsto u(t,x)$ when $d \leq 5$, and are polar when $d \geq 7$ (the case $d=6$ is open).
\item[\textnormal{(b)}] Fix $t \in (0,T]$. Singletons are nonpolar for $x \mapsto u(t,x)$ when $d=1$, and are polar when $d \geq 3$ (the case $d=2$ is open).
\item[\textnormal{(c)}]  Fix $x \in (0,1)$. Singletons are not polar for $t \mapsto u(t,x)$ when $d \leq 3$ and are polar when $d \geq 5$ (the case $d=4$ is open).
\end{itemize}
\end{cor}

Another consequence of Theorem~\ref{t2} is the Hausdorff dimension of the range
of the process $u$.
\begin{cor} \label{c3bis}
Assume {\bf P1} and {\bf P2}.
\begin{itemize}
\item[\textnormal{(a)}] If $d>6$, then 
$\dimh   (u((0,T] \times (0,1))) = 6$ a.s.
\item[\textnormal{(b)}] Fix $t \in \mathbb{R}_+$. If $d>2$, then 
$\dimh   (u(\{t\} \times (0,1))) = 2$ a.s.
\item[\textnormal{(c)}] Fix $x \in (0,1)$. If $d>4$, then 
$\dimh  (u(\mathbb{R}_+ \times \{x\})) = 4$ a.s.
\end{itemize}
\end{cor}

As in \textsc{Dalang, Khoshnevisan, and Nualart} \cite{Dalang:05}, it is also possible to use Theorem~\ref{t1} to obtain results concerning level sets of $u$. 
Define
\begin{equation*}\begin{split}
    \mathcal{L}(z\,;u) &:= \left\{ (t\,,x)\in I \times
        J:\ u(t\,,x) = z\right\},\\
    \mathcal{T}(z\,;u) &= \left\{ t\in I:\ u(t\,,x)=z
        \text{ for some }x\in J \right\},\\
    \mathcal{X}(z\,;u) &= \left\{ x\in J:\ u(t\,,x)=z
        \text{ for some }t\in I\right\},\\
    \mathcal{L}_{x}(z\,;u) &:= \left\{ t\in I:\
        u(t\,,x)=z\right\},\\
    \mathcal{L}^{t}(z\,;u) &:= \left\{ x\in J:\
        u(t\,,x)=z\right\}.
\end{split}\end{equation*}
We note that $\mathcal{L}(z\,;u)$ is the level set of $u$ at level $z$, $\mathcal{T}(z\,;u)$ (resp.~$\mathcal{X}(z\,;u)$) 
is the projection of $\mathcal{L}(z\,;u)$ onto $I$ (resp.~$J$), and $\mathcal{L}_{x}(z\,;u)$ (resp.~$\mathcal{L}^{t}(z\,;u)$)
 is the $x$-section (resp.~$t$-section) of $\mathcal{L}(z\,;u)$.

\begin{thm} \label{t4}
  Assume {\bf P1} and {\bf P2}. Then for all $\eta>0$ and $R>0$
    there exists a positive and finite constant $c$ such that
    the following holds for all compact sets
    $E\subset (0,T] \times (0,1)$,
    $F\subset (0,T]$,
    $G\subset (0,1)$, and for all $z\in B(0\,,R)$:
    \begin{enumerate}
        \item[\textnormal{(a)}]
            $c^{-1} \,  \textnormal{Cap}_{(d+\eta)/2}^{{\bf \Delta}}
			(E) \le \P\{ \mathcal{L}(z\,;u) \cap E
			\neq\varnothing \}\le c \,\mathcal{H}^{{\bf \Delta}}_{(d-\eta)/2}(E)
			$;
        \item[\textnormal{(b)}]
            $c^{-1} \textnormal{Cap}_{(d-2+\eta)/4}(F) \le \P\{ \mathcal{T}(z\,;u)\cap F\neq\varnothing\}\le c \,\mathcal{H}_{(d-2-\eta)/4}(F)
			$;
        \item[\textnormal{(c)}]
            $ c^{-1} \, \textnormal{Cap}_{(d-4+\eta)/2} (G)\le \P\{\mathcal{X}(z\,;u)\cap G\neq\varnothing\}\le c \,\mathcal{H}_{(d-4-\eta)/2}(G)
			$;
        \item[\textnormal{(d)}]  for all $x \in (0,1)$,
            $c^{-1} \,\textnormal{Cap}_{(d+\eta)/4}(F)\le \P\{\mathcal{L}_x(z\,;u) \cap F\neq\varnothing\}\le c \,\mathcal{H}_{(d-\eta)/4}(F)$;
        \item[\textnormal{(e)}] for all $t\in (0,T]$,
	        $ c^{-1} \,\textnormal{Cap}_{d/2}(G) \le \P\{ \mathcal{L}^t(z\,;u) \cap G\neq\varnothing\}
			\le c \,\mathcal{H}_{(d-\eta)/2}(G)$.
    \end{enumerate}
\end{thm}

\begin{cor} \label{c5}
Assume {\bf P1} and {\bf P2}.
Choose and fix $z\in\R^d$.
	\begin{itemize}
		\item[\textnormal{(a)}] If $2 < d< 6$, then
			$\dimh\, \cT(z \, ; u) = \frac14 (6-d)$
			a.s.\ on $\{\cT(z \, ; u)\neq\varnothing\}$.
		\item[\textnormal{(b)}] If $4 < d< 6$ (i.e. $d=5$), then
			$\dimh\, \cX(z \, ; u) = \frac12 (6-d)$
			a.s.\ on $\{\cX(z \, ; u)\neq\varnothing\}$.
		\item[\textnormal{(c)}] If $1\leq d< 4$, then
			$\dimh\, \cL_x(z \, ; u) = \frac14 (4-d)$
			a.s.\ on $\{\cL_x(z \, ; u)\neq\varnothing\}$.
		\item[\textnormal{(d)}] If $d=1$, then
			$\dimh\, \cL^t(z\, ; u) = \frac12(2-d)=\frac{1}{2}$
			a.s.\ on $\{\cL^t(z \, ;u)\neq\varnothing\}$.
	\end{itemize}
	In addition, all four right-most events have positive probability.
\end{cor}

\begin{rema} The results of the two theorems and corollaries above should be compared with those of \textsc{Dalang, Khsohnevisan and Nualart} \cite{Dalang:05}.
\end{rema}

\section{Proof of Theorems \ref{t2}, \ref{t4} and their corollaries (assuming Theorem \ref{t1})}\label{sec2}

We first recall that equation (\ref{e1}) is formal: a rigorous formulation, following \textsc{Walsh} \cite{Walsh:86}, is as follows.
Let $W^i=(W^i(s,x))_{s \in \mathbb{R}_+, \, x \in [0,1]}$, $i=1,...,d$, be independent Brownian sheets defined on a probability
space $(\Omega, \mathcal{F}, \P)$, and
set $W=(W^1,...,W^d)$. For $t \geq 0$, let $\mathcal{F}_t=\sigma\{W(s,x),\ s \in [0,t],\ x \in [0,1]\}$.
We say that a process $u=\{u(t,x), \, t \in [0,T], \, x
\in [0,1]\}$ is adapted to $(\mathcal{F}_t)$ if $u(t,x)$ is
${\mathcal{F}}_t$-measurable for each $(t,x) \in [0,T] \times [0,1]$.
We say that $u$ is a solution of (\ref{e1}) if $u$ is adapted to $(\mathcal{F}_t)$ and if for $i \in \{1,\dots,d \}$,
\begin{equation}\begin{split} \label{e2} 
u_i(t,x)&= \int_0^t \int_0^1 G_{t-r}(x \,,v) \, \sum_{j=1}^d \s_{i,j}(u(r \,,v)) W^j(dr\,,dv) \\
& \qquad \qquad + \int_0^t \int_0^1 G_{t-r}(x\,,v) \, b_i(u(r\,,v)) \,dr dv,
\end{split}
\end{equation}
where $G_t(x \,,y)$ denotes the Green kernel for the heat
equation with Neumann boundary conditions (see \textsc{Walsh} \cite[Chap 3]{Walsh:86}), and the stochastic integral in (\ref{e2})
is interpreted as in \cite{Walsh:86}.

Adapting the results from~\cite{Walsh:86} to the case $d\geq 1$, one
can show that there exists a unique continuous process $u=\{u(t,x), \, t \in [0,T], \, x
\in [0,1]\}$ adapted to $(\mathcal{F}_t)$ that is a solution of \textnormal{(\ref{e1})}.
Moreover, it is shown in \textsc{Bally, Millet, and Sanz-Sol\'e} \cite{Bally:95} that for any $s,t \in [0,T]$ with $s \leq t$, $x,y \in [0,1]$, and $p>1$,
\begin{equation} \label{holder}
\E[|u(t,x)-u(s,y)|^p] \leq C_{T,p} ({\bf \Delta}((t\,,x)\,; (s\,,y)))^{p/2},
\end{equation}
where ${\bf \Delta}$ is the parabolic metric defined in (\ref{eq:Delta}).
In particular, for any $0<\alpha < 1/2$, $u$ is a.s.~$\alpha$-H\"older continuous in $x$ and
$\alpha/2$-H\"older continuous in $t$.

Assuming Theorem \ref{t1}, we now prove Theorems \ref{t2}, \ref{t4}  and their corollaries.

\begin{proof}[Proof of Theorem \ref{t2}]
(a) In order to prove the upper bound we use \textsc{Dalang, Khoshnevisan, and Nualart}
\cite[Theorem 3.3]{Dalang:05}. Indeed, Theorem \ref{t1}(a) and (\ref{holder}) imply that the hypotheses (i)
and (ii), respectively, of this theorem, are satisfied, and so the conclusion (with $\beta=d-\eta$) is too.

In order to prove the lower bound, we shall use of \cite[Theorem 2.1]{Dalang:05}. This requires checking
hypotheses {\bf A1} and {\bf A2} in that paper. Hypothesis {\bf A1} is a lower bound on the one-point density
function $p_{t,x}(z)$, which is an immediate consquence of Theorem \ref{t1}(b). Hypothesis {\bf A2}
is an upper bound on the two-point density function $p_{s,y;t,x}(z_1, z_2)$, which involves a parameter
$\beta$; we take $\beta=d +\eta$. In this case, Hypothesis {\bf A2} is an immediate consequence of Theorem \ref{t1}(c).
Therefore, the lower bound in Theorem \ref{t2}(a) follows from \cite[Theorem 2.1]{Dalang:05}. This proves (a).

(b) For the upper bound, we again refer to \cite[Theorem 3.3]{Dalang:05} (see also \cite[Theorem 3.1]{Dalang:05}).
For the lower bound, which involves $\text{Cap}_{d-2}(A)$ instead of $\text{Cap}_{d-2+\eta}(A)$, we refer
to \cite[Remark 2.5]{Dalang:05} and observe that hypotheses ${\bf A1}^t$ and ${\bf A2}^t$ there are satisfied with $\beta=d$
(by Theorem \ref{t1}(d)). This proves (b).

(c) As in (a), the upper bound follows from \cite[Theorem 3.3]{Dalang:05} with $\beta=d-\eta$ (see also \cite[Theorem 3.1(3)]{Dalang:05}), and the lower bound follows from \cite[Theorem 2.1(3)]{Dalang:05}, with $\beta=d+\eta$.
Theorem \ref{t2} is proved.
\end{proof}

\begin{proof}[Proof of Corollary \ref{c3}]
We first prove (a). Let $z \in \mathbb{R}^d$. If $d \leq 5$, then there is $\eta>0$ such that $d-6+\eta < 0$, and thus
$\textnormal{Cap}_{d-6+\eta}(\{ z\})=1$. Hence, the lower bound of Theorem \ref{t2} (a) implies that
$\{ z\}$ is not polar. On the other hand, if $d>6$, then for small $\eta>0$, $d-6-\eta > 0$. Therefore,
$\mathcal{H}_{d-6-\eta}(\{ z\})=0$ and the upper bound of Theorem \ref{t2}(a) implies that
$\{ z\}$ is polar. This proves (a). One proves (b) and (c) exactly along the same lines using Theorem \ref{t2}(b) and (c).
\end{proof}

\begin{proof}[Proof of Theorem \ref{t4}]
For the upper bounds in (a)-(e), we use \textsc{Dalang, Khoshnevisan, and Nualart}
\cite[Theorem 3.3]{Dalang:05} whose assumptions we verified above with $\beta=d-\eta$; these upper bounds then
follow immediately from \cite[Theorem 3.2]{Dalang:05}.

For the lower bounds in (a)-(d), we use \cite[Theorem 2.4]{Dalang:05} since we have shown above that
the assumptions of this theorem, with $\beta=d+\eta$, are satisfied by Theorem \ref{t1}. For the lower bound in (e),
we refer to \cite[Remark 2.5]{Dalang:05} and note that by Theorem \ref{t1}(d), Hypothesis ${\bf A2}^t$ there
is satisfied with $\beta=d$.
This proves Theorem \ref{t4}.
\end{proof}

\begin{proof}[Proof of Corollaries  \ref{c3bis} and \ref{c5}]

The final positive-probability assertion in Corollary \ref{c5} is an
	  immediate consequence of Theorem \ref{t4} and
	Taylor's theorem \textsc{Khoshnevisan} \cite[Corollary 2.3.1 p.~523]{Khoshnevisan:02}. 
	
Let $E$ be a random set. When it exists,
	the codimension of $E$ is the real number 
	$\beta \in [0\,,d]$ such that for all compact sets $A \subset
	\mathbb{R}^d$,
	\begin{equation*}
		\P \{ E \cap A \neq \varnothing \}
		\begin{cases}
			>0 & \text{whenever $\dimh (A)> \beta$}, \\
			=0 & \text{whenever $\dimh (A)< \beta$}.
		\end{cases}
	\end{equation*}
	See \textsc{Khoshnevisan}
	\cite[Chap.11, Section 4]{Khoshnevisan:02}. When it is
	well defined, we write
	the said codimension as $\text{codim}(E)$.
Theorems \ref{t2} and \ref{t4} imply that for $d \geq 1$:
$\textnormal{codim}(u(\mathbb{R}_+ \times (0,1)))=$ $(d-6)^+$;
$\textnormal{codim}(u(\{t\} \times (0,1)))$ $=(d-2)^+$;
$\textnormal{codim}(u(\mathbb{R}_+ \times \{x\}))$ $=(d-4)^+$;
$\textnormal{codim}(\cT(z))=$ $(\frac{d-2}{4})^+$;
$\textnormal{codim}(\cX(z))=$ $(\frac{d-4}{2})^+$;
$\textnormal{codim}(\cL_x(z))=$ $\frac{d}{4}$; and
$\textnormal{codim}(\cL^t(z))=$ $\frac{d}{2}$. According to Theorem 4.7.1 of
	\textsc{Khoshnevisan}
	\textnormal{\cite[Chapter 11]{Khoshnevisan:02}},
	given a random set $E$ in $\mathbb{R}^n$ whose
	codimension is strictly between $0$ and $n$,
	\begin{equation} \label{codim}
		\dimh E+\textnormal{codim } E=n
		\qquad\text{a.s.~on}\ \{E\neq\varnothing\}.
	\end{equation}
This implies the statements of Corollaries \ref{c3bis} and \ref{c5}.
\end{proof}

\section{Elements of Malliavin calculus}\label{sec3}

In this section, we introduce, following \textsc{Nualart} \cite{Nualart:95} (see also \textsc{Sanz-Sol\'e} \cite{Sanz:05}), some
elements of Malliavin calculus.
Let $\mathcal{S}$ denote the class of smooth random variables of the form
\begin{equation*}
F=f(W(h_1),...,W(h_n)),
\end{equation*}
where $n \geq 1$, $f \in \mathcal{C}^{\infty}_P(\mathbb{R}^n)$, the set of real-valued
functions $f$ such that $f$ and all its partial derivatives have
at most polynomial growth, $h_i \in \mathcal{H}:=L^2([0,T]
\times [0,1], \mathbb{R}^d)$, and $W(h_i)$ denotes the Wiener integral
\begin{equation*}
W(h_i)= \int_0^T \int_0^1 h_i(t,x) \cdot W(dx, dt), \; \; 1\leq i \leq
n.
\end{equation*}
Given $F \in \mathcal{S}$, its derivative is defined to be the $\mathbb{R}^d$-valued
stochastic process $DF=(D_{t,x}F= (D^{(1)}_{t,x}F,...,D^{(d)}_{t,x}F) , \, (t,x) \in [0,T] \times
[0,1])$ given by
\begin{equation*}
D_{t,x} F= \sum_{i=1}^n \frac{\partial f}{\partial x_i}
(W(h_1),...,W(h_n)) h_i (t,x).
\end{equation*}
More generally, we can define the derivative $D^k F$ of order $k$ of $F$ by setting
\begin{equation*}
D^{k}_{\alpha}F=\sum_{i_1,...,i_k=1}^n \frac{\partial}{\partial{x_{i_1}}}  \cdots
\frac{\partial}{\partial{x_{i_k}}} f(W(h_1),...,W(h_n)) h_{i_1}(\alpha_1) \otimes \cdots \otimes 
h_{i_k}(\alpha_k),
\end{equation*}
where $\alpha=(\alpha_1,...,\alpha_k)$, and $\alpha_i=(t_i,x_i)$,
$1 \leq i \leq k$.
\vskip 12pt
For $p, k\geq 1$, the space $\mathbb{D}^{k,p}$ is the closure of $\mathcal{S}$ with
respect to the seminorm $\Vert \cdot \Vert ^p _{k,p}$ defined by
\begin{equation*}
\Vert F \Vert ^p _{k,p} = \E[|F|^p] + \sum_{j=1}^k
 \E[\Vert D^j F \Vert ^p _{\mathcal{H}^{\otimes j}}],
\end{equation*}
where
\begin{equation*}
\Vert D^j F \Vert ^2 _{\mathcal{H}^{\otimes j}}=\sum_{i_1,...,i_j=1}^d \int_0^T dt_1 \int_0^1 dx_1
\cdots \int_0^T dt_j \int_0^1 dx_j \, \biggl(D_{(t_1,x_1)}^{(i_1)} \cdots D_{(t_j,x_j)}^{(i_j)} F \biggr)^2. 
\end{equation*}
We set $(\mathbb{D}^{\infty})^d= \cap_{p \geq 1}\cap_{k \geq 1}\mathbb{D}^{k,p}$.
\vskip 12pt
The derivative operator $D$ on $L^2(\Omega)$ has an
adjoint, termed the Skorohod integral and denoted by $\delta$, which
is an unbounded operator on $L^2(\Omega, \mathcal{H})$. Its domain, denoted by Dom~$\delta$, is the
set of elements $u \in L^2(\Omega, \mathcal{H})$ such that there exists a constant $c$ such that $|
E [\langle DF, u \rangle_{\mathcal{H}}] | \leq c \Vert F \Vert _{0,2}$, for any $F \in
\mathbb{D}^{1,2}$.
\vskip 12pt
If $u \in$ Dom~$\delta$, then $\delta (u)$ is the element of
$L^2 (\Omega)$ characterized by the following duality
relation:
\begin{equation*}
\E [F \delta (u)]= \E \biggl[\sum_{j=1}^d
\int_0^T \int_0^1 D^{(j)}_{t,x} F \; u_j(t,x) \, dt dx \biggr], \;
\; \text{for all} \; F \in  \mathbb{D}^{1,2}.
\end{equation*}

A first application of Malliavin calculus to the study of
probability laws is the following global criterion for smoothness
of densities.
\begin{thm} \label{3t1}\textnormal{~\cite[Thm.2.1.2 and Cor.2.1.2]{Nualart:95} or~\cite[Thm.5.2]{Sanz:05}}
Let $F=(F^1,...,F^d)$ be an $\mathbb{R}^d$-valued random vector satisfying the following
two conditions: 
\begin{enumerate}
\item[\textnormal{(i)}] $F \in (\mathbb{D}^{\infty})^d$; 
\item[\textnormal{(ii)}] the Malliavin matrix of $F$ defined by
$\gamma_F=(\langle DF^i, D F^j \rangle_{\mathcal{H}})_{1\leq i,j\leq d}$ is
invertible a.s. and $(\textnormal{det} \; \gamma_F)^{-1} \in L^p(\Omega)$ for all $p\geq1$.
\end{enumerate}
Then the probability law of $F$ has an
infinitely differentiable density function.
\end{thm}

A random vector $F$ that satisfies conditions (i) and (ii) of Theorem~\ref{3t1} is said to be
nondegenerate. For a nondegenerate random vector, the following
integration by parts formula plays a key role.
\begin{prop} \label{n} \textnormal{~\cite[Prop.3.2.1]{Nualart:98} or~\cite[Prop.5.4]{Sanz:05}}
Let $F=(F^1,...,F^d) \in (\mathbb{D}^{\infty})^d$ be a nondegenerate
random vector, let $G \in \mathbb{D}^{\infty}$ and let $g \in
\mathcal{C}^{\infty}_P(\mathbb{R}^d)$. Fix $k\geq1$. Then for any
multi-index $\alpha=(\alpha_1,...,\alpha_k) \in \{1,\dots,d
\}^k$, there is an element $H_{\alpha}(F,G) \in \mathbb{D}^{\infty}$
such that
\begin{equation*}
\E [(\partial_{\alpha}g) (F) G]= \E [g(F)
H_{\alpha}(F,G)].
\end{equation*}
In fact, the random variables $H_{\alpha}(F,G)$ are recursively given by
\begin{equation*}\begin{split}
H_{\alpha}(F,G)& = H_{(\alpha_k)}(F,H_{(\alpha_1,\dots,\alpha_{k-1})}(F,G)), \\ 
H_{(i)}(F,G)& =\sum_{j=1}^d\delta(G \, (\gamma_{F}^{-1})_{i, j} \,
DF^j).
\end{split}\end{equation*}
\end{prop}

Proposition~\ref{n} with $G=1$ and $\alpha=(1,...,d)$ implies
the following expression for the density of a
nondegenerate random vector.
\begin{cor}\label{no} \textnormal{~\cite[Corollary 3.2.1]{Nualart:98}}
Let $F=(F^1,...,F^d) \in (\mathbb{D}^{\infty})^d$ be a nondegenerate
random vector and let $p_F(z)$ denote the density of $F$. Then for every subset
$\sigma$ of the set of indices $\{1,...,d\}$,
\begin{equation*}
p_F(z)=(-1)^{d-|\sigma|} \E [1_{\{F^i>z^i, i \in \sigma, \,
F^i < z^i, i \not\in \sigma \}} H_{(1,...,d)}(F,1)],
\end{equation*}
where $|\sigma|$ is the cardinality of $\sigma$, and, in
agreement with Proposition~\ref{n},
\begin{equation*}
H_{(1,...,d)}(F,1)=\delta((\gamma_{F}^{-1} DF)^d
\delta((\gamma_{F}^{-1} DF)^{d-1} \delta( \cdots
\delta((\gamma_{F}^{-1} DF)^{1})\cdots))).
\end{equation*}
\end{cor}

The next result gives a criterion for uniform boundedness of the density of a
nondegenerate random vector.
\begin{prop} \label{normHc}
For all $p>1$ and $\ell \geq1$, let $c_1=c_1(p)>0$ and $c_2=c_2(\ell,p)\geq 0$ be fixed.
Let $F \in (\mathbb{D}^{\infty})^d$ be a nondegenerate random vector such that
\begin{itemize}
\item[\textnormal{(a)}] $\E [(\textnormal{det}
\,\gamma_{F} )^{-p}] \leq c_1$;
\item[\textnormal{(b)}] $\E [\Vert D^{l}(F^i)
\Vert^p_{\mathcal{H}^{\otimes \ell}}] \leq c_2 ,\; i=1,...,d$.
\end{itemize}
Then the density of $F$ is uniformly bounded, and the bound does not depend on $F$ but only on the constants
$c_1(p)$ and $c_2(\ell,p)$.
\end{prop}

\begin{proof}
The proof of this result uses the same arguments as in the proof
of \textsc{Dalang and Nualart} \cite[Lemma 4.11]{Dalang:04}. Therefore, we will only give the
main steps.

Fix $z \in \mathbb{R}^d$. Thanks to Corollary~\ref{no} and the
Cauchy-Schwarz inequality we find that
\begin{equation*}
\vert p_F(z) \vert \leq \Vert H_{(1,...,d)}(F,1) \Vert_{0,2}.
\end{equation*}
Using the continuity of the Skorohod integral $\delta$ 
(cf. \textsc{Nualart} \cite[Proposition 3.2.1]{Nualart:95} and \textsc{Nualart} \cite[(1.11) and p.131]{Nualart:98}) and H\"older's inequality for Malliavin norms (cf. \textsc{Watanabe} \cite[Proposition 1.10,
p.50]{Watanabe:84}), we obtain
\begin{equation} \label{iteration}
\Vert H_{(1,...,d)}(F,1) \Vert_{0,2}  \leq c \Vert
H_{(1,...,d-1)}(F,1) \Vert_{1,4}  \sum_{j=1}^d
\Vert (\gamma_{F}^{-1})_{d, j} \Vert_{1,8} \, \Vert D(F^j)
\Vert_{1,8}.
\end{equation}
In agreement with hypothesis (b), $\Vert D(F^j)
\Vert_{m,p} \leq c$. In order to bound the second factor in (\ref{iteration}), note that
\begin{equation} \label{defigamma}
\Vert (\gamma_{F}^{-1})_{i,j}
\Vert_{m,p}=\biggl\{ \E [|(\gamma_{F}^{-1})_{i, j}|^p ] +
\sum_{k=1}^m \E [\Vert D^{k} (\gamma_{F}^{-1})_{i, j}
\Vert^p _{\mathcal{H}^{\otimes k}} ] \biggr\}^{1/p}.
\end{equation}
For the first term in (\ref{defigamma}), we use Cramer's formula to get that
\begin{equation*}
|(\gamma_{F}^{-1})_{i, j}|=|(\text{det} \,
\gamma_F)^{-1} (A_F)_{i, j}|,
\end{equation*}
where $A_F$ denotes the cofactor matrix of $\gamma_{F}$. By means of
Cauchy-Schwarz inequality and hypotheses (a) and (b)
we find that
\begin{equation*}\begin{split}
\E [ ((\gamma_{F}^{-1})_{i, j})^p ] &\leq c_{d,p}
\{ \E [ (\text{det} \, \gamma_{F})^{-2p}]\}^{1/2} \times
\{\E [ \Vert
D(F)\Vert_{\mathcal{H}}^{4p(d-1)}]\}^{1/2}\\ 
 \leq c_{d,p},
\end{split}\end{equation*}
where none of the constants depend on $F$.
For the second term on the right-hand side of (\ref{defigamma}), we iterate the equality (cf. \textsc{Nualart} \cite[Lemma
2.1.6]{Nualart:95})
\begin{equation} \label{eulalia}
D ( \gamma_{F}^{-1} )_{i, j}=- \sum_{k,\ell=1}^d ( \gamma_{F}^{-1}
)_{i, k} D ( \gamma_{F})_{k, \ell} ( \gamma_{F}^{-1} )_{\ell, j},
\end{equation}
in the same way as in the proof of \textsc{Dalang and Nualart} \cite[Lemma 4.11]{Dalang:04}.
Then, appealing again to hypotheses (a) and (b) and iterating
the inequality (\ref{iteration}) to bound the first factor on the right-hand side of (\ref{defigamma}), we obtain the uniform boundedness
of $p_F(z)$.
\end{proof}
\vskip 12pt
We finish this section with a result that will be used later on to bound negative moments of a random variable,
as is needed to check hypothesis (a) of Proposition \ref{normHc}.
\begin{prop} \label{deter}
Suppose $Z\ge 0$ is a random variable for
    which we can find $\e_0\in(0,1)$, processes
    $\{Y_{i,\e}\}_{\e\in(0,1)}$ ($i=1,2$), and  constants
    $c>0$ and $0\le \alpha_2\le \alpha_1$
    with the property that
    $Z \ge \min(
    c \e^{\alpha_1} - Y_{1,\e} \,
    c \e^{\alpha_2} - Y_{2,\e})$  for all $\e\in(0,\e_0)$.
    Also suppose that we can find $\beta_i>\alpha_i$
    ($i=1,2$), not depending on $\e_0$, such that
    \begin{equation*}
        C(q) := \sup_{0<\e<1}\max\left(
        \frac{\E [ |Y_{1,\e} |^q ]}{\e^{q\beta_1}},      \frac{ \E [ |Y_{2,\e} |^q ]}{\e^{q\beta_2}}\right) <\infty
        \quad\text{for all }q\ge 1.
    \end{equation*}
    Then for all $p\ge 1$,
    there exists a constant $c'\in(0,\infty)$, not depending on
    $\e_0$, such that
    $\E [ |Z|^{-p}] \le c' \e_0^{-p\alpha_1}$.
\end{prop}

\begin{rema}
    This lemma is of interest mainly when $\beta_2 \leq \alpha_1$.
\end{rema}

\begin{proof}
    Define $k := (2/c)\e_0^{-\alpha_1}$.
    Suppose that $y\ge k$,  and let $\e:=(2/c)^{1/\alpha_1}
    y^{-1/\alpha_1}$. Then $0<\e\le\e_0$,
    $y^{-1}=(c/2)\e^{\alpha_1}$, and for all $q \geq 1$,
    \begin{equation*}\begin{split}
        \P \left\{ Z^{-1}>y \right\} &= \P \left\{ Z<y^{-1}\right\}\\
        &\le \P \left\{ Y_{1,\e}\ge \frac{c}{2}\e^{\alpha_1}\right\}
            + \P \left\{ Y_{2,\e} \ge c\e^{\alpha_2}
            - \frac{c}{2}\e^{\alpha_1}\right\}\\
        &\le \frac{C(q)}{c^q}\left( 2^q \e^{q(\beta_1-\alpha_1)}
            + \e^{q\beta_2} \left[ \e^{\alpha_2}
            - \frac12 \e^{\alpha_1} \right]^{-q} \right).
    \end{split}\end{equation*}
    The inequality
    $\e^{\alpha_2}-(1/2)\e^{\alpha_1}\ge(1/2)\e^{\alpha_2}$
    implies that
    \begin{equation*}
        \P \left\{ Z^{-1}>y \right\}
        \le \frac{C(q)}{c^q}\left( 2^q \e^{q(\beta_1-\alpha_1)}
        + 2^q \e^{q(\beta_2-\alpha_2)} \right)
        \le a y^{-qb},
    \end{equation*}
    where $a$ and $b$ are positive and finite constants that
    do not depend on $y, \e_0$ or $q$. We apply this
    with $q:=(p/b)+1$ to find that for all $p\ge 1$,
    \begin{equation*}\begin{split}
        \E \left[ |Z|^{-p}\right] &= p\int_0^\infty y^{p-1}
            \P \left\{ Z^{-1}>y \right\}\, dy\\
        &\le k^p +ap\int_{k}^\infty y^{-b-1}\, dy
            = k^p + \left(\frac{ap}{b}\right) k^{-b}.
    \end{split}\end{equation*}
    Because $k\ge (2/c)$ and $b>0$, it follows that
    $\E [ |Z|^{-p}] \le  ( 1+ c_1(ap/b) ) k^p$,
    where $c_1:=(c/2)^{b+p}$. This is the desired result.
\end{proof}

\section{Existence, smoothness and uniform boundedness of the one-point density}\label{sec4}

Let $u=\{ u(t,x), \, t \in [0,T], x \in [0,1] \}$ be the solution of
equation (\ref{e2}). In this section, we prove the existence, smoothness and uniform boundedness
of the density of the random vector $u(t,x)$. In particular, this will prove Theorem \ref{t1}(a).
\vskip 12pt
The first result concerns the Malliavin differentiability of $u$ and the
equations satisfied by its derivatives. We refer
to \textsc{Bally and Pardoux} \cite[Proposition 4.3,\, (4.16),\, (4.17)] {Bally:98} for its proof in
dimension one. As we work coordinate by coordinate, the following proposition
follows in the same way and its proof is therefore omitted.
\begin{prop} \label{p4}
Assume {\bf P1}. Then $u(t,x) \in (\mathbb{D}^{\infty})^d$ for any $t \in [0,T]$ and $x \in
[0,1]$. Moreover, its iterated
derivative satisfies
\begin{equation*} \begin{split}
& D^{(k_1)}_{r_1,v_1}  \cdots D^{(k_n)}_{r_n,v_n} (u_i(t,x)) \\ 
& =\sum_{l=1}^d G_{t-r_l}(x,v_l) \biggl( D_{r_1,v_1}^{(k_1)} \cdots D_{r_{l-1},v_{l-1}}^{(k_{l-1})}  D_{r_{l+1},v_{l+1}}^{(k_{l+1})}
\cdots D_{r_n,v_n}^{(k_n)} (\sigma_{i k_l}(u(r_l,v_l)))\biggr) \\ 
& \qquad +\sum_{j=1}^d \int_{r_1 \vee \cdots \vee
r_n}^t \int_0^1 G_{t-\theta} (x, \eta) \prod_{l=1}^n
D^{(k_l)}_{r_l,v_l}(\sigma_{ij}(u(\theta, \eta))) W^j(d\theta, d
\eta) \\
&\qquad +\int_{r_1 \vee \cdots \vee r_n}^t \int_0^1
G_{t-\theta}(x,\eta) \prod_{l=1}^n
D^{(k_l)}_{r_l,v_l}(b_i(u(\theta, \eta))) \, d\theta d\eta
\end{split}\end{equation*}
if $t \leq r_1 \vee \cdots \vee r_n$ and $D^{(k_1)}_{r_1,v_1}
\cdots D^{(k_n)}_{r_n,v_n}(u_i(t,x))=0$ otherwise. Finally, for any $p>1$,
\begin{equation}  \label{sup}
\textnormal{sup}_{(t,x) \in [0,T] \times [0,1]} \E \biggl[
\big\Vert D^{n}(u_i(t,x)) \big\Vert_{\mathcal{H}^{\otimes n}}^p \biggl] < +\infty.
\end{equation}
\end{prop}

Note that, in particular, the first-order Malliavin derivative satisfies, for
$r<t$,
\begin{equation}  \label{deri}
D_{r,v}^{(k)} (u_i(t,x))= G_{t-r}(x,v) \sigma_{ik} (u(r,v))
+a_i(k,r,v,t,x),
\end{equation}
where
\begin{equation}\begin{split} \label{a} 
&a_i(k,r,v,t,x)=\sum_{j=1}^d \int_r^t \int_0^1 G_{t-\theta}(x,
\eta) D^{(k)}_{r,v}(\sigma_{ij}(u(\theta, \eta))) \, W^j(d\theta,
d\eta) \\
& \qquad \qquad \qquad \qquad \qquad + \int_r^t \int_0^1
G_{t-\theta}(x, \eta) D^{(k)}_{r,v}(b_i(u(\theta, \eta))) \,
d\theta d \eta,
\end{split}\end{equation}
and $D_{r,v}^{(k)}(u_i(t,x))=0$ when $r>t$. \vskip 12pt The next
result proves property (a) in
Proposition~\ref{normHc} when $F$ is replaced by $u(t,x)$.
\begin{prop} \label{p5}
Assume {\bf P1} and {\bf P2}. Let $I$ and $J$ two compact intervals as in Theorem \ref{t1}. Then, for any $p \geq 1$,
\begin{equation*}
\E \big[(\textnormal{det}\, \gamma_{u(t,x)})^{-p}\big]
\end{equation*}
is uniformly bounded over $(t,x) \in I \times J$.
\end{prop}

\begin{proof}
This proof follows \textsc{Nualart} \cite[Proof of (3.22)]{Nualart:98}, where it is shown that for fixed $(t,x)$,
$\E [(\textnormal{det} \gamma_{u(t,x)})^{-p}] < +\infty$. Our emphasis here is on the uniform bound over $(t,x) \in I \times J$.
Assume that $I=[t_1, t_2]$ and $J=[x_1, x_2]$, where $0<t_1<t_2\leq T$, $0<x_1<x_2<1$. Let $(t,x) \in I \times J$ be fixed. We write \begin{equation*}\textnormal{det}\, \gamma_{u(t,x)} \geq \biggl( \textnormal{inf}_{\xi \in
\mathbb{R}^d:  \Vert \xi \Vert =1} \, \xi^{T} \gamma_{u(t,x)}
\xi \biggr)^d.\end{equation*} Let $\xi=(\xi_1,...,\xi_d) \in \mathbb{R}^d$ with $\Vert \xi \Vert=1$ and
fix $\epsilon \in (0,1)$. Note the inequality
    \begin{equation}\label{eq:23-2}
        (a+b)^2\ge \frac23 a^2-2b^2,
    \end{equation}
    valid for all $a,b\ge 0$. Using (\ref{deri}) and the fact that $\gamma_{u(t,x)}$ is a matrix
    whose entries are inner-products, this implies that
\begin{equation*}\begin{split}
\xi^{T} \gamma_{u(t,x)} \xi &=\int_0^t dr
\int_0^1 dv \, \bigg\Vert \sum_{i=1}^d D_{r,v}(u_i(t,x)) \xi_i \bigg\Vert^2 \\
&\geq \int_{t(1-\epsilon)}^t dr
\int_0^1 dv \, \bigg\Vert \sum_{i=1}^d D_{r,v}(u_i(t,x)) \xi_i \bigg\Vert^2 \geq  I_1 -I_2,
\end{split}\end{equation*}
where
\begin{equation*}\begin{split}
I_1=& \frac{2}{3}\int_{t(1-\epsilon)}^t dr \int_0^1 dv \, \sum_{k=1}^d
\left(\sum_{i=1}^d G_{t-r}(x,v) \sigma_{ik}(u(r,v)) \xi_i\right)^2, \\
I_2=& 2 \int_{t(1-\epsilon)}^t dr \int_0^1 dv \, \sum_{k=1}^d
\left(\sum_{i=1}^d a_i(k,r,v,t,x)\xi_i \right)^2,
\end{split}\end{equation*}
and $a_i(k,r,v,t,x)$ is defined in (\ref{a}).
In accord with hypothesis {\bf P2} and thanks to Lemma \ref{(A.3)},
\begin{equation} \label{equa2}
I_1 \geq c (t \epsilon)^{1/2},
\end{equation}
where $c$ is uniform over $(t,x) \in I \times J$.

Next we apply the Cauchy-Schwarz inequality to find that, for any $q\geq 1$,
\begin{equation*}
\E \biggl[ \textnormal{sup}_{\xi \in \mathbb{R}^d : \Vert \xi
\Vert=1} |I_2|^q
\biggr] \leq c (\E [\vert A_1 \vert^q]+\E [\vert A_2 \vert^q]),
\end{equation*}
where
\begin{equation*}\begin{split}
A_1&= \sum_{i,j,k=1}^d
\int_{t(1-\epsilon)}^t dr \int_0^1 dv \, \left(\int_r^t \int_0^1
G_{t-\theta}(x, \eta)
D^{(k)}_{r,v}(\sigma_{ij}(u(\theta, \eta))) \, W^j(d\theta, d\eta)
\right)^2  ,
\\
A_2 &= \sum_{i,k=1}^d  \int_{t(1-\epsilon)}^t
dr \int_0^1 dv \, \left( \int_r^t \int_0^1 G_{t-\theta}(x, \eta)
D^{(k)}_{r,v}(b_i(u(\theta, \eta))) \, d\theta d \eta \right)^2.
\end{split}\end{equation*}

We bound the $q$-th moment of $A_1$ and $A_2$ separately.
As regards $A_1$, we use Burkholder's inequality for martingales with values in a Hilbert space
(Lemma \ref{valuedm}) to obtain
\begin{equation} \label{pre}
\E [\vert A_1 \vert^q] \leq c \sum_{k,i=1}^d  \E \biggl[ \bigg\vert\int_{t(1-\epsilon)}^{t} d \theta
\int_0^1 d \eta  \int_{t(1-\epsilon)}^t dr \int_0^1 dv  \,
 \Theta^2 \bigg\vert^q\biggr],
\end{equation}
where
\begin{equation*}\begin{split}
\Theta &:= 1_{ \{ \theta>r \} }
G_{t-\theta}(x, \eta) \bigg\vert D^{(k)}_{r,v}(\sigma_{ij}(u(\theta, \eta)))  \bigg\vert \\
&\leq c 1_{ \{ \theta>r \} } G_{t-\theta}(x, \eta) \bigg\vert \sum_{l=1}^d D^{(k)}_{r,v}(u_l(\theta, \eta)) \bigg\vert,
\end{split}\end{equation*}
thanks to hypothesis {\bf P1}. Hence,
\begin{equation*}
\E [\vert A_1 \vert^q] \leq c \sum_{k=1}^d  \E \biggl[ \bigg\vert\int_{t(1-\epsilon)}^{t} d \theta  \int_0^1 d \eta\,
G^2_{t-\theta}(x, \eta)  \int_{t(1-\epsilon)}^{t \wedge \theta} dr \int_0^1 dv  \,
 \Psi^2 \bigg\vert^q\biggr],
\end{equation*}
where $\Psi:=\sum_{l=1}^d D^{(k)}_{r,v}(u_l(\theta, \eta))$.
We now apply H\"older's inequality with respect to the measure
$G^2_{t-\theta}(x,\eta) d \theta d \eta$ to find that
\begin{equation*}\begin{split}
&\E [\vert A_1 \vert^q] \leq C \bigg\vert \int_{t(1-\epsilon)}^t d \theta \int_0^1 d \eta \,
G^2_{t-\theta}(x,\eta) \bigg\vert^{q-1} \\
&\qquad\times\int_{t(1-\epsilon)}^t d \theta \int_0^1 d \eta \,
G^2_{t-\theta}(x,\eta) \, \sum_{k=1}^d \E\biggl[ \bigg\vert
\int_{t(1-\epsilon)}^t dr \int_0^1 dv \, \Psi^2 \bigg\vert^q \biggr].
\end{split}\end{equation*}
Lemmas \ref{(A.5)} and \ref{morien} assure that
\begin{equation*}
\E [\vert A_1 \vert^q] \leq C_T (t\epsilon)^{\frac{q-1}{2}} (t\epsilon)^{q/2}
\int_{t(1-\epsilon)}^t \int_0^1 G^2_{t-\theta}(x,\eta) \, d
\theta d \eta \leq C_T (t \epsilon)^q,
\end{equation*}
where $C_T$ is uniform over $(t,x) \in I \times J$.

We next derive a similar bound for $A_2$. By the Cauchy--Schwarz inequality,
\begin{equation*}
        \E \left[ |A_2|^q\right]
        \le c(t\e)^q
             \sum_{i,k=1}^d
            \E \left[ \left| \int_{t(1-\e)}^t
            dr \int_0^1 dv \int_r^t d\theta
            \int_0^1 d\eta\, \Phi^2 \right|^q \right],
    \end{equation*}
    where $\Phi:=G_{t-\theta}(x,\eta) |D^{(k)}_{r,v}\left(b_i(u(\theta,\eta))\right)|$.
    From here on,
    the $q$-th moment of $A_2$ is estimated as that
    of $A_1$ was; cf.\ \eqref{pre}, and this yields
    $\E [|A_2|^q]\le C_T(t\e)^{2q}$.

    Thus, we have proved that
\begin{equation} \label{way}
\E \biggl[ \textnormal{sup}_{\xi \in \mathbb{R}^d : \Vert \xi
\Vert=1} |I_2|^q \biggr] \leq C_T (t \epsilon)^q,
\end{equation}
where the constant $C_T$ is clearly uniform over $(t,x) \in I \times J$.

Finally, we apply Proposition~\ref{deter} with $Z:=\textnormal{inf}_{\Vert \xi \Vert =1} (\xi^{T} \gamma_{u(t,x)}
\xi)$, $Y_{1,\epsilon}=$ $Y_{2,\epsilon}=$ $\textnormal{sup}_{\Vert \xi
\Vert=1} I_2$, $\epsilon_0=1$, $\alpha_1=\alpha_2=1/2$ and $\beta_1=\beta_2=1$,
to get
\begin{equation*}
\E  \biggl[(\textnormal{det} \gamma_{u(t,x)})^{-p}
\biggr] \leq C_T ,
\end{equation*}
where all the constants are clearly uniform over $(t,x) \in I \times J$. This is the desired result.
\end{proof}

\begin{cor}\label{cor4.3}
Assume {\bf P1} and {\bf P2}. Fix $T>0$ and let $I$ and $J$ be a compact intervals as in Theorem \ref{t1}. Then, for any $(t,x) \in (0,T] \times (0,1)$, $u(t,x)$ is a nondegenerate
random vector and its density funciton is infinitely differentiable and uniformly bounded over $z \in \mathbb{R}^d$
and $(t,x)\in I \times J$.
\end{cor}

\begin{proof}[Proof of Theorem \ref{t1}(a)]
This is a consequence of Propositions~\ref{p4} and~\ref{p5} together with
Theorem~\ref{3t1} and Proposition~\ref{normHc}.
\end{proof}

\section{The Gaussian-type lower bound on the one-point density}\label{sec5}

The aim of this section is to prove the lower bound of Gaussian-type for the
density of $u$ stated in Theorem \ref{t1}(b). The proof of this result was given
in \textsc{Kohatsu-Higa} \cite[Theorem 10]{Kohatsu:03} for dimension $1$, therefore we will only
sketch the main steps.

\vskip 12pt

\begin{proof}[Proof of Theorem \ref{t1}(b)]
We follow~\cite{Kohatsu:03} and we show that for each $(t,x)$, $F=u(t,x)$ is a $d$-dimensional uniformly elliptic
random vector and then we apply~\cite[Theorem 5]{Kohatsu:03}. Let
\begin{equation*}
F_n^i=\int_0^{t_n} \int_0^1 G_{t-r}(x,v) \sum_{j=1}^d \sigma_{ij}(u(r,v))
 W^j(dr,dv)+ \int_0^{t_n} \int_0^1 G_{t-r}(x,v) \, b_i(u(r,v))
\,dr dv,
\end{equation*}
$1 \leq i \leq d$, where $0=t_0 < t_1 < \cdots < t_N=t$ is a sufficiently fine partition
of $[0,t]$. Note that $F_n \in
\mathcal{F}_{t_n}$. Set $g(s,y)=G_{t-s}(x,y)$. We shall need the following two lemmas.
\begin{lem}  \textnormal{~\cite[Lemma 7]{Kohatsu:03}}
Assume {\bf P1} and {\bf P2}. Then:
\begin{enumerate}
\item[\textnormal{(i)}] $\Vert F_n^i \Vert_{k,p} \leq c_{k,p}, \; 1 \leq i \leq d$;

\item[\textnormal{(ii)}] $\Vert ((  \gamma_{F_n}(t_{n-1}))_{i j})^{-1} \Vert_{p, t_{n-1}}  \leq
c_p  (\Delta_{n-1}(g))^{-1}= c_p (\Vert g  \Vert^2_{L^2([t_{n-1},t_n] \times [0,1])})^{-1}$,
\end{enumerate}
where $\gamma_{F_n}(t_{n-1})$ denotes the conditional Malliavin matrix of $F_n$ given $\mathcal{F}_{t_{n-1}}$ and $\Vert\cdot\Vert_{p, t_{n-1}}$ denotes
the conditional $L^p$-norm.
\end{lem}
We define
\begin{equation*}\begin{split}
u^{n-1}_i(s_1,y_1)= & \int_0^{t_{n-1}} \int_0^1
G_{s_1-s_2}(y_1,y_2) \sum_{j=1}^d \sigma_{ij}(u(s_2,y_2)) W^j(ds_1,dy_2)
\\  &+ \int_0^{t_{n-1}} \int_0^1 G_{s_1-s_2}(y_1,y_2) \,
b_i(u(s_2,y_2)) \,ds_2 dy_2, \quad  1 \leq i \leq d.
\end{split}\end{equation*}
Note that  $u^{n-1} \in \mathcal{F}_{t_{n-1}}$. As in~\cite{Kohatsu:03}, the following holds.
\begin{lem}   \label{larturo} \textnormal{~\cite[Lemma 8]{Kohatsu:03}}
Under hypothesis {\bf P1}, for $s \in [t_{n-1},t_n]$,
\begin{equation*}
\Vert u_i(s,y)-u_i^{n-1}(s,y)  \Vert_{n, p, t_{n-1}} \leq (s-t_{n-1})^{1/8}, \qquad  1 \leq i \leq d,
\end{equation*}
where  $\Vert \cdot\Vert_{n,p, t_{n-1}}$ denotes
the conditional Malliavin norm given $\mathcal{F}_{t_{n-1}}$.
\end{lem}

The rest of the proof of Theorem~\ref{t1}(b) follows along the same
lines as in~\cite{Kohatsu:03} for $d=1$. We only sketch the
remaining main points where the fact that $d>1$ is important.
In order to obtain the expansion of $F^i_n-F^i_{n-1}$
as in~\cite[Lemma 9]{Kohatsu:03}, we proceed as follows. By the mean
value theorem,
\begin{align*} \nonumber
&F_n^i-F_{n-1}^i=  \int_{t_{n-1}}^{t_n} \int_0^1 G_{t-r}(x,v) \,
\sum_{j=1}^d \sigma_{ij}(u^{n-1}(r,v)) W^j(dr, dv) \\ 
&\qquad + \int_{t_{n-1}}^{t_n} \int_0^1 G_{t-r}(x,v) \,
b_i(u(r,v)) \,dr dv
\\ \nonumber & \qquad + \int_{t_{n-1}}^{t_n} \int_0^1
G_{t-r}(x,v) \sum_{j,l=1}^d (\int_0^1 \partial_l \sigma_{ij}(u(r,v,
\lambda)) d \lambda) (u_{l}(r,v)-u_l^{n-1}(r,v)) W^j(dr,dv),
\end{align*}
where $u(r,v, \lambda)= (1- \lambda) u(r,v) + \lambda
u^{n-1}(r,v)$. Using the terminology of~\cite{Kohatsu:03}, the first term is a process of order
$1$ and the next two terms are residues of order $1$ (as
in~\cite{Kohatsu:03}). In the next step, we write the residues of
order $1$ as the sum of processes of order $2$ and residues of order $2$ and
$3$ as follows:
\begin{align*} \nonumber
&\int_{t_{n-1}}^{t_n} \int_0^1 G_{t-r}(x,v) \, b_i(u(r,v)) \,dr dv\\ 
& \qquad = \int_{t_{n-1}}^{t_n} \int_0^1 G_{t-r}(x,v) \, b_i(u^{n-1}(r,v)) \,dr dv  \\ \nonumber
&  \qquad\qquad + \int_{t_{n-1}}^{t_n} \int_0^1 G_{t-r}(x,v) \sum_{l=1}^d
(\int_0^1 \partial_l b_i(u(r,v, \lambda)) d \lambda) (u_l(r,v)-u_l^{n-1}(r,v))  \,
dr dv
\end{align*}
and
\begin{equation*}\begin{split}
& \int_{t_{n-1}}^{t_n} \int_0^1 G_{t-r}(x,v) \sum_{j,l=1}^d
(\int_0^1 \partial_l \sigma_{ij}(u(r,v, \lambda)) d \lambda) (u_l(r,v)-u_l^{n-1}(r,v))
 W^j(dr,dv) \\ 
&=  \int_{t_{n-1}}^{t_n} \int_0^1 G_{t-r}(x,v) \sum_{j,l=1}^d  \partial_l \sigma_{ij}(u^{n-1}(r,v))  (u_l(r,v)-u_l^{n-1}(r,v))
W^j(dr,dv) \\ 
& \qquad +\int_{t_{n-1}}^{t_n} \int_0^1 G_{t-r}(x,v)
\sum_{j,l,l^{\prime}=1}^d (\int_0^1 \partial_l \partial_{l^{\prime}} \sigma_{ij}(u(r,v, \lambda)) d \lambda) \\
& \qquad \qquad \times  (u_l(r,v)-u_l^{n-1}(r,v))
(u_{l^{\prime}}(r,v)-u_{l^{\prime}}^{n-1}(r,v))   W^j(dr,dv).
\end{split}\end{equation*}
It is then clear that the remainder of the proof of~\cite[Lemma 9]{Kohatsu:03} follows for $d>1$ along the same lines as in~\cite{Kohatsu:03}, working coordinate
by coordinate.

Finally, in order to complete the proof of the proposition, it suffices to verify the hypotheses of~\cite[Theorem 5]{Kohatsu:03}. Again the proof follows as in
the proof of~\cite[Theorem 10]{Kohatsu:03},
working coordinate by coordinate. We will only sketch the proof of his (H2c), where hypothesis {\bf P2} is used:
\begin{equation*}\begin{split}
& (\Delta_{n-1}(g))^{-1} \int_{t_{n-1}}^{t_n}  \int_0^1 (G_{t-r}(x,v))^2 \Vert \sigma(u^{n-1}(r,v)) \xi \Vert^2 \, dr dv \\
& \qquad \geq \rho^2  (\Delta_{n-1}(g))^{-1}   \int_{t_{n-1}}^{t_n}  \int_0^1 (G_{t-r}(x,v))^2 \, dr dv =\rho^2 >0,
\end{split}\end{equation*}
by the definition of $g$.
This concludes the proof of Theorem \ref{t1} (b).
\end{proof}

\section{The Gaussian-type upper bound on the two-point density}\label{sec6}

Let $p_{s,y; \, t,x}(z_1,z_2)$ denote the joint
density of the $2d$-dimensional random vector
\begin{equation*}(u_1(s,y),...,u_d(s,y), u_1(t,x),...,u_d(t,x)),\end{equation*}
for $s,t \in (0,T]$, $x,y \in (0,1)$, $(s,y) \neq (t,x)$ and
$z_1,z_2 \in \mathbb{R}^d$ (the existence of this joint density will be a consequence of Theorem~\ref{3t1},
Proposition~\ref{p4} and Theorem~\ref{ga}).
\vskip 12pt
The next subsections lead to the proofs of Theorem \ref{t1}(c) and (d).

\subsection{Bounds on the increments of the Malliavin derivatives}

In this subsection, we prove an upper bound for the Sobolev norm of the derivative of the increments of our process $u$.
For this, we will need the following preliminary estimate.
\begin{lem} \label{l1}
For any $s,t \in [0,T]$, $s \leq t$, and $x,y \in [0,1]$,
\begin{equation*}
\int_0^T \int_0^1 (g(r,v))^2 \, dr dv \leq C_T
(|t-s|^{1/2}+|x-y|),
\end{equation*}
where \begin{equation*}g(r,v):=g_{t,x,s,y}(r,v)=1_{\{r \leq t\}} G_{t-r}(x,v)-1_{\{r \leq s\}}
G_{s-r}(y,v).\end{equation*}
\end{lem}

\begin{proof}
Using \textsc{Bally, Millet, and Sanz-Sol\'e} \cite[Lemma B.1]{Bally:95} with $\alpha=2$, we see that
\begin{equation*} \begin{split}
&\int_0^T \int_0^1 (g(r,v))^2 \, dr dv \\ 
&\leq \int_s^t \int_0^1
(G_{t-r}(x,v))^2 \, dr dv + 2 \int_0^s \int_0^1
(G_{t-r}(x,v)-G_{s-r}(x,v))^2\, dr dv \\ 
& \qquad+ 2 \int_0^s \int_0^1 (G_{s-r}(x,v)-G_{s-r}(y,v))^2\, dr
dv\\
&\leq C_T (|t-s|^{1/2}+|x-y|).
\end{split}\end{equation*}
\end{proof}

\begin{prop} \label{3p1}
Assuming {\bf P1}, for any $s,t \in [0,T]$, $s \leq t$,
$x,y \in [0,1]$, $p>1$, $m \geq 1$,
\begin{equation*}
\E \biggl[\big\Vert D^m(u_i(t,x)-u_i(s,y)) \big\Vert^p_{\mathcal{H}^{\otimes m}}
\biggr] \leq C_T  \bigl(|t-s|^{1/2}+|x-y|\bigr)^{p/2}, \;  \; i=1,...,d.
\end{equation*}
\end{prop}

\begin{proof}
Let $m=1$. Consider the function $g(r,v)$ defined in
Lemma~\ref{l1}. Using the integral equation (\ref{deri})
satisfied by the first-order Malliavin derivative,
we find that
\begin{equation*}
\E \biggl[\big\Vert D(u_i(t,x)-u_i(s,y)) \big\Vert^p_{\mathcal{H}} \biggr] \leq C
\bigl(\E [\vert I_1\vert^{p/2}]+\E[\vert I_2\vert^{p/2}]+\E[\vert I_3\vert^{p/2}]\bigr),
\end{equation*}
where
\begin{align*} 
I_1 &= \sum_{k=1}^d \int_0^T dr \int_0^1 dv \, (g(r,v) \sigma_{ik}(u(r,v)))^2,   \\ 
I_2 &= \sum_{j,k=1}^d \int_0^T dr \int_0^1 dv \left(
 \int_0^T \int_0^1 g(\theta,\eta)
D^{(k)}_{r,v}(\sigma_{ij}(u(\theta, \eta))) W^j(d\theta, d\eta) \right)^2, \\
 I_3 &= \sum_{k=1}^d \int_0^T dr \int_0^1 dv \, \left(
\int_0^T\int_0^1 g(\theta,\eta) D^{(k)}_{r,v}(b_i(u(\theta,\eta))) d\theta d\eta \right)^2.
\end{align*}
We bound the $p/2$-moments of $I_1, I_2$ and $I_3$ separately.

By hypothesis {\bf P1} and Lemma~\ref{l1}, $E[\vert I_1\vert^{p/2}] \leq C_T  (|t-s|^{1/2}+|x-y|)^{p/2}$.
Using Burkholder's inequality for Hilbert-space-valued
martingales (Lemma \ref{valuedm}) and hypothesis {\bf P1}, we obtain
\begin{equation*}
\E[\vert I_2\vert^{p/2}] \leq C  \sum_{k=1}^d \E\biggl[ \bigg\vert \int_0^T d \theta \int_0^1 d \eta \,
(g(\theta,\eta))^2 \int_0^T dr \int_0^1 dv \, \Theta^2
\bigg\vert^{p/2} \biggr],
\end{equation*}
where $\Theta:=\sum_{l=1}^d
D^{(k)}_{r,v}(u_l(\theta, \eta))$. From H\"older's inequality with respect to the measure
$(g(\theta, \eta))^2 d \theta d \eta$, we see that this is bounded above by
\begin{align*} \nonumber
 & C  \biggl( \int_0^T \int_0^1 (g(\theta, \eta))^2 d \theta
d \eta \biggr)^{\frac{p}{2}-1} \\ 
&\qquad \times \textnormal{sup}_{(\theta, \eta) \in
[0,T] \times[0,1]} \sum_{k=1}^d \E \biggr[ \bigg\vert \int_0^T d \theta \int_0^1 d \eta (g(\theta, \eta))^2 \int_0^T dr \int_0^1 dv \, \Theta^2
\bigg\vert^{p/2} \biggr] \\ \nonumber
&\leq  C_T  (|t-s|^{1/2}+|x-y|)^{p/2},
\end{align*}
thanks to (\ref{sup}) and
Lemma~\ref{l1}.

We next derive a similar bound for $I_3$. By the Cauchy--Schwarz inequality,
\begin{equation*}
\E[\vert I_3\vert^{p/2}]\leq C_T  \sum_{k=1}^d \E\biggl[ \bigg\vert \int_0^T d \theta \int_0^1 d \eta \,
(g(\theta,\eta))^2 \int_0^T dr \int_0^1 dv \, \Theta^2
\bigg\vert^{p/2} \biggr].
\end{equation*}
From here on, the $p/2$-moment of $I_3$ is estimated as was that of $I_2$, and this yields
 $\E[\vert I_3\vert^{p/2}]\leq  C_T  (|t-s|^{1/2}+|x-y|)^{p/2}$. This proves the desired result for $m=1$.
The case $m>1$ follows using the stochastic differential equation satisfied by the iterated Malliavin derivatives
(Proposition~\ref{p4}), H\"older's and Burkholder's inequalities, hypothesis {\bf P1}, (\ref{sup}) and
Lemma~\ref{l1} in the same way as we did for $m=1$, to obtain the desired bound.
\end{proof}

\subsection{Study of the Malliavin matrix}

For $s,t \in [0,T]$, $s \leq t$, and $x,y \in [0,1]$
consider the $2d$-dimensional random vector
\begin{equation}  \label{vector}
Z:=(u_1(s,y),...,u_d(s,y), u_1(t,x)-u_1(s,y),...,u_d(t,x)-u_d(s,y)).
\end{equation}
Let $\gamma_Z$ the
Malliavin matrix of $Z$. Note that $\gamma_Z=((\gamma_Z)_{m,l})_{m,l=1,...,2d}$ is a symmetric $2d \times 2d$
random matrix with four $d \times d$ blocs of the form
\begin{equation*}
\gamma_Z =\left(
\begin{array}{ccc}
\gamma_Z^{(1)} & \vdots & \gamma_Z^{(2)}\\
\cdots & \vdots & \cdots \\
\gamma_Z^{(3)}& \vdots & \gamma_Z^{(4)}\\
\end{array}
\right),
\end{equation*}
where
\begin{equation*} \begin{split}
\gamma_Z^{(1)} &= (\langle D(u_i(s,y)), D(u_j(s,y)) \rangle_{\mathcal{H}})_{i,j=1,...,d},\\
\gamma_Z^{(2)} &= (\langle D(u_i(s,y)), D(u_j(t,x)-u_j(s,y)) \rangle_{\mathcal{H}})_{i,j=1,...,d},\\ 
\gamma_Z^{(3)} &= (\langle D(u_i(t,x)-u_i(s,y)), D(u_j(s,y))  \rangle_{\mathcal{H}})_{i,j=1,...,d},\\
\gamma_Z^{(4)} &= (\langle D(u_i(t,x)-u_i(s,y)),  D(u_j(t,x)-u_j(s,y))  \rangle_{\mathcal{H}})_{i,j=1,...,d}.
\end{split}\end{equation*}
We let {\bf (1)} denote the set of indices $\{1,...,d\} \times \{1,...,d\}$, {\bf (2)} the set $\{1,...,d\} \times \{d+1,...,2d\}$,
{\bf (3)} the set $\{d+1,...,2d\} \times \{1,...,d\}$ and {\bf (4)} the set $\{d+1,...,2d\} \times \{d+1,...,2d\}$.

\vskip 12pt
The following theorem gives an estimate on the Sobolev norm of the entries of the inverse of the matrix $\gamma_Z$, which depends on the position of the entry in the matrix.
\begin{thm} \label{ga}
Fix $\eta, T>0$. Assume {\bf P1} and {\bf P2}. Let $I$ and $J$ be two compact intervals as in Theorem \ref{t1}.
\begin{enumerate}
\item[\textnormal{(a)}] For any $(s,y) \in I \times J$, $(t,x) \in I \times J$, $s \leq t$, $(s,y) \neq (t,x)$, $k \geq 0$, $p>1$,
\begin{equation*}
\Vert (\gamma_Z^{-1})_{m,l} \Vert_{k,p}\leq
\begin{cases}
c_{k,p,\eta,T}(|t-s|^{1/2}+|x-y|)^{-d \eta}&
\text{if \; $(m,l) \in {\bf (1)}$}, \\
c_{k,p,\eta,T} (|t-s|^{1/2}+|x-y|)^{-1/2-d \eta} &
\text{if \; $(m,l) \in {\bf (2)}$ or ${\bf (3)}$},  \\
c_{k,p,\eta,T} (|t-s|^{1/2}+|x-y|)^{-1-d \eta}& \text{if \; $(m,l) \in {\bf (4)}$}.
\end{cases}
\end{equation*}
\item[\textnormal{(b)}] For any $s=t \in (0,T]$, $(t,y) \in I \times J$, $(t,x) \in I \times J$ , $x \neq y$, $k \geq 0$, $p>1$,
\begin{equation*}
\Vert (\gamma_Z^{-1})_{m,l} \Vert_{k,p}\leq
\begin{cases}
c_{k,p,T}&
\text{if \; $(m,l) \in {\bf (1)}$ }, \\
c_{k,p,T} |x-y|^{-1/2} &
\text{if \; $(m,l) \in {\bf (2)}$ or ${\bf (3)}$},  \\
c_{k,p,T} |x-y|^{-1}& \text{if \; $(m,l) \in {\bf (4)}$}.
\end{cases}
\end{equation*}
\end{enumerate}
(Note the slight improvements in the exponents in case \textnormal{(b)} where $s=t$.)
\end{thm}

The proof of this theorem is deferred to Section \ref{p63}. We assume it for the moment and complete the proof of Theorem \ref{t1}(c) and (d).

\subsection{Proof of Theorem \ref{t1}(c) and (d)}\label{sec63}

Fix two compact intervals $I$ and $J$ as in Theorem \ref{t1}. Let $(s,y), (t,x) \in I \times J$, $s \leq t$, $(s,y) \neq (t,x)$, and
$z_1,z_2 \in \mathbb{R}^d$. Let $Z$ be as in (\ref{vector}) and let $p_Z$ be the density of $Z$. Then
\begin{equation*}
p_{s,y; \, t,x}(z_1,z_2)=p_Z(z_1, z_1-z_2).
\end{equation*}
Apply Corollary~\ref{no} with $\sigma=\{ i \in \{1,...,d\} : z_1^i
-z_2^i \geq 0 \}$ and H\"older's inequality to see that
\begin{equation*}
p_Z(z_1, z_1-z_2) \leq \prod_{i=1}^d \biggl(\P
\biggl\{|u_i(t,x)-u_i(s,y) | > |z_1^i -z_2^i | \biggr\}
\biggr)^{\frac{1}{2d}} \times \Vert H_{(1,...,2d)}(Z,1) \Vert_{0,2}.
\end{equation*}
Therefore, in order to prove the desired results (c) and (d) of Theorem \ref{t1}, it suffices to prove that:
\begin{equation} \label{equa1}
\prod_{i=1}^d \biggl(\P
\biggl\{|u_i(t,x)-u_i(s,y)| > |z_1^i-z_2^i|
\biggr\}\biggr)^{\frac{1}{2d}} \leq c\exp \biggl( -\frac{\Vert
z_1-z_2 \Vert^2}{c_T (|t-s|^{1/2}+|x-y|)}\biggr), 
\end{equation}
\begin{equation} \label{equa2b}
\Vert H_{(1,...,2d)}(Z,1) \Vert_{0,2} \leq c_T
(|t-s|^{1/2}+|x-y|)^{-(d+\eta)/2}, 
 \end{equation}
 and if $s=t$, then
\begin{equation} \label{equa3}
 \Vert H_{(1,...,2d)}(Z,1) \Vert_{0,2} \leq c_T |x-y|^{-d/2}.
\end{equation}
\vskip 12pt

\begin{proof}[Proof of (\ref{equa1})]
Let $\tilde{u}$ denote the solution
of (\ref{e2}) for $b \equiv 0$. Consider the continuous one-parameter
martingale $(M_u=(M_u^1,...,M_u^d), \, 0 \leq u \leq t)$ defined by
\begin{equation*}
M^i_u=
\begin{cases}
\int_0^u \int_0^1 (G_{t-r}(x,v)-G_{s-r}(y,v)) \,
\sum_{j=1}^d \sigma_{ij}(\tilde{u}(r,v)) \,  W^j(dr, dv)  &
\text{if $0 \leq u \leq s $}, \\
& \\ \int_0^s \int_0^1
(G_{t-r}(x,v)-G_{s-r}(y,v)) \, \sum_{j=1}^d \sigma_{ij}(\tilde{u}(r,v)) \,  W^j(dr, dv) &  \\
\qquad \qquad + \int_s^u \int_0^1 G_{t-r}(x,v) \,
\sum_{j=1}^d \sigma_{ij}(\tilde{u}(r,v)) \,  W^j(dr, dv) & \text{if $s \leq u \leq
t$},
\end{cases}
\end{equation*}
for all $i=1,...,d$, with respect to the filtration
($\mathcal{F}_u$, $0 \leq u \leq t$). Notice that
\begin{equation*}
M_0=0, \; \;  M_t=\tilde{u}(t,x)-\tilde{u}(s,y).
\end{equation*}
Moreover, by hypothesis {\bf P1} and Lemma~\ref{l1},
\begin{equation*} \begin{split}
\langle M^i \rangle_t&=\int_0^s \int_0^1
(G_{t-r}(x,v)-G_{s-r}(y,v))^2 \, \sum_{j=1}^d (\sigma_{ij}(\tilde{u}(r,v)))^2 \, dr
dv \\ 
& \qquad \qquad \qquad  +\int_s^t \int_0^1 (G_{t-r}(x,v))^2 \,\sum_{j=1}^d (\sigma_{ij}(\tilde{u}(r,v)))^2 \, dr
dv \\ 
&\leq C\int_0^T \int_0^1 (g(r,v))^2\, dr dv \\
& \leq C_T (|t-s|^{1/2}+|x-y|).
\end{split}\end{equation*}
By the exponential martingale inequality \textsc{Nualart} \cite[A.5]{Nualart:95},
\begin{equation} \label{b0}
\P \biggl\{|\tilde{u}^i(t,x)-\tilde{u}^i(s,y) | > |z_1^i
-z_2^i | \biggr\} \leq 2 \exp \biggl( -\frac{| z^i_1-z^i_2 |^2}{C_T
(|t-s|^{1/2}+|x-y|)}\biggr).
\end{equation}

We will now treat the case $b \not\equiv 0$ using Girsanov's
theorem. Consider the random variable
\begin{equation*} \begin{split}
&L_t=\text{exp}\biggl(- \int_0^t \int_0^1 \sigma^{-1}(u(r,v)) \,
b(u(r,v)) \cdot W(dr,dv) \\
& \qquad \qquad \qquad -\frac{1}{2} \int_0^t \int_0^1 \Vert
\sigma^{-1}(u(r,v)) \, b(u(r,v)) \Vert^2 \, dr dv \biggr).
\end{split}\end{equation*} 
The following Girsanov's theorem holds.
\begin{thm} \label{3t8bis} \textnormal{~\cite[Prop.1.6]{Nualart:94}}
$\E[L_t]=1$, and if $\tilde{P}$ denotes the
probability measure on $(\Omega, {\mathcal{F}})$ defined by
\begin{equation*}
\frac{d\tilde{\P}}{d \P} (\omega)= L_{t} (\omega),
\end{equation*}
then ${\tilde{W}}(t,x)=W(t,x)+\int_{0}^t \int_0^x
\sigma^{-1}(u(r,v)) \, b(u(r,v))\, dr dv$ is a standard Brownian
sheet under $\tilde{\P}$.
\end{thm}
Consequently, the law of $u$ under $\tilde{\P}$ coincides
with the law of $\tilde{u}$ under $\P$. Consider now the
random variable
\begin{equation*} \begin{split}
&J_t=\text{exp}\biggl(- \int_0^t \int_0^1
\sigma^{-1}(\tilde{u}(r,v)) \, b(\tilde{u}(r,v)) \cdot W(dr,dv) \\
&  \qquad \qquad \qquad+\frac{1}{2} \int_0^t \int_0^1 \Vert
\sigma^{-1}(\tilde{u}(r,v)) \, b(\tilde{u}(r,v)) \Vert^2 \, dr dv
\biggr).
\end{split}\end{equation*} 
Then, by Theorem \ref{3t8bis}, the Cauchy-Schwarz inequality and (\ref{b0}),
\begin{equation*} \begin{split}
&\P \biggl\{|u_i(t,x)-u_i(s,y) | > |z_1^i -z_2^i |
\biggr\} = \E_{\tilde{\P}} \biggl[
1_{\{|u^i(t,x)-u^i(s,y) |
|z_1^i -z_2^i |  \}} L_t^{-1} \biggr] \\ 
&   \qquad \qquad \qquad= \E_{\P} \biggl[
1_{\{|\tilde{u}^i(t,x)-\tilde{u}^i(s,y) |
|z_1^i -z_2^i |  \}} J_t^{-1} \biggr] \\ 
&   \qquad \qquad \qquad \leq \biggl(\P \biggl\{
|\tilde{u}^i(t,x)-\tilde{u}^i(s,y)|
|z_1^i - z_2^i | \biggr\} \biggr)^{1/2} \biggl( \E_{\P}[J_t^{-2}] \biggr)^{1/2} \\ 
&\qquad \qquad \qquad \leq  2 \exp \biggl( -\frac{| z^i_1-z^i_2 |^2}{C_T
(|t-s|^{1/2}+|x-y|)}\biggr)  \biggl( \E_{\P}
[J_t^{-2}] \biggr)^{1/2}.
\end{split}\end{equation*} 
Now, hypothesis {\bf P1} and {\bf P2} give
\begin{equation*} \begin{split}
\E_{\P} [J_t^{-2}]&\leq \E_{\P}
\biggl[ \text{exp}\biggl(\int_0^t \int_0^1 2 \,
\sigma^{-1}(\tilde{u}(r,v)) \,
b(\tilde{u}(r,v)) \cdot W(dr,dv) \\ 
& \qquad  \qquad \qquad-\frac{1}{2} \int_0^t \int_0^1 4 \, \Vert
\sigma^{-1}(\tilde{u}(r,v)) \, b(\tilde{u}(r,v)) \Vert^2 \, dr dv
\biggr) \\
& \qquad\qquad \qquad\times \text{exp}\biggl(\int_0^t \int_0^1
  \Vert \sigma^{-1}(\tilde{u}(r,v)) \, b(\tilde{u}(r,v)) \Vert^2
\, dr dv
\biggr)\biggr] \\
& \leq C,
\end{split}\end{equation*} 
since the second exponential is bounded and the first is an exponential martingale.

Therefore, we have proved that
\begin{equation*}
\P \biggl\{|u_i(t,x)-u_i(s,y) | > |z_1^i -z_2^i |
\biggr\}\leq C \exp \biggl( -\frac{| z^i_1-z^i_2 |^2}{C_T
(|t-s|^{1/2}+|x-y|)}\biggr),
\end{equation*}
from which we conclude that
\begin{equation*}
\prod_{i=1}^d \biggl(\P \biggl\{|u_i(t,x)-u_i(s,y) | >
|z_1^i -z_2^i | \biggr\} \biggr)^{\frac{1}{2d} } \leq C \exp
\biggl( -\frac{\Vert z_1-z_2 \Vert^2}{C_T
(|t-s|^{1/2}+|x-y|)}\biggr).
\end{equation*}
This proves (\ref{equa1}).
\end{proof}

\vskip 12pt
\begin{proof}[Proof of (\ref{equa2b})]
As in (\ref{iteration}), using the continuity of the Skorohod integral $\delta$ and H\"older's inequality for Malliavin norms, we obtain
\begin{equation*} \begin{split}
&\Vert H_{(1,...,2d)}(Z,1) \Vert_{0,2}  \leq C \Vert
H_{(1,...,2d-1)}(Z,1) \Vert_{1,4}  \biggl( \sum_{j=1}^{d} \Vert
(\gamma_{Z}^{-1})_{2d, j} \Vert_{1,8} \, \Vert D(Z^j)
\Vert_{1,8}
\\
&\qquad \qquad \qquad  \qquad \qquad \qquad \qquad +
\sum_{j=d+1}^{2d} \Vert (\gamma_{Z}^{-1})_{2d, j} \Vert_{1,8}
\, \Vert D(Z^j) \Vert_{1,8}\biggr).
\end{split}\end{equation*} 
Notice that the entries of $\gamma_{Z}^{-1}$ that appear in this expression belong to sets {\bf (3)} and {\bf (4)}
of indices, as defined before Theorem~\ref{ga}. From Theorem~\ref{ga}(a) and Propositions~\ref{p4} and~\ref{3p1}, we find that this is bounded above by
\begin{equation*}
C_T \Vert H_{(1,...,2d-1)}(Z,1) \Vert_{1,4} \biggl( \sum_{j=1}^{d} (|t-s|^{1/2}+|x-y|)^{-\frac{1}{2}-d \eta} +\sum_{j=d+1}^{2d} (|t-s|^{1/2}+|x-y|)^{-1-d \eta+\frac{1}{2}} \biggr),
\end{equation*}
that is, by
\begin{equation*}
C_T \Vert H_{(1,...,2d-1)}(Z,1) \Vert_{1,4} (|t-s|^{1/2}+|x-y|)^{-1/2-d \eta}.
\end{equation*}
Iterating this procedure $d$ times (during which we only encounter coefficients $(\gamma_Z^{-1})_{m,l}$ for $(m,l)$ in blocs ${\bf (3)}$ and ${\bf (4)}$, cf. Theorem~\ref{ga}(a)), we get, for some integers $m_0,k_0>0$,
\begin{equation*}
\Vert H_{(1,...,2d)}(Z,1) \Vert_{0,2}  \leq C_T \Vert
H_{(1,...,d)}(Z,1) \Vert_{m_0,k_0} (|t-s|^{1/2}+|x-y|)^{-d/2-d^2 \eta}.
\end{equation*}
Again, using the continuity of $\delta$ and H\"older's inequality
for the Malliavin norms, we obtain
\begin{equation*} \begin{split}
&\Vert H_{(1,...,d)}(Z,1) \Vert_{m,k}  \leq C \Vert
H_{(1,...,d-1)}(Z,1) \Vert_{m_1,k_1}  \biggl( \sum_{j=1}^{d} \Vert
(\gamma_{Z}^{-1})_{d, j} \Vert_{m_2,k_2} \, \Vert D(Z^j)
\Vert_{m_3,k_3}
\\
& \qquad \qquad \qquad \qquad \qquad \qquad \qquad \qquad +
\sum_{j=d+1}^{2d} \Vert (\gamma_{Z}^{-1})_{d, j}
\Vert_{m_4,k_4} \, \Vert D(Z^j) \Vert_{m_5,k_5}\biggr),
\end{split}\end{equation*} 
for some integers $m_i,k_i>0$, $i=1,...,5$. This time, the entries of $\gamma_Z^{-1}$ that appear in this expression come from the sets ${\bf (1)}$ and ${\bf (2)}$ of indices. 
We appeal again to
Theorem~\ref{ga}(a) and Propositions~\ref{p4} and~\ref{3p1} to get
\begin{equation*}
\Vert H_{(1,...,d)}(Z,1) \Vert_{m,k}  \leq C_T \Vert
H_{(1,...,d-1)}(Z,1) \Vert_{m_1,k_1}(|t-s|^{1/2}+|x-y|)^{-d \eta}.
\end{equation*}
Finally, iterating this procedure $d$ times (during which we encounter coefficients $(\gamma_Z^{-1})_{m,l}$ for $(m,l)$
in blocs ${\bf (1)}$ and ${\bf (2)}$ only, cf. Theorem~\ref{ga}(a)), and choosing $\eta^{\prime}=4d^2\eta$, we conclude that
\begin{equation*}
\Vert H_{(1,...,2d)}(Z,1) \Vert_{0,2}  \leq C_T
(|t-s|^{1/2}+|x-y|)^{-(d+\eta^{\prime})/2},
\end{equation*}
which proves (\ref{equa2b}) and concludes the proof of
Theorem \ref{t1}(c).
\end{proof}
\vskip 12pt

\begin{proof}[Proof of (\ref{equa3})]
In order to prove (\ref{equa3}), we proceed exactly along the same lines as in the proof of (\ref{equa2b}) but we appeal to Theorem \ref{ga}(b). This concludes the proof of
Theorem \ref{t1}(d).
\end{proof}

\subsection{Proof of Theorem \ref{ga}} \label{p63}

Let $Z$ as in (\ref{vector}).
Since the inverse of the matrix $\gamma_Z$ is the inverse of its determinant multiplied by its cofactor matrix, we examine these two factors separately.
\begin{prop} \label{3p2}
Fix $T>0$ and let $I$ and $J$ be compact intervals as in Theorem \ref{t1}. Assuming {\bf P1}, for any $(s,y), (t,x) \in I \times J$, $(s,y) \neq (t,x)$, $p>1$,
\begin{equation*}
\E[|(A_Z)_{m,l}|^p]^{1/p}\leq
\begin{cases}
c_{p,T} (|t-s|^{1/2}+|x-y|)^{d}&
\text{if \; $(m,l) \in {\bf (1)} $}, \\
c_{p,T} (|t-s|^{1/2}+|x-y|)^{d-\frac{1}{2}}&
\text{if \; $(m,l) \in {\bf (2)}$ or ${\bf (3)} $},  \\
c_{p,T} (|t-s|^{1/2}+|x-y|)^{d-1}& \text{if \; $(m,l) \in {\bf (4)}$},
\end{cases}
\end{equation*}
where $A_Z$ denotes the cofactor matrix of $\gamma_Z$.
\end{prop}

\begin{proof}
We consider the four different cases.
\vskip 12pt

\noindent $\bullet$ If $(m,l) \in {\bf (1)} $, we claim that
\begin{equation} \begin{split} \label{adjointmatrix} 
\vert (A_Z)_{m,l}  \vert &\leq  C \sum_{k=0}^{d-1} \biggl\{ \Vert D(u(s,y)) \Vert^{2k}_{\mathcal{H}} \times
\Vert D(u(t,x)-u(s,y)) \Vert^{2(d-1-k)}_{\mathcal{H}} \\
&\qquad \qquad \times \Vert D(u(s,y)) \Vert^{2(d-1-k)}_{\mathcal{H}} \times  \Vert D(u(t,x)-u(s,y)) \Vert^{2(k+1)}_{\mathcal{H}} \biggr\} .
\end{split}\end{equation} 
Indeed, let $A_Z^{m,l}=(a^{m,l}_{\bar{m},\bar{l}})_{\bar{m},\bar{l}=1,...,2d-1}$ be the $(2d-1) \times (2d-1)$-matrix obtained by removing
from $\gamma_Z$ its row $m$ and column $l$. Then
\begin{equation*}
(A_Z)_{m,l} =\text{det} \, ((A_Z)^{m,l}) = \sum_{\pi \text{ permutation of }  (1,...,2d-1)} a_{1, \pi(1)}^{m,l} \cdots    a_{2d-1, \pi(2d-1)}^{m,l}.
\end{equation*}
Each term of this sum contains one entry from each row and column of $A_Z^{m,l}$. If there are $k$ entries taken from bloc ${\bf (1)}$ of $\gamma_Z$, these occupy $k$ rows
and columns of $A_Z^{m,l}$. Therefore, $d-1-k$ entries must come from the $d-1$ remaining rows of bloc ${\bf (2)}$, and the same number from the columns of bloc ${\bf (3)}$.
Finally, there remain $k+1$ entries to be taken from bloc ${\bf (4)}$. Therefore,
\begin{equation*} \begin{split}
&\vert (A_Z)_{m, l} \vert  \leq  C  \sum_{k=0}^{d-1} \sum
\biggl\{ (\text{product of } k \text{ entries from } {\bf (1)} ) \\ 
&\qquad \times  (\text{product of } d-1-k \text{ entries from } {\bf (2)} )\\ 
&\qquad  \times  (\text{product of } d-1-k \text{ entries from } {\bf (3)} ) \times (\text{product of } k+1 \text{ entries from } {\bf (4)} ) \biggr\}.
\end{split}\end{equation*} 
Adding all the terms and using the particular form of these terms establishes (\ref{adjointmatrix}).

   Regrouping the various factors in (\ref{adjointmatrix}), applying the Cauchy-Schwarz inequality and using (\ref{sup})
and Proposition~\ref{3p1}, we obtain
\begin{equation*} \begin{split}
\E \biggl[ \vert (A_Z)_{m,l} \vert^p \biggr]  &\leq C
\sum_{k=0}^{d} \E \biggl[\Vert D(u(s,y))
\Vert^{2(d-1)p}_{\mathcal{H}}
\times  \Vert D(u(t,x)-u(s,y)) \Vert^{2dp}_{\mathcal{H}} \biggr] \\
&\leq C_T (|t-s|^{1/2}+|x-y|)^{dp}.
\end{split}\end{equation*} 
\vskip 12pt

\noindent $\bullet$ If $(m,l) \in {\bf (2)}$ or $(m,l) \in {\bf (3)}$, then using the same arguments as above, we obtain
\begin{equation*} \begin{split}
&\vert (A_Z)_{m,l} \vert \\
& \leq  C \sum_{k=0}^{d-1} \sum \biggl\{ \Vert D(u(s,y)) \Vert^{2(d-1-k)}_{\mathcal{H}} \times  \Vert D(u(s,y)) \Vert^{k}_{\mathcal{H}} \times
\Vert D(u(t,x)-u(s,y)) \Vert^{k}_{\mathcal{H}} \\ 
& \qquad \times \Vert D(u(s,y)) \Vert^{k+1}_{\mathcal{H}} \times  \Vert D(u(t,x)-u(s,y)) \Vert^{k+1}_{\mathcal{H}}
\times  \Vert D(u(t,x)-u(s,y)) \Vert^{2(d-1-k)}_{\mathcal{H}}\biggr\} \\
& \leq C \sum_{k=0}^{d-1} \biggl\{ \Vert D(u(s,y)) \Vert^{2d-1}_{\mathcal{H}} \times \Vert D(u(t,x)-u(s,y)) \Vert^{2d-1}_{\mathcal{H}}\biggr\},
\end{split}\end{equation*} 
from which we conclude, using (\ref{sup}) and Proposition~\ref{3p1}, that
\begin{equation*}
\E \biggl[ \vert (A_Z)_{m,l} \vert^{p} \biggr] \leq C_T
(|t-s|^{1/2}+|x-y|)^{(d-\frac{1}{2})p}.
\end{equation*}
\vskip 12pt

\noindent $\bullet$ If $(i,j) \in {\bf (4)}$, we obtain
\begin{equation*} \begin{split}
\vert (A_Z)_{m,l} \vert &\leq  C \sum_{k=0}^{d-1} \sum \biggl\{ \Vert D(u(s,y)) \Vert^{2(k+1)}_{\mathcal{H}} \times
\Vert D(u(s,y)) \Vert^{2(d-1-k)}_{\mathcal{H}} \\ 
&\qquad \times \Vert D(u(t,x)- D(u(s,y))) \Vert^{2(d-1-k)}_{\mathcal{H}} \times \Vert D(u(t,x)- D(u(s,y))) \Vert^{2k}_{\mathcal{H}} \biggr\} \\
&  \leq C \sum_{k=0}^{d-1} \biggl\{ \Vert D(u(s,y)) \Vert^{2d}_{\mathcal{H}} \times \Vert D(u(t,x)-u(s,y)) \Vert^{2d-2}_{\mathcal{H}}\biggr\},
\end{split}\end{equation*} 
from which we conclude that
\begin{equation*}
\E \biggl[ \vert (A_Z)_{m,l} \vert^{p} \biggr] \leq C_T
(|t-s|^{1/2}+|x-y|)^{(d-1)p}.
\end{equation*}

   This concludes the proof of the proposition.
\end{proof}

\begin{prop} \label{3p3}
Fix $\eta,T>0$. Assume {\bf P1} and {\bf P2}. Let $I$ and $J$ be compact intervals as in Theorem \ref{t1}.
\begin{enumerate}
\item[\textnormal{(a)}] There exists
$C$ depending on $T$ and $\eta$ such that for any $(s,y), (t,x) \in I\times J$, $(s,y) \neq (t,x)$, $p>1$,
\begin{equation} \label{bruce}
\E \big[\big(\textnormal{det} \, \gamma_Z\big)^{-p}\big]^{1/p} \leq C
(|t-s|^{1/2}+|x-y|)^{-d(1+\eta)}.
\end{equation}
\item[\textnormal{(b)}] There exists
$C$ only depending on $T$ such that for any $s=t \in I$,
$x,y \in J$, $x \neq y$, $p>1$,
\begin{equation*}
\E \big[\big(\textnormal{det} \, \gamma_Z \big)^{-p} \big]^{1/p} \leq
C (|x-y|)^{-d}.
\end{equation*}
\end{enumerate}
\end{prop}

Assuming this proposition, we will be able to conclude the proof of Theorem~\ref{ga}, after establishing the following estimate on the derivative of the Malliavin matrix.
\begin{prop} \label{dgamma}
Fix $T>0$. Let $I$ and $J$ be compact intervals as in Theorem \ref{t1}. Assuming {\bf P1}, for any $(s,y),(t,x) \in I \times J$, $(s,y) \neq (t,x)$, $p>1$ and $k \geq 1$,
\begin{equation*}
\E \big[\Vert D^k(\gamma_Z)_{m,l} \Vert_{\mathcal{H}^{\otimes k}} ^{p}\big]^{1/p}\leq
\begin{cases}
c_{k,p,T} &
\text{if \; $(m,l) \in {\bf (1)} $}, \\
c_{k,p,T} (|t-s|^{1/2}+|x-y|)^{1/2}&
\text{if \; $(m,l) \in {\bf (2)}$ or ${\bf (3)} $},  \\
c_{k,p,T} (|t-s|^{1/2}+|x-y|)& \text{if \; $(m,l) \in {\bf (4)}$}.
\end{cases}
\end{equation*}
\end{prop}

\begin{proof} We consider the four different blocs.
\begin{enumerate}
\item[$\bullet$] If $(m,l) \in {\bf (4)}$, proceeding as in \textsc{Dalang and Nualart} \cite[p.2131]{Dalang:04} and appealing to
Proposition~\ref{3p1} twice, we obtain
\begin{equation*} \begin{split}
& \E \big[\big\Vert D^k (\gamma_{Z})_{m,l} \big\Vert^{p}_{\mathcal{H}^{\otimes k}} \big]  \\ 
&=\E \big[\big\Vert D^k\big( \int_0^T dr \int_0^1 dv \, D_{r,v}(u_m(t,x)-u_m(s,y)) \cdot
 D_{r,v}(u_l(t,x)-u_l(s,y)) \big) \big\Vert^p_{\mathcal{H}^{\otimes k}}\big] \\ 
&\leq(k+1)^{p-1} \sum_{j=0}^k \left(\begin{array}{c} \!\!k\!\! \\  \!\!j\! \! \end{array} \right)^p
\E \big[\big\Vert \int_0^T dr \int_0^1 dv \,  D^j D_{r,v}(u_m(t,x)-u_m(s,y)) \\ 
&\qquad \qquad \qquad \qquad \qquad \qquad\cdot
D^{k-j} D_{r,v}(u_l(t,x)-u_l(s,y))  \big\Vert^{p}_{\mathcal{H}^{\otimes k}}\big] \\ 
&\leq \tilde{C}_T (k+1)^{p-1} \sum_{j=0}^k \left(\begin{array}{c} \!\!k\!\! \\  \!\!j\! \! \end{array} \right)^p
\big\{ \big(E\big[ \big\Vert D^j D(u_m(t,x)-u_m(s,y)) \big\Vert^{2p}_{\mathcal{H}^{\otimes (j+1)}}  \big]\big)^{1/2} \\
&\qquad \qquad \qquad \qquad \qquad \qquad\times
\big(\E\big[\big\Vert D^{k-j} D(u_l(t,x)-u_l(s,y)) \big\Vert^{2p}_{\mathcal{H}^{\otimes (k-j+1)}} \big]\big)^{1/2} \big\}  \\
& \leq C_T (|t-s|^{1/2}+|x-y|)^p.
\end{split}\end{equation*} 

\item[$\bullet$] If $(m,l) \in {\bf (2)}$ or $(m,l) \in {\bf (3)}$, proceeding as above and appealing to (\ref{sup})
and Proposition~\ref{3p1}, we get
\begin{equation*} \begin{split}
& \E \big[\big\Vert D^k (\gamma_{Z})_{m,l} \big\Vert^{p}_{\mathcal{H}^{\otimes k}} \big]  \\ 
&\leq  \tilde{C}_T (k+1)^{p-1} \sum_{j=0}^k \left(\begin{array}{c} \!\!k\!\! \\  \!\!j\! \! \end{array} \right)^p
\big\{ \big( \E\big[ \big\Vert D^j D(u_m(t,x)-u_m(s,y)) \big\Vert^{2p}_{\mathcal{H}^{\otimes (j+1)}}  \big]\big)^{1/2} \\ 
&\qquad \qquad \qquad \qquad \qquad \qquad\times
\big( \E\big[\big\Vert D^{k-j} D(u_l(s,y)) \big\Vert^{2p}_{\mathcal{H}^{\otimes (k-j+1)}} \big]\big)^{1/2} \big\}  \\
& \leq C_T (|t-s|^{1/2}+|x-y|)^{p/2}.
\end{split}\end{equation*}
\item[$\bullet$] If $(m,l) \in {\bf (1)}$, using (\ref{sup}), we obtain
\begin{equation*} \begin{split}
& \E \big[\big\Vert D^k (\gamma_{Z})_{m,l} \big\Vert^{p}_{\mathcal{H}^{\otimes k}} \big]  \\ 
&\leq  \tilde{C}_T (k+1)^{p-1} \sum_{j=0}^k \left(\begin{array}{c} \!\!k\!\! \\  \!\!j\! \! \end{array} \right)^p
\big\{ \big( \E\big[ \big\Vert D^j D(u_m(s,y)) \big\Vert^{2p}_{\mathcal{H}^{\otimes (j+1)}}  \big]\big)^{1/2} \\
&\qquad \qquad \qquad \qquad \qquad \qquad\times
\big( \E \big[\big\Vert D^{k-j} D(u_l(s,y)) \big\Vert^{2p}_{\mathcal{H}^{\otimes (k-j+1)}} \big]\big)^{1/2} \big\}  \\
& \leq C_T.
\end{split}\end{equation*}
\end{enumerate}
\end{proof}

\vskip 12pt

\begin{proof}[Proof of Theorem \ref{ga}]
When $k=0$, the result follows directly using the fact that the inverse of a matrix is the inverse of its determinant multiplied by its cofactor matrix and the estimates of Propositions~\ref{3p2} and \ref{3p3}.

For $k \geq1$, we shall establish the following two properties.
\begin{enumerate}
\item[\textnormal{(a)}] For any  $(s,y),(t,x) \in I \times J$, $(s,y) \neq (t,x)$, $s \leq t$,  $k \geq 1$ and $p>1$,
\begin{equation*}
\E [\Vert D^k (\gamma_Z^{-1})_{m,l} \Vert_{\mathcal{H}^{\otimes k}}^p]^{1/p} \leq
\begin{cases}
c_{k,p,\eta,T} (|t-s|^{1/2}+|x-y|)^{-d \eta}&
\text{if  $(m,l) \in {\bf (1)}$ }, \\
c_{k,p,\eta,T} (|t-s|^{1/2}+|x-y|)^{-1/2-d \eta} &
\text{if  $(m,l) \in {\bf (2)}$, ${\bf (3)}$},  \\
c_{k,p,\eta,T} (|t-s|^{1/2}+|x-y|)^{-1-d \eta}& \text{if  $(m,l) \in {\bf (4)}$}.
\end{cases}
\end{equation*}
\item[\textnormal{(b)}] For any $s=t \in I$, $x,y \in J$, $x \neq y$, $k \geq 1$ and $p>1$,
\begin{equation*}
\E [\Vert D^k (\gamma_Z^{-1})_{m,l} \Vert_{\mathcal{H}^{\otimes k}}^p]^{1/p} \leq
\begin{cases}
c_{k,p,T}&
\text{if \; $(m,l) \in {\bf (1)}$ }, \\
c_{k,p,T} |x-y|^{-1/2} &
\text{if \; $(m,l) \in {\bf (2)}$ or ${\bf (3)}$},  \\
c_{k,p,T} |x-y|^{-1}& \text{if \; $(m,l) \in {\bf (4)}$}.
\end{cases}
\end{equation*}
\end{enumerate}
Since
\begin{equation*}
\Vert (\gamma_Z^{-1})_{m,l} \Vert_{k,p}=\bigg\{ 
\E [\vert (\gamma_Z^{-1})_{m,l} \vert^p]+\sum_{j=1}^k \E [\Vert D^j (\gamma_Z^{-1})_{m,l} \Vert_{\mathcal{H}^{\otimes j}}^p]\bigg\}^{1/p},
\end{equation*}
(a) and (b) prove the theorem.

We now prove (a) and (b).
When $k=1$, we will use (\ref{eulalia}) written as a matrix product:
\begin{equation} \label{productmatrix}
D(\gamma_{Z}^{-1})=\gamma_{Z}^{-1} D (\gamma_{Z}) \gamma_{Z}^{-1}.
\end{equation}
Writing (\ref{productmatrix}) in bloc product matrix notation with blocs
{\bf(1)}, {\bf(2)}, {\bf(3)} and {\bf(4)}, we get that
\begin{equation*} \begin{split}
D ( (\gamma_{Z}^{-1})^{(1)} )&=(\gamma_{Z}^{-1})^{(1)} D (\gamma_{Z}^{(1)}) (\gamma_{Z}^{-1})^{(1)}+
(\gamma_{Z}^{-1})^{(1)} D (\gamma_{Z}^{(2)}) (\gamma_{Z}^{-1})^{(3)} \\ 
& \qquad +(\gamma_{Z}^{-1})^{(2)} D (\gamma_{Z}^{(3)}) (\gamma_{Z}^{-1})^{(1)}+(\gamma_{Z}^{-1})^{(2)} D (\gamma_{Z}^{(4)}) (\gamma_{Z}^{-1})^{(3)}, \\ 
D ( (\gamma_{Z}^{-1})^{(2)} )&=(\gamma_{Z}^{-1})^{(1)} D (\gamma_{Z}^{(1)}) (\gamma_{Z}^{-1})^{(2)}+
(\gamma_{Z}^{-1})^{(1)} D (\gamma_{Z}^{(2)}) (\gamma_{Z}^{-1})^{(4)} \\ 
& \qquad +(\gamma_{Z}^{-1})^{(2)} D (\gamma_{Z}^{(3)}) (\gamma_{Z}^{-1})^{(2)}+(\gamma_{Z}^{-1})^{(2)} D (\gamma_{Z}^{(4)}) (\gamma_{Z}^{-1})^{(4)}, \\ 
D ( (\gamma_{Z}^{-1})^{(3)} )&=(\gamma_{Z}^{-1})^{(3)} D (\gamma_{Z}^{(1)}) (\gamma_{Z}^{-1})^{(1)}+
(\gamma_{Z}^{-1})^{(3)} D (\gamma_{Z}^{(2)}) (\gamma_{Z}^{-1})^{(3)} \\ 
& \qquad +(\gamma_{Z}^{-1})^{(4)} D (\gamma_{Z}^{(3)}) (\gamma_{Z}^{-1})^{(1)}+(\gamma_{Z}^{-1})^{(4)} D (\gamma_{Z}^{(4)}) (\gamma_{Z}^{-1})^{(3)}, \\ 
D ( (\gamma_{Z}^{-1})^{(4)} )&=(\gamma_{Z}^{-1})^{(3)} D (\gamma_{Z}^{(1)}) (\gamma_{Z}^{-1})^{(2)}+
(\gamma_{Z}^{-1})^{(3)} D (\gamma_{Z}^{(2)}) (\gamma_{Z}^{-1})^{(4)} \\
& \qquad +(\gamma_{Z}^{-1})^{(4)} D (\gamma_{Z}^{(3)}) (\gamma_{Z}^{-1})^{(2)}+(\gamma_{Z}^{-1})^{(4)} D (\gamma_{Z}^{(4)}) (\gamma_{Z}^{-1})^{(4)}.
\end{split}\end{equation*}
It now suffices to apply H\"older's inequality to each block and
use the estimates of the case $k=0$ and Proposition \ref{dgamma}
to obtain the desired result for $k=1$. For instance, for $(m,l)\in {\bf(1)}$,
\begin{equation*} \begin{split}
&\E \biggl[ \big\Vert\big( (\gamma_{Z}^{-1})^{(2)} D (\gamma_{Z}^{(4)})
(\gamma_{Z}^{-1})^{(3)} \big)_{m,l} \big\Vert^p_{\mathcal{H}} \biggr]^{1/p} \\ 
& \leq \sup_{m_1,l_1} \E \biggl[\big\vert \big( (\gamma_{Z}^{-1})^{(2)} \big)_{m_1,l_1} \big\vert^{2p}
\biggr]^{1/(2p)} \sup_{m_2,l_2} \E\biggl[\big\Vert \big( D (\gamma_{Z}^{(4)}) \big)_{m_2,l_2}
\big\Vert^{4p}_{\mathcal{H}} \biggr]^{1/(4p)} \\
& \qquad \qquad \qquad \times \sup_{m_3,l_3}  \E \biggl[\big\vert
\big( (\gamma_{Z}^{-1})^{(3)} \big)_{m_3,l_3} \big\vert^{4p} \biggr]^{1/(4p)} \\
&\leq c \, (|t-s|^{1/2}+|x-y|)^{-\frac{1}{2}-d \eta +1 -\frac{1}{2}-d \eta }=c \, (|t-s|^{1/2}+|x-y|)^{-2d \eta}.
\end{split}\end{equation*}

For $k\geq 1$, in order to calculate $D^{k+1}(\gamma_Z^{(\cdot)})$, we will need to compute
$D^k(\gamma_{Z}^{-1} D (\gamma_{Z}) \gamma_{Z}^{-1})$.
For bloc numbers $i_1,i_2,i_3 \in \{1,2,3,4\}$ and $k \geq 1$, we have
\begin{equation*} \begin{split}
&D^k\big((\gamma_{Z}^{-1})^{(i_1)} D (\gamma_{Z}^{(i_2)}) (\gamma_{Z}^{-1})^{(i_3)}\big) \\ 
& =
\sum_{\begin{array}{c} \scriptstyle  j_1+j_2+j_3=k \\ \scriptstyle j_i \in \{0,...,k\} \end{array}} \left(\begin{array}{c} \!\!k\!\! \\  \!\!j_1 \, j_2 \, j_3 \! \! \end{array} \right)
 D^{j_1}\big((\gamma_{Z}^{-1})^{(i_1)}\big) D^{j_2}\big(D (\gamma_{Z}^{(i_2)})\big)D^{j_3}\big((\gamma_{Z}^{-1})^{(i_3)}\big).
\end{split}\end{equation*}
Note that by Proposition \ref{dgamma}, the norms of the derivatives $D^{j_2}\big(D (\gamma_{Z}^{(i_2)})$ of $\gamma_{Z}^{(i_2)}$ are of the same order for all
$j_2$.
Hence, we appeal again to H\"older's inequality and Proposition \ref{dgamma}, and use a recursive argument in order to obtain the desired bounds.
\end{proof}

\vskip 12pt

\begin{proof}[Proof of Proposition \ref{3p3}]
The main idea for the proof of Proposition \ref{3p3} is to use a perturbation argument.
Indeed, for $(t,x)$ close to $(s,y)$,
the matrix $\gamma_Z$ is close to
\begin{equation*}
\hat{\gamma} =\left(
\begin{array}{ccc}
\gamma_Z^{(1)} & \vdots & 0\\
\cdots & \vdots & \cdots \\
0 & \vdots & 0\\
\end{array}
\right).
\end{equation*}
The matrix $\hat{\gamma}$ has $d$ eigenvectors of the form $(\hat{\lambda}^1,{\bf 0}),...,(\hat{\lambda}^d, {\bf 0})$,
where $\hat{\lambda}^1,...,\hat{\lambda}^d \in \mathbb{R}^d$ are eigenvectors of $\gamma_Z^{(1)}=\gamma_{u(s,y)}$,
and ${\bf 0}=(0,...,0) \in \mathbb{R}^d$, and $d$ other eigenvectors of the form $({\bf 0},e^i)$ where $e^1,...,e^d$ is a basis
of $\mathbb{R}^d$. These last eigenvectors of $\hat{\gamma}$ are associated with the eigenvalue $0$, while the former
are associated with eigenvalues of order $1$, as can be seen in the proof of Proposition~\ref{p5}.

We now write
\begin{equation} \label{edf}
\text{det} \, \gamma_Z = \prod_{i=1}^{2d} (\xi^i)^{T} \gamma_Z
\xi^i,
\end{equation}
where $\xi=\{\xi^1,...,\xi^{2d}\}$ is an orthonormal basis of
${\R}^{2d}$ consisting of eigenvectors of $\gamma_Z$.
We then expect that for $(t,x)$ close to $(s,y)$, there will be $d$ eigenvectors close to the subspace generated by the
$(\hat{\lambda}^i, {\bf 0})$, which will contribute a factor of order $1$ to the product in (\ref{edf}), and $d$ other eigenvectors,
close to the subspace generated by the $({\bf 0},e^i)$, that will each contribute a factor of order $(|t-s|^{1/2}+|x-y|)^{-1-\eta}$
to the product. Note that if we do not distinguish between these two types of eigenvectors, but simply bound below
the product by the smallest eigenvalue to the power $2d$, following the approach used in the proof of Proposition~\ref{p5}, then we would obtain
$C (|t-s|^{1/2}+|x-y|)^{-2dp}$ in the right-hand side of (\ref{bruce}), which would not be the correct order.

We now carry out this somewhat involved perturbation argument. Consider the spaces $E_1=\{ (\lambda, {\bf 0}) : \lambda \in {\R}^d, {\bf
0} \in \mathbb{R}^d\}$ and $E_2=\{ ({\bf 0}, \mu) : \mu \in {\R}^d, {\bf 0}
\in \mathbb{R}^d\}$. Note that every $\xi^i$ can be written as
\begin{equation} \label{base}
\xi^i=(\lambda_i, \mu_i)= \alpha_i ({\tilde{\lambda}}^i,{\bf 0})
+ \sqrt{1-\alpha_i^2} \, ({\bf 0}, {\tilde{\mu}}^i),
\end{equation}
where $\lambda_i, \mu_i \in \mathbb{R}^d$, $({\tilde{\lambda}}^i,{\bf 0})
\in E_1$, $({\bf 0}, {\tilde{\mu}}^i) \in E_2$, with $\Vert
{\tilde{\lambda}}^i \Vert=\Vert {\tilde{\mu}}^i \Vert=1$ and $0\leq
\alpha_i \leq 1$. Note in particular that $\Vert \xi^i \Vert^2=\Vert \lambda_i
\Vert^2 + \Vert \mu_i \Vert^2=1$ (norms of elements of $\R^d$ or $\R^{2d}$ are Euclidean norms).

\begin{lem} \label{pertur}
Given a sufficiently small $\alpha_0>0$, with probability one, there exist at least $d$ of these vectors, say $\xi^1,...,\xi^d$,
such that $\alpha_1 \geq \alpha_0,...,\alpha_d \geq \alpha_0$.
\end{lem}

\begin{proof}
Observe that as $\xi$ is an orthogonal family and for $i \neq j$, the Euclidean inner product of $\xi^i$ and $\xi^j$ is
\begin{equation*}
\xi^i \cdot \xi^j = \alpha_i \alpha_j \, ({\tilde{\lambda}}^i
\cdot {\tilde{\lambda}}^j)+\sqrt{1-\alpha_i^2}
\sqrt{1-\alpha_j^2} \,  ({\tilde{\mu}}^i \cdot {\tilde{\mu}}^j)=0.
\end{equation*}
For $\alpha_0>0$, let $D=\{ i \in \{ 1,...,2d\} : \alpha_i <
\alpha_0 \}$. Then, for $i,j \in D$, $i \neq j$, if $\alpha_0 <
\frac{1}{2}$, then
\begin{equation*}
\vert {\tilde{\mu}}^i \cdot {\tilde{\mu}}^j \vert =
\frac{\alpha_i \alpha_j }{\sqrt{1-\alpha_i^2} \sqrt{1-\alpha_j^2}}
\vert {\tilde{\lambda}}^i \cdot {\tilde{\lambda}}^j \vert \leq
\frac{\alpha_0^2}{1-\alpha_0^2} \Vert {\tilde{\lambda}}^i\Vert
\Vert {\tilde{\lambda}}^j \Vert \leq \frac{1}{3}
\alpha_0^2.
\end{equation*}
Since the diagonal terms of the matrix $({\tilde{\mu}}^i \cdot {\tilde{\mu}}^j)_{i,j \in D}$ are all equal to $1$, for $\alpha_0$ sufficiently small, it follows that $\text{det}
(({\tilde{\mu}}^i \cdot {\tilde{\mu}}^j)_{i,j \in D}) \neq 0$.
Therefore, $\{ {\tilde{\mu}}^i, \, i \in D \}$ is a linearly independent family, and,
as $({\bf 0}, {\tilde{\mu}}^i) \in E_2$, for $i=1,...,2d$, we conclude that a.s.,
$\text{card}(D) \leq \text{dim}(E_2)=d$. We can therefore assume that $\{1,...,d\} \subset D^{c}$ and so $\alpha_1 \geq \alpha_0$,...,$\alpha_d \geq \alpha_0$.
\end{proof}

\vskip 12pt
By Lemma~\ref{pertur} and Cauchy-Schwarz inequality one can write
\begin{equation*} \begin{split}
\E \big[\big(\textnormal{det} \, \gamma_Z\big)^{-p}\big]^{1/p}  &\leq
\biggl( \E \biggl[\biggl( \prod_{i=1}^{d} (\xi^i)^{T}
\gamma_Z \xi^i \biggr)^{-2p} \biggr] \biggr)^{1/(2p)} \\
& \qquad \qquad  \times \biggl( \E \left[
\left(\inf_{\substack{\xi =(\lambda, \mu) \in \R^{2d} :\\ \|
\lambda \| ^2 + \|  \mu \| ^2=1}} \,\xi^{\sf T} \gamma_Z  \xi
\right)^{-2dp} \right] \biggr)^{1/(2p)}.
\end{split}\end{equation*}
With this, Propositions \ref{pr:small-eigen:1} and \ref{large} below conclude the proof of Proposition~\ref{3p3}.
\end{proof}

\subsubsection{Small Eigenvalues}

Let $I$ and $J$ two compact intervals as in Theorem \ref{t1}.
\begin{prop}\label{pr:small-eigen:1}
Fix $\eta,T>0$. Assume {\bf P1} and {\bf P2}.
\begin{enumerate}
\item[\textnormal{(a)}] There exists
$C$ depending on $\eta$ and $T$ such that for all $s,t \in I$, $0 < t-s <1$, $x,y \in J$, $x \neq y$,
and $p>1$,
\begin{equation*}
\E\left[ \left(      \inf_{\substack{\xi =(\lambda, \mu) \in \R^{2d} :\\ \|  \lambda
\| ^2 + \|  \mu \| ^2=1}} \,\xi^{\sf T} \gamma_Z  \xi
\right)^{-2dp} \right] \leq C (|t-s|^{1/2}+|x-y|)^{-2dp(1+\eta)}.
\end{equation*}
\item[\textnormal{(b)}] There exists
$C$ depending only on $T$ such that for all $s =  t \in I$, $x,y \in J$, $x \neq y$,
and $p>1$,
\begin{equation*}
\E \left[ \left(
\inf_{\substack{\xi =(\lambda, \mu) \in \R^{2d} :\\ \|  \lambda
\| ^2 + \|  \mu \| ^2=1}} \,\xi^{\sf T} \gamma_Z  \xi
\right)^{-2dp} \right] \leq C (|x-y|)^{-2dp}.
\end{equation*}
\end{enumerate}
\end{prop}

\begin{proof} We begin by proving (a). Since $\gamma_Z$ is a matrix of inner products, we can write
\begin{equation*}
\xi^{\sf T} \gamma_Z  \xi = \sum_{k=1}^d \int_0^T dr \int_0^1 dv \biggl(
 \sum_{i=1}^d \left( \lambda_i D^{(k)}_{r,v}(u(s,y)) + \mu_i
(D^{(k)}_{r,v}(u(t,x))-D^{(k)}_{r,v}(u(s,y))) \right) \biggr)^2.
\end{equation*}
Therefore, for $\epsilon \in (0, t-s)$,
\begin{equation*}
\xi^{\sf T} \gamma_Z  \xi \geq J_1+J_2,
\end{equation*}
where
\begin{equation*} \begin{split}
J_1 &:= \sum_{k=1}^d \int_{s-\e}^s dr \int_0^1 dv \left( \sum_{i=1}^d (\lambda_i-\mu_i)\left[
G_{s-r}(y,v) \sigma_{ik}(u(r,v))+a_i(k,r,v,s,y)\right] + W\right)^2,\\
J_2 &:= \sum_{k=1}^d \int_{t-\e}^t dr \int_0^1 dv\,
W^2,
\end{split}\end{equation*}
$a_i(k,r,v,s,y)$ is defined in (\ref{a}) and
\begin{equation*}
W := \sum_{i=1}^d \left[ \mu_i G_{t-r}(x,v) \sigma_{ik}
(u(r,v)) + \mu_i a_i(k,r,v,t,x)\right].
\end{equation*}
From here on, the proof is divided into two cases.
\vskip 12pt

    \noindent\emph{Case 1.} In the first case, we assume that
    $|x-y|^2\le t-s$. Choose and fix an $\e\in(0,t-s)$.
    Then we may write
    \begin{equation*}
        \inf_{\|\xi\|=1} \xi^{\sf T}\gamma_Z\xi \ge
        \min\left( \inf_{ \|\xi\|=1\,,\|\mu\|\ge
        \e^{\eta/2}} J_2\,, \inf_{\|\xi\|=1\,,\|\mu\|\le\e^{\eta/2}}
        J_1\right).
    \end{equation*}
    We are going to prove that
    \begin{equation}\label{goal:J2J1}\begin{split}
        \inf_{ \|\xi\|=1\,,\|\mu\|\ge
            \e^{\eta/2}} J_2 &\ge \e^{\frac12+\eta} - Y_{1,\e},\\
        \inf_{\|\xi\|=1\,,\|\mu\|\le\e^{\eta/2}}
            J_1 & \ge \e^{1/2} - Y_{2,\e},
    \end{split}\end{equation}
    where, for all $q\ge 1$,
    \begin{equation}  \label{goal:J2J1:UB}
        \E \left[\left| Y_{1,\e}\right|^q\right] \le c_1(q) \e^q \; \; \;
            \text{and} \; \; \;
        \E \left[\left| Y_{2,\e}\right|^q\right] \le c_2(q) \e^{
            q(\frac12+\eta)}.
    \end{equation}
    We assume these, for the time being, and finish the proof
    of the proposition in Case $1$. Then we will return to proving \eqref{goal:J2J1} and (\ref{goal:J2J1:UB}).

We can combine \eqref{goal:J2J1} and (\ref{goal:J2J1:UB}) with Proposition~\ref{deter} to find that
\begin{equation*}\begin{split}
\E \left[ \left(\inf_{\|\xi\|=1} \xi^{\sf T}\gamma_Z\xi \right)^{-2pd}\right]
            & \le c
            (t-s)^{-2pd(\frac12+\eta)}\\
        &\le \tilde{c} \left[ (t-s)^{1/2} + |x-y|\right]^{-2pd
            (1+2\eta)},
    \end{split}\end{equation*}
    whence follows the proposition in the case that
    $|x-y|^2\le t-s$. Now we complete our proof
    of Case 1 by deriving \eqref{goal:J2J1} and (\ref{goal:J2J1:UB}).

    Let us begin with the term that involves $J_2$. Inequality (\ref{eq:23-2}) implies that
    \begin{equation*}
        \inf_{ \|\xi\|=1\,,\|\mu\|\ge
        \e^{\eta/2}} J_2 \ge  \hat{Y}_{1,\e} - Y_{1,\e},
    \end{equation*}
    where
    \begin{equation*}\begin{split}
        &\hat{Y}_{1,\e} := \frac23 \inf_{\|\mu\|\ge
            \e^{\eta/2}} \sum_{k=1}^d \int_{t-\e}^t
            dr\int_0^1dv\, \left( \sum_{i=1}^d
            \mu_i \sigma_{ik}(u(r,v))\right)^2
            G^2_{t-r}(x,v),\\
        &Y_{1,\e} := 2\sup_{\|\mu\|\ge
            \e^{\eta/2}} \sum_{k=1}^d\int_{t-\e}^t
            dr\int_0^1dv\, \left( \sum_{i=1}^d
            \mu_i  a_i(k,r,v,t,x)\right)^2.
    \end{split}\end{equation*}
In agreement with hypothesis {\bf P2}, and thanks to
Lemma \ref{(A.3)},
\begin{equation*}
        \hat{Y}_{1,\e} \ge c\inf_{\|\mu\|\ge\e^{\eta/2}}
        \|\mu\|^2\e^{1/2} \ge c\e^{\frac12+\eta}.
\end{equation*}
Next we apply Lemma~\ref{lem:7.7} below [with $s:=t$] to find  that $\E [|Y_{1,\e}|^q]\le c\e^q$. This proves the  bounds in \eqref{goal:J2J1} and (\ref{goal:J2J1:UB}) that concern $J_2$ and $Y_{1,\e}$.

    In order to derive the second bound in \eqref{goal:J2J1},
    we appeal to \eqref{eq:23-2} once more to find that
    \begin{equation*}
        \inf_{ \|\xi\|=1\,,\|\mu\|\le
        \e^{\eta/2}} J_1 \ge  \hat{Y}_{2,\e} - Y_{2,\e},
    \end{equation*}
    where
    \begin{equation*}
        \hat{Y}_{2,\e} := \frac23 \inf_{\|\mu\|\le
            \e^{\eta/2}} \sum_{k=1}^d \int_{s-\e}^s
            dr\int_0^1dv \left( \sum_{i=1}^d
            (\lambda_i-\mu_i) \sigma_{ik}(u(r,v))\right)^2
            G^2_{s-r}(y,v),
    \end{equation*}
    and
    \begin{equation*}
        Y_{2,\e} := 2\left( W_1+W_2+W_3 \right),
    \end{equation*}
    where
    \begin{equation*}\begin{split}
        &W_1 := \sup_{\|\mu\|\le
            \e^{\eta/2}} \sum_{k=1}^d\int_{s-\e}^s
            dr\int_0^1dv\, \left( \sum_{i=1}^d
            \mu_i  G_{t-r}(x,v)  \sigma_{ik}(u(r,v))\right)^2,\\
        &W_2 := \sup_{\|\xi\|=1} \sum_{k=1}^d\int_{s-\e}^s
            dr\int_0^1dv\, \left( \sum_{i=1}^d
            (\lambda_i-\mu_i) a_i(k,r,v,s,y) \right)^2,\\
        &W_3 := \sup_{\|\mu\|\le
            \e^{\eta/2}} \sum_{k=1}^d\int_{s-\e}^s
            dr\int_0^1dv\, \left( \sum_{i=1}^d
            \mu_i a_i(k,r,v,t,x)\right)^2.
    \end{split}\end{equation*}
    Hypothesis {\bf P2} and Lemma \ref{(A.3)}
    together yield
    \begin{equation}\label{Y_{2,e}}
        \hat{Y}_{2,\e} \ge c \e^{1/2}.
    \end{equation}
    Next, we apply the Cauchy--Schwarz inequality to find that
    \begin{equation*}\begin{split}
        \E\left[ \left| W_1 \right|^q\right] &\le
            \sup_{\|\mu\|\le\e^{\eta/2}}\|\mu\|^{2q} \times
            \E\left[ \left| \sum_{k=1}^d \int_{s-\e}^s
            dr \int_0^1 dv\, \sum_{i=1}^d
            \left( \sigma_{ik}(u(r,v))\right)^2 G^2_{t-r}(x,v)
            \right|^q\right]\\
        &\le c \e^{q\eta} \, \left| \sum_{k=1}^d \int_{s-\e}^s
            dr \int_0^1 dv\, G^2_{t-r}(x,v) \right|^q,
    \end{split}\end{equation*}
    thanks to hypothesis {\bf P1}. In light of this,
    Lemma \ref{int:G^2} implies that
    $\E\left[ \left| W_1 \right|^q\right] \le c \e^{q(\frac12+\eta)}.$

    In order to bound the $q$-th moment of $|W_2|$, we use
    the Cauchy--Schwarz inequality together with hypothesis {\bf P1}, and write
    \begin{equation*}\begin{split}
        \E \left[ \left| W_2\right|^q\right]
        &\le\sup_{\|\mu\|\le \e^{\eta/2}}
            \|\lambda-\mu\|^{2q}  \times \E\left[
            \left| \sum_{k=1}^d \int_{s-\e}^s dr
            \int_0^1 dv
            \sum_{i=1}^d a_i^2(k,r,v,s,y) 
            \right|^q\right]\\
        &\le C \E\left[
            \left| \sum_{k=1}^d \int_{s-\e}^s dr
            \int_0^1 dv  \sum_{i=1}^d a_i^2(k,r,v,s,y)
            \right|^q\right].
    \end{split}\end{equation*}
We apply Lemma \ref{lem:7.7} below [with $s:=t$] to find that $\E \left[ \left| W_2\right|^q\right] \le c \e^q.$

    Similarly, we find using Lemma \ref{lem:7.7} that
    \begin{equation*}\begin{split}
        \E\left[ \left| W_3\right|^q\right]
        &\le\sup_{\|\mu\|\le \e^{\eta/2}}
            \|\mu\|^{2q} \times \E\left[
            \left| \sum_{k=1}^d \int_{s-\e}^s dr
            \int_0^1 dv  \sum_{i=1}^d a_i^2(k,r,v,t,x)
            \right|^q\right]\\
        &\le c  \e^{q\eta} \, (t-s+\e)^{q/2}\e^{q/2}\\
        &\le c \e^{q(\frac12+\eta)}.
    \end{split}\end{equation*}

    The preceding bounds for $W_1$, $W_2$, and $W_3$
    prove, in conjunction, that
    $\E[|Y_{2,\e}|^q]\le c_2(q)\e^{q(\frac12+\eta)}$.
    This and \eqref{Y_{2,e}} together prove
    the bounds in \eqref{goal:J2J1} and (\ref{goal:J2J1:UB}) that concern $J_1$ and $Y_{2,\e}$, whence follows
    the result in Case 1.

    \vskip 12pt

    \noindent\emph{Case 2.} Now we work on the second case
    where $|x-y|^2\ge t-s\ge 0$. Let $\e>0$ be such that
    $(1+\alpha)\e^{1/2}<\frac{1}{2} |x-y|$, where
    $\alpha>0$ is large but fixed; its specific value will be decided
    on later. Then
    \begin{equation*}
        \xi^{\sf T} \gamma_Z \xi \ge I_1 + I_2 + I_3,
    \end{equation*}
    where
    \begin{equation*}\begin{split}
        I_1 &:= \sum_{k=1}^d \int_{s-\e}^s dr
            \int_{y-\sqrt{\e}}^{y+\sqrt{\e}} dv \,
            \left( \mathcal{S}_1 + \mathcal{S}_2\right)^2,\\
        I_2 &:=\sum_{k=1}^d \int_{s-\epsilon}^s dr
            \int_{x-\sqrt{\epsilon}}^{x+\sqrt{\epsilon}} dv\,
            \left( \mathcal{S}_1 + \mathcal{S}_2\right)^2,\\
        I_3 &:=\sum_{k=1}^d \int_{(t-\epsilon) \vee s}^t dr
            \int_{x-\sqrt{\epsilon}}^{x+\sqrt{\epsilon}} dv \,
            \mathcal{S}_2^2,
    \end{split}\end{equation*}
    and
    \begin{equation*}\begin{split}
        \mathcal{S}_1 &:=\sum_{i=1}^d (\lambda_i-\mu_i) \left[
            G_{s-r}(y,v) \sigma_{i,k}(u(r,v))+  a_i(k,r,v,s,y) \right],\\
        \mathcal{S}_2 &:=
            \sum_{i=1}^d \mu_i \left[ G_{t-r}(x,v) \sigma_{ik}(u(r,v)) +
            a_i(k,r,v,t,x) \right].
    \end{split}\end{equation*}

    From here on, Case 2 is divided into two further sub-cases.

    \vskip 12pt

\noindent\emph{Sub-Case A.}
Suppose, in addition, that $\e\ge t-s$. In this case, we are going to prove that
\begin{equation}\label{RCD:3}
\inf_{\|\xi\|=1} \xi^{\sf T}\gamma_Z\xi \ge
c \e^{1/2} - Z_{1,\e},
\end{equation}
where for all $q\ge 1$,
\begin{equation} \label{RCD:3a}
\E\left[ \left| Z_{1,\e} \right|^q\right]\le c(q)\e^{3q/4}.
\end{equation}
Apply (\ref{eq:23-2}) to find that
\begin{equation*}
I_1 \ge \frac23 \tilde{A}_1 - B_1^{(1)} - B_1^{(2)},
\end{equation*}
where
\begin{eqnarray} \nonumber
\tilde{A}_1 &:=& \sum_{k=1}^d \int_{s-\e}^s dr
\int_{y-\sqrt\e}^{y+\sqrt\e} dv \left( \sum_{i=1}^d
\left[ (\lambda_i-\mu_i) G_{s-r}(y,v)
+ \mu_i G_{t-r}(x,v)\right] \sigma_{ik}(u(r,v))\right)^2,\\ \label{B11}
B_1^{(1)} &:=& 4 \|\lambda-\mu\|^2\sum_{k=1}^d \int_{s-\e}^s dr\int_{y-\sqrt\e}^{
y+\sqrt\e} dv\, \sum_{i=1}^d a_i^2(k,r,v,s,y),\\ \label{B12}
B_1^{(2)}  &:=& 4 \|\mu\|^2
            \sum_{k=1}^d \int_{s-\e}^s dr\int_{y-\sqrt\e}^{
            y+\sqrt\e} dv\,
            \sum_{i=1}^d a_i^2(k,r,v,t,x).
\end{eqnarray}
Using the inequality 
    \begin{equation} \label{qua}
    (a-b)^2 \geq \frac{2}{3} a^2 -2 ab,
    \end{equation}
    we see that
    \begin{equation*}
        \tilde{A}_1 \ge \frac23 A_1 - B^{(3)}_1,
    \end{equation*}
    where
    \begin{equation*}\begin{split}
        A_1 &:=\sum_{k=1}^d \int_{s-\e}^s dr
            \int_{y-\sqrt{\e}}^{y+\sqrt{\e}} dv \left(
            \sum_{i=1}^d (\lambda_i-\mu_i) G_{s-r}(y,v)
            \sigma_{ik}(u(r,v))\right)^2,\\
        B_1^{(3)} &:= 2 \sum_{k=1}^d \int_{s-\epsilon}^s dr
            \int_{y-\sqrt{\e}}^{y+\sqrt{\e}} dv \left(
            \sum_{i=1}^d  (\lambda_i-\mu_i) G_{s-r}(y,v) \sigma_{ik}(u(r,v))
            \right) \\
        &\hskip1.45in \times\left(\sum_{i=1}^d \mu_i G_{t-r}(x,v)
            \sigma_{ik}(u(r,v))  \right).
    \end{split}\end{equation*}
    We can combine terms to find that
    \begin{equation*}
        I_1 \ge \frac49 A_1 - B_1^{(1)} - B_1^{(2)}
        - B_1^{(3)}.
    \end{equation*}

    We proceed in like manner for $I_2$, but obtain
    slightly sharper
    estimates as follows. Owing to \eqref{qua},
    \begin{equation*}
        I_2 \ge \frac{2}{3} A_2 - B^{(1)}_2- B^{(2)}_2- B^{(3)}_2,
    \end{equation*}
    where
    \begin{equation*}\begin{split}
        A_2 &:=  \sum_{k=1}^d \int_{s-\e}^s dr
            \int_{x-\sqrt{\e}}^{x+\sqrt{\e}} dv \left(
            \sum_{i=1}^d \mu_i G_{t-r}(x,v) \sigma_{ik}(u(r,v)) \right)^2,\\
        B_2^{(1)} &:= 2\sum_{k=1}^d \int_{s-\e}^s  dr\int_{x-\sqrt\e}^{
        x+\sqrt\e}  dv\left( \sum_{i=1}^d
        \mu_i G_{t-r}(x,v) \sigma_{ik}
        (u(r,v)) \right)\\
    &\hskip1.74in \times\left(\sum_{i=1}^d \mu_i
        a_i(k,r,v,t,x) \right),\\
    B_2^{(2)} &:= 2\sum_{k=1}^d \int_{s-\e}^s  dr\int_{x-\sqrt\e}^{
        x+\sqrt\e}  dv\left( \sum_{i=1}^d
        \mu_i G_{t-r}(x,v) \sigma_{ik}
        (u(r,v)) \right)\\
    &\hskip1.74in \times\left(\sum_{i=1}^d (\lambda_i-\mu_i)
        a_i(k,r,v,s,y) \right), \\
        B_2^{(3)} &:= 2\sum_{k=1}^d \int_{s-\epsilon}^s dr
            \int_{x-\sqrt{\e}}^{x+\sqrt{\e}} dv \left(
            \sum_{i=1}^d  \mu_i G_{t-r}(x,v) \sigma_{ik}(u(r,v))
            \right) \\
        &\hskip1.45in \times\left(\sum_{i=1}^d (\lambda_i-\mu_i) G_{s-r}(y,v)
            \sigma_{ik}(u(r,v))  \right).
    \end{split}\end{equation*}
    Finally, we appeal to \eqref{eq:23-2} to find that
    \begin{equation*}
        I_3 \ge \frac23 A_3-B_3,
    \end{equation*}
where
\begin{eqnarray}\nonumber
        A_3 &:=&  \sum_{k=1}^d \int_{(t-\e) \vee s}^t dr
            \int_{x-\sqrt{\e}}^{x+\sqrt{\epsilon}} dv \left(
            \sum_{i=1}^d \mu_i G_{t-r}(x,v) \sigma_{ik}(u(r,v)) \right)^2,\\
        B_3 &:=&2 \sum_{k=1}^d \int_{(t-\e) \vee s}^t dr
            \int_{x-\sqrt{\e}}^{x+\sqrt{\e}} dv \left(
            \sum_{i=1}^d \mu_i a_i(k,r,v,t,x)   \right)^2.
            \label{B3}
\end{eqnarray}

    By hypothesis {\bf P2},
    \begin{equation*}\begin{split}
        A_1 + A_2 + A_3 &\ge \rho^2\left(
            \|\lambda-\mu\|^2\int_{s-\e}^s dr
            \int_{y-\sqrt\e}^{y+\sqrt\e} dv\,
            G^2_{s-r}(y,v) \right.\\
        &\hskip.5in +\|\mu\|^2\int_{s-\e}^s dr
            \int_{x-\sqrt\e}^{x+\sqrt\e} dv\,
            G^2_{t-r}(x,v) \\
        &\hskip.5in + \left. \|\mu\|^2 \int_s^t dr\int_{x-\sqrt\e}^{
            x+\sqrt\e} dv\, G^2_{t-r}(x,v) \right).
    \end{split}\end{equation*}
    Note that we have used the defining assumption of
    Sub-Case A, namely, that $\e\ge t-s$. Next, we group the last two integrals and apply
    Lemma \ref{(A.3)} to find that
    \begin{equation}\label{subcaseA:A1A2A3}\begin{split}
        A_1 + A_2 + A_3
        &\ge c \left(
            \|\lambda-\mu\|^2\e^{1/2} +
            \|\mu\|^2\int_{t-\e}^t
            dr\int_{x-\sqrt\e}^{x+\sqrt\e}
            dv\, G^2_{t-r}(x,v) \right)\\
        &\ge c\left( \|\lambda-\mu\|^2 +\|\mu\|^2\right)\e^{1/2}\\
        &\ge c \e^{1/2}.
    \end{split}\end{equation}

    We are aiming for \eqref{RCD:3}, and propose to bound the
    absolute moments
    of $B_1^{(i)}$, $B_2^{(i)}$, $i=1,2,3$ and $B_3$, separately.
    According to Lemma \ref{lem:7.7} below with $s=t$,
    \begin{equation}\label{subcaseA:B_3}
        \E \left[ \sup_{\|\xi\|=1} \left| B_3 \right|^q\right]
        \le c(q) \e^q.
    \end{equation}

    Next we bound the absolute moments of $B^{(i)}_1$, $i=1,2,3$.
    Using hypothesis {\bf P1}
    and Lemma \ref{lem:7.7}, with $t=s$,
    we find that for all $q\ge 1$,
\begin{equation}\label{subcaseA:B_1^1}
        \E \left[ \sup_{\| \xi\|=1}
        \left| B_1^{(1)} \right|^q \right]\le c\e^q.
\end{equation}
In the same way, we see that
\begin{equation}\label{EB3}
        \E \left[ \sup_{\| \xi
        \| =1} \left| B_1^{(2)} \right|^q \right]
        \leq c  (t-s+\e)^{q/2}\e^{q/2}.
\end{equation}
    We are in the sub-case A where $t-s\le\e$. Therefrom,
    we obtain the following:
    \begin{equation}\label{subcaseA:B_1^2}
        \E \left[ \sup_{\| \xi
        \| =1} \left| B_1^{(2)} \right|^q \right]
        \le c \e^q.
    \end{equation}

    Finally, we turn to bounding the absolute moments of
    $B_1^{(3)}$. Hypothesis {\bf P1} assures us
    that
    \begin{equation*}\begin{split}
        \left| B_1^{(3)} \right| &\leq c
            \int_{s-\e}^s dr \,
            \int_{y-\sqrt{\e}}^{y+\sqrt{\e}} dv \, G_{s-r}(y,v)
            G_{t-r}(x,v)\\
        &\leq c \int_{s-\e}^s dr \,
            \int_0^1 dv \, G_{s-r}(y,v)
            G_{t-r}(x,v)\\
        &= c\int_{s-\e}^s dr\, G_{t+s-2r}(x,y),
    \end{split}\end{equation*}
    thanks to the semi-group property \textsc{Walsh} \cite[(3.6)]{Walsh:86}
    (see (\ref{square}) below).
    This and Lemma \ref{(A.1)}
    together prove that
    \begin{equation*}\begin{split}
        \left| B_1^{(3)} \right|  &\leq  c  \int_0^\e
            \frac{du}{\sqrt{t-s+2u}} \exp\left(
            -\frac{| x-y |^2}{2(t-s+2u)} \right) \\
        &\leq  c \int_0^{\e}
            \frac{du}{\sqrt{t-s+2u}} \exp\left(-
            \frac{\alpha^2 \e}{2(t-s+2u)}\right),
    \end{split}\end{equation*}
    since $| x-y | \geq 2 (1+\alpha) \e^{1/2} \geq \alpha \e^{1/2}$.
    Now we can change variables [$z:=2(t-s+2u)/(\alpha^2\e)$],
    and use the bounds
    $0\le t-s\le \e$ to find that
    \begin{equation}\label{subcaseA:B_1^3}\begin{split}
        \left| B_1^{(3)} \right|  &\leq  c \e^{1/2} \Psi(\alpha),
        \quad\text{where}\quad
        \Psi(\alpha) := \alpha
        \int_0^{6/\alpha^2}z^{-1/2} e^{-1/z}\, dz.
    \end{split}\end{equation}
    Following exactly in the same way, we see that
    \begin{equation}\label{subcaseA:B_2^3}
\left| B_2^{(3)} \right|  \leq  c \e^{1/2} \Psi(\alpha).
    \end{equation}

    We can combine \eqref{subcaseA:B_1^1},
    \eqref{subcaseA:B_1^2} as follows:
     \begin{equation}\label{subcaseA:B_1}
         \E \left[ \sup_{\|\xi\|=1} \left| B_1^{(1)}
         + B_1^{(2)}  \right|^q \right] \le
         c(q) \e^q.
     \end{equation}
     On the other hand, we will see in Lemmas \ref{lem:B_2^1} and \ref{lem:B_2^2} below that
     \begin{equation}\label{subcaseA:B_2}
         \E \left[ \sup_{\| \xi \| =1} \left| B_2^{(1)}
         + B_2^{(2)} \right|^q \right]
        \le c(q)\e^{3q/4}.
     \end{equation}
     Now, by (\ref{subcaseA:A1A2A3}), (\ref{subcaseA:B_3}), (\ref{subcaseA:B_1^3}), (\ref{subcaseA:B_2^3}),
     (\ref{subcaseA:B_1}) and (\ref{subcaseA:B_2}),
     \begin{equation*}\begin{split}
        &\inf_{\|\xi\|=1} \xi^{\sf T}\gamma_Z \xi \\
        &\ge  \left( \frac13 A_1 + A_2 + A_3\right)
            -B_1^{(3)} - B_2^{(3)} - \left(
             B_1^{(1)} + B_1^{(2)} + B_2^{(1)}
             + B_2^{(2)} +B_3\right)\\
        &\ge c_1\e^{1/2} - c_2\Psi(\alpha)\e^{1/2}
            - Z_{1,\e},
     \end{split}\end{equation*}
where
     $Z_{1,\e}:=B_1^{(1)} + B_1^{(2)} + B_2^{(1)}
     + B_2^{(2)} +B_3$
     satisfies $\E [|Z_{1,\e}|^q]\le c_1(q)\e^{3q/4}$.
     Because $\lim_{\nu\to\infty}\Psi(\nu)=0$,
     we can choose and fix $\alpha$ so large
     that $c_2\Psi(\alpha)\le c_1/4$ for the $c_1$
     and $c_2$ of the preceding displayed equation. This yields,
     \begin{equation}\begin{split} \label{n50}
        \inf_{\|\xi\|=1} \xi^{\sf T}\gamma_Z \xi
        \ge c\e^{1/2} - Z_{1,\e},
     \end{split}\end{equation}
as in (\ref{RCD:3}) and (\ref{RCD:3a}).
\vskip 12pt
\noindent\emph{Sub-Case B.} In this final (sub-) case we  suppose that
     $\e\le t-s\le |x-y|^2.$ Choose and fix $0<\e<t-s$. During the course of
     our proof of Case 1, we established
     the following:
\begin{equation*}
         \inf_{\|\xi\|=1}
         \xi^{\sf T}\gamma_Z\xi\ge \min\left(
         c\e^{\frac12+\eta} - Y_{1,\e} \,,
         c\e^{1/2} - Y_{2,\e}\right),
\end{equation*}
where
\begin{equation*}
         \E\left[ \left| Y_{1,\e}\right|^q\right] \le c(q)\e^{q}
         \quad\text{and}\quad
         \E\left[ \left| Y_{2,\e}\right|^q\right] \le c(q)\e^{q(\frac12+\eta)}.
\end{equation*}
     See \eqref{goal:J2J1} and \eqref{goal:J2J1:UB}.
     Consider this in conjunction with \eqref{RCD:3} to find that
     for all $0<\e<\frac14(1+\alpha)^{-2}|x-y|^2$,
\begin{equation*}
         \inf_{\|\xi\|=1} \xi^{\sf T}\gamma_Z\xi
         \ge \min\left( c \e^{\frac12+\eta} - Y_{1,\e} ~,~
         c \e^{1/2} - Y_{2,\e}- Z_{1,\e}\mathbf{1}_{
         \{t-s<\e\}}\right).
\end{equation*}
Because of this and \eqref{RCD:3a}, Proposition \ref{deter} implies that
\begin{equation*}\begin{split}
        \E\left[ \left( \inf_{\|\xi\|=1} \xi^{\sf T}\gamma_Z\xi
             \right)^{-2pd}\right] & \le c|x-y|^{2(-2dp)(\frac12+\eta)}\\
        &\le c( |t-s|^{1/2} + |x-y|)^{-2dp(1+2\eta)}.
\end{split}\end{equation*}
This concludes the proof of Proposition \ref{pr:small-eigen:1}(a).

     If $t=s$, then Sub--Case B does not arise, and so we get directly from (\ref{n50}) and Proposition \ref{deter} that
\begin{equation*}
\E\left[ \left( \inf_{\|\xi\|=1} \xi^{\sf T}\gamma_Z\xi
             \right)^{-2pd}\right] \le c|x-y|^{-2dp}.
\end{equation*}
This proves (b) and concludes the proof of Proposition~\ref{pr:small-eigen:1}.
\end{proof}

\begin{rema}
If $\sigma$ and $b$ are constant, then $a_i=0$, so $\eta$ can be taken to be $0$.
This gives the correct upper bound in the Gaussian case, which shows that the method of proof of Proposition \ref{pr:small-eigen:1} is rather tight.
\end{rema}

We finally prove three results that we have used in the proof of Proposition~\ref{pr:small-eigen:1}.
\begin{lem}\label{lem:7.7}
Assume ${\bf P1}$.
    For all $T>0$ and $q\ge1$, there exists a constant
    $c=c(q,T)\in(0,\infty)$ such that for every
    $0<\e\le s\le t\le T$ and $x\in[0,1]$,
    \begin{equation*}
        \E\left[ \left( \sum_{k=1}^d \int_{s-\e}^s
        dr \int_0^1 dv\, \sum_{i=1}^d
        a_i^2(k,r,v,t,x) \right)^q \right]
        \le c(t-s+\e)^{q/2}\e^{q/2}.
    \end{equation*}
\end{lem}

\begin{proof}
    Define
    \begin{equation*}
        A := \sum_{k=1}^d \int_{s-\e}^s
        dr \int_0^1 dv\, \sum_{i=1}^d
        a_i^2(k,r,v,t,x).
    \end{equation*}
Use (\ref{a}) to write
    \begin{equation*}
        \E\left[ |A|^q \right]  \le c\left( \E\left[
            |A_1|^q\right] +\E\left[ |A_2|^q\right] \right),
            \end{equation*}
        where
        \begin{equation*}
        A_1 := \sum_{i,j,k=1}^d
            \int_{s-\e}^s dr \int_0^1 dv \left| \int_r^t \int_0^1
            G_{t-\theta}(x,\eta) D^{(k)}_{r,v}
            \left( \sigma_{ij}(u(\theta,
            \eta))\right)\, W^j(d\theta,d\eta)\right|^2,
            \end{equation*}
            and
            \begin{equation*}
        A_2 := \sum_{i,k=1}^d
            \int_{s-\e}^s
            dr \int_0^1 dv \left| \int_r^t d\theta
            \int_0^1 d\eta\, G_{t-\theta}(x,\eta) D^{(k)}_{r,v}
            \left( b_i (u(\theta,
            \eta))\right)\right|^2.
    \end{equation*}
    We bound the $q$-th moment of $A_1$ and $A_2$
    separately.

    As regards $A_1$, we apply the Burkholder inequality
    for Hilbert-space-valued martingales (Lemma \ref{valuedm})
    to find that
    \begin{equation}\label{eq:A_1:bd}
        \E\left[ |A_1|^q\right] \le c
        \sum_{i,j,k=1}^d
        \E\left[ \left| \int_{s-\e}^t
        d\theta \int_0^1 d\eta \int_{s-\e}^s
        dr \int_0^1 dv \, \Theta^2 \right|^q \right],
    \end{equation}
    where
    \begin{equation*}\begin{split}
        \Theta &:= \mathbf{1}_{\{\theta>r\}}
            G_{t-\theta}(x,\eta) \left|
            D^{(k)}_{r,v}\left( \sigma_{ij}(u(\theta,\eta))\right)
            \right|\\
        &\le c \mathbf{1}_{\{\theta>r\}} G_{t-\theta}
            (x,\eta) \left| \sum_{l=1}^d D^{(k)}_{r,v}
            \left( u_l (\theta,\eta) \right) \right|,
    \end{split}\end{equation*}
thanks to hypothesis {\bf P1}. Thus,
\begin{equation*}
        \E\left[ |A_1|^q\right] \le c\sum_{k=1}^d
            \E\Bigg[ \Bigg| \int_{s-\e}^t
            d\theta \int_0^1 d\eta \,
            G^2_{t-\theta}(x,\eta) \int_{s-\e}^{s\wedge \theta}
            dr \int_0^1 dv \, \left(
            \sum_{l=1}^d D^{(k)}_{r,v}
            \left( u_l(\theta,\eta)\right) \right)^2 \Bigg|^q \Bigg].
\end{equation*}
    We apply H\"older's inequality with respect to the measure
    $G^2_{t-\theta}(x,\eta)\, d\theta\, d\eta$ to find that
    \begin{equation}\label{A10}\begin{split}
        \E\left[ |A_1|^q\right] &\le c \left(
            \int_{s-\e}^t d\theta\int_0^1 d\eta\,
            G^2_{t-\theta}(x,\eta) \right)^{q-1} \\
        &\qquad\ \qquad \times\int_{s-\e}^t
            d\theta \int_0^1 d\eta \,
            G^2_{t-\theta}(x,\eta) \sum_{k=1}^d
            \E\left[ \left| \int_{s-\e}^{s\wedge\theta}
            dr\int_0^1 dv\,  \Upsilon^2 \right|^q \right],
    \end{split}\end{equation}
where $\Upsilon := \sum_{l=1}^d D^{(k)}_{r,v} ( u_l(\theta,\eta))$. Lemma \ref{(A.5)} assures us that
    \begin{equation}\label{A11}
        \left( \int_{s-\e}^t d\theta\int_0^1 d\eta\,
        G^2_{t-\theta}(x,\eta) \right)^{q-1} \le
        c (t-s+\e)^{(q-1)/2}.
    \end{equation}
    On the other hand, Lemma \ref{morien} implies that
    \begin{equation*}\begin{split}
        \sum_{k=1}^d \E\left[ \left| \int_{s-\e}^{s\wedge\theta}
        dr\int_0^1 dv\,  \Upsilon^2 \right|^q \right]
        \le c\e^{q/2},
    \end{split}\end{equation*}
    where $c\in(0,\infty)$ does not depend on
    $(\theta,\eta,s,t,\e,x)$. Consequently,
    \begin{equation}\label{A12}\begin{split}
        &\int_{s-\e}^t d\theta \int_0^1 d\eta \,
            G^2_{t-\theta}(x,\eta)
            \sum_{k=1}^d
            \E\left[ \left| \int_{s-\e}^{s\wedge\theta}
            dr\int_0^1 dv\,  \Upsilon^2 \right|^q \right]\\
        &\le c\e^{q/2}\int_{s-\e}^t
            d\theta \int_0^1 d\eta \,
            G^2_{t-\theta}(x,\eta)\\
        &\le c\e^{q/2} (t-s+\e)^{1/2}.
    \end{split}\end{equation}
    Equations \eqref{A10}, \eqref{A11}, and \eqref{A12}
    together imply that
    \begin{equation}\label{A1:UB}
        \E\left[ |A_1|^q\right] \le c (t-s+\e)^{q/2}
        \e^{q/2}.
    \end{equation}
    This is the desired bound for the $q$-th moment
    of $A_1$. Next we derive a similar bound for
    $A_2$. This will finish the proof.
    By the Cauchy--Schwarz
    inequality
    \begin{equation*}
        \E\left[ |A_2|^q\right]
        \le c(t-s+\e)^q
             \sum_{i,k=1}^d
            \E\left[ \left| \int_{s-\e}^s
            dr \int_0^1 dv \int_r^t d\theta
            \int_0^1 d\eta\, \Phi^2 \right|^q \right],
    \end{equation*}
    where $\Phi:=G_{t-\theta}(x,\eta) |D^{(k)}_{r,v}\left(b_i(u(\theta,\eta))\right)|$.
    From here on,
    the $q$-th moment of $A_2$ is estimated as that
    of $A_1$ was; cf.\ \eqref{eq:A_1:bd}, and this yields
    $\E[|A_2|^q]\le c(t-s+\e)^{3q/2}\e^{q/2}$.
    This completes the proof.
\end{proof}

\begin{rema}
    It is possible to prove that $\E[|A_1|]$ is
    \emph{at least} a constant times
    $(t-s+\e)^{1/2} \e^{1/2}$. In this sense, the preceding
    result is not improvable.
\end{rema}

\begin{lem}\label{lem:B_2^1}
Assume ${\bf P1}$. Fix $T>0$ and $q \geq 1$. Then there exists $c=c(q,T)$ such that for all $x\in[0,1]$, $0\le s\le t \leq T$,
    and $\e\in(0,1)$,
    \begin{equation*}
        \E\left[ \sup_{\mu\in\R^d:\,
        \|\mu\|\le 1}\left| B_2^{(1)}\right|^q
        \right]\le c \e^{3q/4}.
    \end{equation*}
\end{lem}

\begin{proof}
    Define
    \begin{equation*}
        \hat{B}_2^{(1)}(k,i) := \int_{s-\e}^s dr
        \int_{x-\sqrt\e}^{x+\sqrt\e} dv\,
        G_{t-r}(x,v) \left| a_i(k,r,v,t,x) \right|.
    \end{equation*}
    Then, by the Cauchy--Schwarz inequality,
    \begin{equation} \label{hatB21:1}
        \E\left[ \left| \hat{B}_2^{(1)}(k,i) \right|^q\right]
        \le J_1 J_2,
    \end{equation}
    where
    \begin{equation*}\begin{split}
        &J_1 := \left( \int_{s-\e}^s dr \int_{x-\sqrt\e}^{x+\sqrt\e}
            dv\, G_{t-r}^2 (x,v)\right)^{q/2},\\
        &J_2 := \E\left[ \left( \int_{s-\e}^s dr
            \int_{x-\sqrt\e}^{x+\sqrt\e} dv\,
            a_i^2(k,r,v,t,x) \right)^{q/2} \right].
    \end{split}\end{equation*}
    On one hand, according to Lemma \ref{int:G^2},
    \begin{equation}\label{hatB21:2a}
        J_1\le c\frac{\e^{q/2}}{(t-s+\e)^{q/4}}.
    \end{equation}
    On the other hand, Lemma \ref{lem:7.7} assures us that
    \begin{equation}\label{hatB21:3}
        J_2 \le c' (t-s+\e)^{q/4}\e^{q/4}.
    \end{equation}
    By combining \eqref{hatB21:1},
    \eqref{hatB21:2a}, and \eqref{hatB21:3},
    we find that $\E[| \hat{B}_2^{(1)} (k,i) |^q] \le c \e^{3q/4}$
    for a constant $c\in(0,\infty)$ that does not depend
    on $\e$. By hypothesis {\bf P1},
    \begin{equation}\label{hathat}
        \E\left[\sup_{\mu\in\R^d:\ \|\mu\|\le 1}
        \left| B_2^{(1)}\right|^q
        \right] \le c\sum_{k=1}^d \E\left[ \sum_{i=1}^d\left|
        \hat{B}_2^{(1)}(k,i)\right|^q
         \right]\le c\e^{3q/4},
    \end{equation}
    as asserted.
\end{proof}

\begin{lem}\label{lem:B_2^2}
Assume ${\bf P1}$. Fix $T>0$ and $q \ge 1$. Then there exists $c=c(q,T)$ such that for any $x\in[0,1]$, $0\le s\le t \le T$,
    and $\e\in(0,1)$,
    \begin{equation*}
        \E\left[ \sup_{\xi=(\lambda,\mu)\in\R^{2d}:\,
        \|\xi\|= 1}\left| B_2^{(2)}\right|^q
        \right]\le c \e^{3q/4}.
    \end{equation*}
\end{lem}

\begin{proof}
    Define
    \begin{equation*}
        \hat{B}_2^{(2)}(k,i) := \int_{s-\e}^s dr
        \int_{x-\sqrt\e}^{x+\sqrt\e} dv\,
        G_{t-r}(x,v) \left| a_i(k,r,v,s,y) \right|.
    \end{equation*}
    Then, by the Cauchy--Schwarz inequality,
    \begin{equation*}
        \E\left[ \left| \hat{B}_2^{(2)}(k,i) \right|^q\right]
        \le J_1 J_2,
    \end{equation*}
    where
    \begin{equation*}\label{hatB21:2}\begin{split}
        &J_1  := \left( \int_{s-\e}^s dr \int_{x-\sqrt\e}^{x+\sqrt\e}
            dv\, G_{t-r}^2 (x,v)\right)^{q/2},\\
        &J_2 :=\E\left[ \left( \int_{s-\e}^s dr
            \int_{x-\sqrt\e}^{x+\sqrt\e} dv\,
            a_i^2(k,r,v,s,y) \right)^{q/2} \right].
    \end{split}\end{equation*}
According to (\ref{hatB21:2a}),
    \begin{equation*}
        J_1 \le c\frac{\e^{q/2}}{(t-s+\e)^{q/4}}  \le c'\e^{q/4}.
    \end{equation*}
    On the other hand, Lemma \ref{lem:7.7}, with $t=s$, assures us that
    $J_2\le c''\e^{q/2}$.  It follows that the $q$-th
    absolute moment of
    $\hat{B}^{(2)}_2(k,i)$ is at most $c\e^{3q/4}$. An appeal
    to the triangle inequality finishes the proof; see \eqref{hathat}
    where a similar argument was worked out in detail.
\end{proof}

\subsubsection{Large Eigenvalues}

\begin{prop} \label{large}
Assume ${\bf P1}$ and ${\bf P2}$.
Fix $T>0$ and $p>1$. Then there exists $C=C(p,T)$ such that for all $0\leq s < t \leq T$ with $t-s < \frac{1}{2}$,
$x,y  \in (0,1)$, $x \neq y$,
\begin{equation*}
\E \biggl[\biggl( \prod_{i=1}^{d} (\xi^i)^{T}
\gamma_Z \xi^i \biggr)^{-p} \biggr] \leq C,
\end{equation*}
where $\xi^1,...,\xi^d$ are the vectors from Lemma \ref{pertur}.
\end{prop}

\begin{proof}
Let $\xi^1,\ldots,\xi^d$, written
as in (\ref{base}), be
such that $\alpha_1 \geq \alpha_0,\ldots,\alpha_d
\geq \alpha_0$ for some
$\alpha_0>0$. In order to simplify the exposition,
we assume that $0 <
\alpha=\alpha_1= \cdots = \alpha_d \leq 1$,
since the general
case follows along the same lines.
Let $0< \e < s\leq t$.
As in the proof of Proposition \ref{pr:small-eigen:1}, we note first that
$\prod_{i=1}^{d} (\xi^i)^{\sf T}  \gamma_Z
\xi^i$ is bounded below by
\begin{equation}\label{large1}\begin{split}
    &\sum_{k=1}^d
        \int_{s-\epsilon}^s dr \int_0^1 dv \Bigg( \sum_{i=1}^d
        \biggl[ \bigg(\alpha
        \tilde{\lambda}^i G_{s-r}(y,v)\\
    &\hskip.5in +\tilde{\mu}^i \sqrt{1-\alpha^2}
        \left( G_{t-r}(x,v)-G_{s-r}(y,v) \right)\bigg)
        \sigma_{ik}(u(r,v))\\
    &\hskip.5in + \alpha \tilde{\lambda}^i
        a_i(k,r,v,s,y)\\
    &\hskip.5in +\tilde{\mu}^i \sqrt{1-\alpha^2}
        \left( a_i(k,r,v,t,x)-a_i(k,r,v,s,y) \right) \biggr] \Bigg)^2 \\
    &\quad + \sum_{k=1}^d \int_{s \vee (t-\epsilon)}^t
        dr \int_0^1 dv \Bigg( \sum_{i=1}^d \biggl[ \tilde{\mu}^i
        \sqrt{1-\alpha^2}\ G_{t-r}(x,v) \sigma_{ik}(u(r,v)) \\
    & \hskip.5in  + \tilde{\mu}^i \sqrt{1-\alpha^2}\
        a_i(k,r,v,t,x) \biggr] \Bigg)^2.
\end{split}\end{equation}
We intend to use Proposition \ref{deter} with $\varepsilon_0 >0$ fixed, so we seek lower bounds for this expression for $0 < \varepsilon < \varepsilon_0$.
\vskip 12pt

\noindent{\bf Case 1.} $t-s \leq \e$.
Then, by (\ref{eq:23-2}), the expression in (\ref{large1}) is bounded below by
\begin{equation*}
    \frac{2}{3} (f_1(s,t,\epsilon,\alpha,
    \tilde \lambda,\tilde  \mu,x,y)+f_2(s,t,\epsilon,\alpha,\tilde \lambda,\tilde \mu,x,y)) - 2 I_\e,
\end{equation*}
where, from hypothesis {\bf P2},
\begin{align}\label{pq1}
    f_1  &\geq c \rho^2 \int_{s-\epsilon}^s dr \int_0^1 dv
        \left\| \alpha \tilde{\lambda} G_{s-r}(y,v)+
        \sqrt{1-\alpha^2}\ \tilde{\mu}
        (G_{t-r}(x,v)-G_{s-r}(y,v)) \right\|^2, \\
    \label{pq2}
    f_2 & \geq c \rho^2 \int_{s \vee (t-\epsilon)}^t dr \int_0^1 dv
        \, \bigg\| \tilde{\mu} \sqrt{1-\alpha^2}\ G_{t-r}(x,v)
        \bigg\|^2,
\end{align}
and $I_\e = I_{1,\e} +  I_{2,\e} +  I_{3,\e}$, where
\begin{equation*}\begin{split}
    I_{1,\e}  &:= \sum_{k=1}^d \int_{s-\e}^s dr \int_0^1 dv\, \left(\sum_{i=1}^d
        \left[\alpha \tilde\lambda^i - \tilde\mu_i \sqrt{1 - \alpha^2}
        \right] a_i(k,r,v,s,y)
        \right)^2, \\
    I_{2,\e} &:= \sum_{k=1}^d \int_{s-\e}^s dr \int_0^1 dv\, \left(\sum_{i=1}^d
        \tilde\mu_i \sqrt{1 - \alpha^2}\ a_i(k,r,v,t,x)   \right)^2,\\
    I_{3,\e} &:= \sum_{k=1}^d \int_{t-\e}^t dr \int_0^1 dv\, \left(\sum_{i=1}^d
        \tilde\mu_i \sqrt{1 - \alpha^2}\ a_i(k,r,v,t,x) \right)^2.
\end{split}\end{equation*}
There are obvious similarities between the terms $I_{1,\e}$ and $B_1^{(1)}$ in (\ref{B11}). Thus, we apply the same method that was used to bound $\E[|B_1^{(1)}|^q]$ to  deduce that $\E[|I_{1,\e}|^q] \leq c(q) \e^{q}$.
 Since $I_{2,\e}$ is similar to $B_1^{(2)}$ from (\ref{B12}) and $t-s \leq \e$, we see using (\ref{subcaseA:B_1^2}) that $\E[|I_{2,\e}|^q] \leq c(q) \e^{q}$. Finally, using the similarity between $I_{3,\e}$ and $B_3$ in (\ref{B3}), we see that $\E[|I_{3,\e}|^q] \leq c(q) \e^q$.

We claim that there exists $\alpha_0>0$, $\epsilon_0>0$ and $c_0>0$ such that
\begin{equation} \label{minoration}\begin{split}
    f_1+f_2 \geq c_0 \sqrt{\epsilon}
    \; \text{ for all } \alpha \in [\alpha_0,1], \, \epsilon \in(0, \epsilon_0],
    \, s,t \in [1,2], \, x,y \in [0,1].
\end{split}\end{equation}
This will imply in particular that for $\e \geq t-s$,
\begin{equation*}
    \prod_{i=1}^d \left( \xi^i \right)^{\sf T}
    \gamma_Z  \xi^i \geq c_0 \e^{1/2} - 2
    I_\e,
 \end{equation*}
 where $\E[| I_\e|^q] \leq c(q) \e^{q}$.

In order to prove (\ref{minoration}), first define
\begin{equation*}
    p_t(x,y) := (4\pi t)^{-1/2}
    e^{-(x-y)^2/(4t)}.
\end{equation*}
In addition, let $g_1(s,t,\epsilon,\alpha,\tilde \lambda,\tilde \mu, x,y)$ and $g_2(s,t,\epsilon,\alpha,\tilde\lambda,\tilde \mu, x,y)$
be defined by the same expressions as the right-hand
sides of (\ref{pq1}) and (\ref{pq2}), but with
$G_{s-r}(x,v)$ replaced by $p_{s-r}(x-v)$, and $\int_0^1$
replaced by $\int_{-\infty}^{+\infty}$.

Observe that $g_1 \geq 0$, $g_2 \geq 0$, and if $g_1=0$,
then for all $v \in \R^d$,
\begin{equation} \label{double}
    \left\| \alpha \,
    p_{s-r}(y-v) \tilde{\lambda} +
    \sqrt{1-\alpha^2}\ (p_{t-r}(x-v) - p_{s-r}(y-v))  \tilde{\mu} \right\|=0.
\end{equation}
If, in addition, $\tilde{\lambda}=\tilde{\mu}$, then we get that for all $v \in \R^d$,
\begin{equation*}
    \left( \alpha-\sqrt{1-\alpha^2}
    \right) p_{s-r}(y-v) + \sqrt{1-\alpha^2} p_{t-r}(x-v) =0.
\end{equation*}
We take Fourier transforms to deduce from this that for all $\xi \in \R^d$,
\begin{equation*}
    \left( \alpha-\sqrt{1-\alpha^2} \right)
    e^{i \xi y} = -\sqrt{1-\alpha^2} e^{i \xi x} e^{(s-t) \xi^2}.
\end{equation*}
If $x=y$, then it follows that $s=t$ and $\alpha-\sqrt{1-\alpha^2}=-\sqrt{1-\alpha^2}$.
Hence, if $\alpha \neq 0$, $x=y$ and $\tilde{\lambda}= \tilde{\mu}$, then $g_1>0$.
We shall make use of this observation shortly.

Because $\| \tilde{\lambda} \| =\| \tilde{\mu} \|=1$,
$f_1$ is bounded below by
\begin{equation*}\begin{split}
    &c \rho^2 \int_{s-\epsilon}^s dr \int_0^1 dv
        \biggl(  \alpha^2 G_{s-r}^2(y,v)+ \left( 1-\alpha^2\right)
        \left(G_{t-r}(x,v)-G_{s-r}(y,v) \right)^2\\
    & \qquad +2 \alpha \sqrt{1-\alpha^2}
        G_{s-r}(y,v) (G_{t-r}(x,v)-G_{s-r}(y,v))
        (\tilde{\lambda} \cdot \tilde{\mu} ) \biggr) \\
    &=  c \rho^2 \int_{s-\epsilon}^s dr \int_0^1 dv
        \biggl( \left( \alpha-\sqrt{1-\alpha^2} \right)^2 G_{s-r}^2(y,v))
        + \left(1-\alpha^2\right) G_{t-r}^2(x,v) \\
    &\qquad +2 \left(\alpha-\sqrt{1-\alpha^2} \right)
        \sqrt{1-\alpha^2} G_{s-r}(y,v) G_{t-r}(x,v) \\
    &\qquad + 2 \alpha \sqrt{1-\alpha^2} G_{s-r}(y,v)
        (G_{t-r}(x,v)-G_{s-r}(y,v)) (\tilde{\lambda} \cdot \tilde{\mu}
        -1)\biggr).
\end{split}\end{equation*}
Recall the semigroup property
\begin{equation} \label{square}
    \int_0^1 dv \, G_{s-r}(y,v) G_{t-r}(x,v)=G_{s+t-2r}(x,y)
\end{equation}
(see \textsc{Walsh} \cite[(3.6)]{Walsh:86}).
We set $h:=t-s$ and change variables
[$\bar{r}:=s-r$] to obtain the following bound:
\begin{equation*}\begin{split}
    &f_1 \geq c \rho^2 \int_{0}^{\epsilon} dr \biggl(
        \left( \alpha-\sqrt{1-\alpha^2} \right)^2 G_{2r}(y,y) +
        \left( 1-\alpha^2 \right) G_{2h+2r}(x,x) \\
    &\qquad \qquad+2 \left( \alpha-\sqrt{1-\alpha^2}
        \right) \sqrt{1-\alpha^2} G_{h+2r}(x,y) \\
    &\qquad \qquad+ 2 \alpha \sqrt{1-\alpha^2} (G_{h+2r}(x,y)-G_{2r}(y,y))
        \left( \tilde{\lambda} \cdot \tilde{\mu} -1 \right)\biggr).
\end{split}\end{equation*}
Recall (\cite[p.318]{Walsh:86}), that
\begin{equation*}
    G_t(x,y)=p_t(x,y)+H_t(x,y),
\end{equation*}
where $H_t(x,y)$ is a continuous function that is
uniformly bounded over $(t,x,y)$ $\in$ $(0,\infty) \times (0,1) \times (0,1)$.
Therefore, $f_1 \geq c \rho^2 \tilde{g}_1 -c \epsilon$, where
\begin{equation*}\begin{split}
    \tilde{g}_1:=&\ \tilde{g}_1(h, \epsilon, \alpha, \tilde{\lambda}, \tilde{\mu},x,y)\\
    =& \int_{0}^{\epsilon} dr \biggl(\left( \alpha-\sqrt{1-\alpha^2} \right)^2
        p_{2r}(y,y) + \left( 1-\alpha^2 \right) p_{2h+2r}(x,x) \\
    &\qquad +2 \left( \alpha-\sqrt{1-\alpha^2} \right)
        \sqrt{1-\alpha^2} p_{h+2r}(x,y) \\
    &\qquad + 2 \alpha \sqrt{1-\alpha^2}
        \left( p_{h+2r}(x,y)-p_{2r}(y,y)\right)
        \left(\tilde{\lambda} \cdot \tilde{\mu}
        -1\right) \biggr).
\end{split}\end{equation*}
We can recognize that
\begin{equation*}\begin{split}
    p_{h+2r}(x,y)-p_{2r}(y,y) &= \frac{\exp(-(x-y)^2/(4(h+2r)))}{\sqrt{4 \pi (h+2r)}}
    -\frac{1}{\sqrt{4 \pi (2r)}}\leq 0.
\end{split}\end{equation*}
Also, $\tilde{\lambda} \cdot \tilde{\mu} -1 \leq 0$. Thus,
\begin{equation*}
    \tilde{g_1} \geq
  \hat{g_1},
\end{equation*}
where
\begin{equation*}\begin{split}
    \hat{g_1} := &\ \hat{g_1}(h, \epsilon, \alpha,x,y) \\
    =& \int_0^{\epsilon} dr \, \biggl(
        \left(\alpha-\sqrt{1-\alpha^2}\right)^2
        p_{2r}(y,y) + \left( 1-\alpha^2 \right) p_{2h+2r}(x,x)\\
    &\qquad   +2\left(\alpha-\sqrt{1-\alpha^2}\right)
        \sqrt{1-\alpha^2}  p_{h+2r}(x,y) \biggl).
\end{split}\end{equation*}
Therefore,
\begin{equation*}\begin{split}
    &\hat{g_1}= \int_0^{\epsilon} dr
    \biggl((\alpha-\sqrt{1-\alpha^2})^2 \frac{1}{\sqrt{8 \pi r}} +
    \left( 1-\alpha^2\right) \frac{1}{\sqrt{8 \pi (h+r)}} \\
    & \qquad \qquad + 2 (\alpha-\sqrt{1-\alpha^2}) \sqrt{1-\alpha^2}
    p_{h+2r}(x,y) \biggr).
\end{split}\end{equation*}

On the other hand, by (\ref{square}) above,
\begin{equation*}\begin{split}
    f_2&\geq \int_0^{\e \wedge (t-s)} dr \left( 1-\alpha^2\right) G_{2r}(y,y)\\
    &\geq \tilde{g}_2:= \int_0^{\e \wedge h} dr\, \left(1-\alpha^2 \right)
        p_{2r}(y,y)-C \e \\
    &=\left( 1-\alpha^2\right) \sqrt{\e \wedge h} - C \e.
\end{split}\end{equation*}
Finally, we conclude that
\begin{equation*}\begin{split}
    f_1+f_2 &\geq \hat{g_1} + \tilde{g}_2 - 2 C \e \\
    &= \left(\alpha-\sqrt{1-\alpha^2} \right)^2 \frac{\sqrt{\e}}{\sqrt{2
        \pi}} + \frac{1-\alpha^2}{\sqrt{2 \pi}}
        \left(\sqrt{h+\e}-\sqrt{h} \right) \\
    &\qquad +2\left(\alpha-\sqrt{1-\alpha^2}\right) \sqrt{1-\alpha^2}
        \int_0^{\e} dr \, p_{h+2r}(x,y)\\
    &\qquad + \frac{1-\alpha^2}{\sqrt{2
        \pi}} \sqrt{\e \wedge h} - 2C \e.
\end{split}\end{equation*}
Now we  consider two different cases.

\vskip 12pt

\noindent\textbf{Case (i).}
Suppose  $\alpha-\sqrt{1-\alpha^2} \geq 0$, that is, $\alpha \geq
2^{-1/2}$. Then
\begin{equation*}
    \epsilon^{-1/2} \left( \hat{g_1} + \tilde{g}_2 \right) \geq
    \phi_1\left(\alpha\,, \frac{h}{\e} \right) - 2 C \e^{1/2},
\end{equation*}
where
\begin{equation*}
    \phi_1(\alpha\,,z) := \frac{1}{\sqrt{2 \pi}}
        \Bigg( \left(\alpha-\sqrt{1-\alpha^2}\right)^2 +
        \left( 1-\alpha^2\right)
        \frac{1}{\sqrt{1+z}+\sqrt{z}}
        + \left( 1-\alpha^2\right) \sqrt{1 \wedge z} \Bigg).
\end{equation*}
Clearly,
\begin{equation*}\begin{split}
    \inf_{\alpha \geq 2^{-1/2} } \,
        \inf_{z >0} \, \phi_1(\alpha,z) &\geq
        \inf_{\alpha > 2^{-1/2} }
        \frac{\left(\alpha-\sqrt{1-\alpha^2}\right)^2 +
        c_0\left( 1-\alpha^2\right)}{\sqrt{2\pi}} \\
    &> \phi_0>0.
\end{split}\end{equation*}
Thus,
\begin{equation*}
    \inf_{\alpha \geq 2^{-1/2}, \, h \geq 0, \, 0 <
    \e \leq \epsilon_0} \, \e^{-1/2}
    \left(\hat{g_1} + \tilde{g}_2\right) >0.
\end{equation*}

\noindent \textbf{Case (ii).}
Now we consider the case where $\alpha-\sqrt{1-\alpha^2} <
0$, that is, $\alpha < 2^{-1/2}$. In this case, 
\begin{equation*}
    \e^{-1/2} \left( \hat{g_1} + \tilde{g}_2 \right) \geq \psi_1
    \left(\alpha\,, \frac{h}{\e} \right) - 2 C \e^{1/2},
\end{equation*}
where
\begin{equation*}\begin{split}
    &\psi_1(\alpha\,, z) : =\frac{1}{\sqrt{2 \pi}} \Bigg(
        \left( \alpha-\sqrt{1-\alpha^2} \right)^2
        + \left( 1-\alpha^2\right) \frac{1}{\sqrt{1+z}+\sqrt{z}} \\
    & \qquad \qquad \qquad  - 2\left(\sqrt{1-\alpha^2}-\alpha\right)
        \sqrt{1-\alpha^2}
        \frac{\sqrt{2}}{\sqrt{2+z}+\sqrt{z}}+\left( 1-\alpha^2\right) \sqrt{1 \wedge
        z} \Bigg).
\end{split}\end{equation*}
Note that $\psi_1(\alpha, z)>0$ if $\alpha \neq 0$.
This corresponds to the observation
made in the lines following (\ref{double}).
Moreover, for $\alpha
\geq \alpha_0 >0$, $\lim_{z \downarrow 0} \psi_1(\alpha, z) \geq
(2 \pi)^{-1/2} \alpha_0^2$, and
\begin{equation*}
    \lim_{z \uparrow + \infty} \psi_1(\alpha\,, z) \geq \inf_{\alpha>0} \,
    \frac{1}{\sqrt{2 \pi}} \biggl( (\alpha-\sqrt{1-\alpha^2})^2 +
    \left( 1-\alpha^2\right) \biggr)>0.
\end{equation*}
Therefore,
\begin{equation*}
\inf_{\alpha\in[\alpha_0,2^{-1/2}], \, z \geq 0} \,\psi_1(\alpha\,, z) >0.
\end{equation*}
This concludes the proof of the claim (\ref{minoration}).
\vskip 12pt

\noindent{\bf Case 2.} $t-s > \e$.
In accord with (\ref{large1}), we are
interested in
\begin{equation*}
    \inf_{1 \geq \alpha \geq \alpha_0}
    \prod_{i=1}^d \left(\xi^i \right)^{\sf T}
    \gamma_Z \xi^i :=\min(E_{1,\e},E_{2,\e}),
\end{equation*}
where
\begin{equation*}\begin{split}
    E_{1,\e} &:= \inf_{\alpha_0 \leq \alpha \leq \sqrt{1-\e^\eta}}
        \prod_{i=1}^d \left(\xi^i \right)^{\sf T}  \gamma_Z  \xi^i,\\
    E_{2,\e} &:= \inf_{\sqrt{1-\e^\eta} \leq \alpha \leq 1}
        \prod_{i=1}^d \left(\xi^i \right)^{\sf T}  \gamma_Z  \xi^i .
\end{split}\end{equation*}
Clearly,
\begin{equation*}
    E_{1,\e} \geq \frac{2}{3} f_2 - 2 I_{3,\e}.
\end{equation*}
Since $\alpha \leq \sqrt{1-\e^\eta}$
is equivalent to $\sqrt{1-\alpha^2} \geq \e^{\eta/2}$, we use hypothesis
{\bf P2} to deduce that
\begin{equation*}
    f_2 \geq c \rho^2 \e^\eta \int_{t-\e}^t dr
    \int_0^1 dv\, G^2_{t-r}(x,v) \geq
    c \rho^2 \e^{\frac{1}{2}+\eta}.
\end{equation*}
Therefore,
\begin{equation*}
    E_{1,\e} \geq c\rho^2 \e^{\frac{1}{2}+\eta} - I_{3,\e},
\end{equation*}
and we have seen that $I_{3,\e}$ has the desirable property $\E\left[\left| I_{3,\e}\right|^q\right] \leq c(q) \e^q$.

In order to  estimate $E_{2,\e}$, we observe using (\ref{large1}) that
\begin{equation*}
    E_{2,\e} \geq \frac{2}{3} \tilde f_1 - \tilde J_{1,\e} - \tilde J_{2,\e} -
    \tilde J_{3,\e} - \tilde J_{4,\e},
\end{equation*}
where
\begin{equation*}\begin{split}
    \tilde f_1 &\geq  \alpha^2 \sum_{k=1}^d \int_{s-\e}^s dr \int_0^1 dv\,
        \left(\sum_{i=1}^d \tilde \lambda^i \sigma_{ik}(u(r,v))\right)^2
         G^2_{s-r}(y,v), \\
    \tilde J_{1,\e} &= 2 \left( 1-\alpha^2\right)   \sum_{k=1}^d
        \int_{s-\e}^s dr \int_0^1
        dv\, \left(\sum_{i=1}^d \tilde \mu_i \sigma_{ik}(u(r,v))\right)^2
        G^2_{t-r}(x,v), \\
    \tilde J_{2,\e} &=  2 \left( 1-\alpha^2\right)   \sum_{k=1}^d
        \int_{s-\e}^s dr \int_0^1
        dv\, \left(\sum_{i=1}^d \tilde \mu_i \sigma_{ik}(u(r,v))\right)^2
        G^2_{s-r}(y,v), \\
    \tilde J_{3,\e} &= 2 \sum_{k=1}^d \int_{s-\e}^s dr \int_0^1 dv\,
        \left(\sum_{i=1}^d \left(\alpha \tilde\lambda^i -
        \tilde\mu_i\sqrt{1-\alpha^2}\right)
        a_i(k,r,v,s,y)\right)^2, \\
    \tilde J_{4,\e} &= 2 \left( 1-\alpha^2\right)
        \sum_{k=1}^d \int_{s-\e}^s dr \int_0^1
        dv\, \left(\sum_{i=1}^d \tilde \mu_i a_i(k,r,v,t,x)\right)^2 .
\end{split}\end{equation*}
Because $\alpha^2 \geq 1 - \e^\eta$ and $\e \leq t-s \leq \frac{1}{2}$,
hypothesis {\bf P2} and Lemma \ref{(A.3)} imply that $\tilde f_1 \geq c \e^{1/2}$. On the other hand,
since $1-\alpha^2 \leq \e^\eta$, we can use hypothesis {\bf P1} and Lemma \ref{int:G^2} to
see that
\begin{equation*}
    \E\left[\left|\tilde J_{1,\e}  \right|^q\right] \leq c(q) \e^{q\eta} \e^{q/2} = c(q)
    \e^{(\frac{1}{2}+\eta)q},
\end{equation*}
and similarly, using Lemma \ref{(A.5)}, $\E[|\tilde J_{2,\e}  |^q] \leq c(q) \e^{(\frac{1}{2}+\eta)q}$. The term
$ \tilde J_{3,\e}$ is equal to $2I_{1,\varepsilon}$, so $\E[\vert \tilde J_{3,\e}\vert^q] \leq c \varepsilon^q$, and $\tilde J_{4,\varepsilon}$ is similar to $B_1^{(2)}$ from (\ref{B12}), so we find using (\ref{EB3}) that
\begin{equation*}
    \E\left[\left|\tilde J_{4,\e}  \right|^q\right]
    \leq c \e^{q\eta} (t-s+\e)^{q/2} \e^{q/2} \leq c
    \e^{(\frac{1}{2}+\eta)q}.
\end{equation*}

We conclude that when $t-s > \e$, then $E_{2,\e} \geq c \e^{1/2} - \tilde J_\e$,
where $\E[| \tilde J_\e |^q] \leq c(q) \e^{(\frac{1}{2}+\eta)q}$. Therefore,
when $t-s > \e$,
 \begin{equation*}
    \inf_{1 \geq \alpha \geq \alpha_0}  \prod_{i=1}^d
    \left(\xi^i\right)^{\sf T}
    \gamma_Z  \xi^i \geq 
    \min\left(c \rho^2 \e^{\frac{1}{2}+\eta} - I_{3,\e} ~,~  c \e^\frac{1}{2}
    - \tilde J_\e\right).
\end{equation*}

Putting together the results of Case 1 and Case 2, we see that for $0 < \e \leq \frac{1}{2}$,
\begin{equation*}
\inf_{\| \xi\| = 1,\, 1 \geq \alpha \geq \alpha_0} \prod_{i=1}^d
\left(\xi^i\right)^{\sf T}  \gamma_Z  \xi^i\geq \min\left(c \rho^2 \e^{\frac{1}{2}+\eta} -
        I_{3,\e},~c\e^\frac{1}{2} - 2 I_\e \mathbf{1}_{\{\e
        \geq t-s \}} - \tilde J_\e \mathbf{1}_{\{\e < t-s \}}\right).
\end{equation*}
We take into account the bounds on moments of $I_{3,\e}$, $I_\e$ and $\tilde
J_\e$, and then use Proposition \ref{deter} to conclude the proof of Proposition \ref{large}.
\end{proof}

\section{Appendix}

On several occasions, we have appealed to the following technical estimates on the Green kernel of the
heat equation.
\begin{lem} \label{(A.1)} \textnormal{\cite[(A.1)]{Bally:98}}
There exists $C>0$ such that for any $0<s<t$ and $x,y \in [0,1]$, $x \neq y$,
\begin{equation*}
G_{t-s}(x,y) \leq C \frac{1}{\sqrt{2 \pi (t-s)}} \exp \biggl( -\frac{\vert x-y \vert^2}{2(t-s)}\biggr).
\end{equation*}
\end{lem}

\begin{lem}\label{(A.3)} \textnormal{\cite[(A.3)]{Bally:98}}
There exists $C>0$ such that for any $t\geq\epsilon>0$ and $x \in [0,1]$,
\begin{equation*}
\int_{t-\epsilon}^t \int_{x-\sqrt{\epsilon}}^{x+\sqrt{\epsilon}} G^2_{t-s}(x,y) dy ds \geq C \sqrt{\epsilon}.
\end{equation*}
\end{lem}


\begin{lem} \textnormal{\cite[(A.5)]{Bally:98}} \label{(A.5)}
There exists $C>0$ such that for any $\epsilon>0$,
$q<\frac{3}{2}$, $t\geq \epsilon$ and $x \in [0,1]$,
\begin{equation*}
\int_{t-\epsilon}^t \int_0^1 G^{2q}_{t-s}(x,y) dy ds \leq C \epsilon^{3/2-q}.
\end{equation*}
\end{lem}

\begin{lem}\label{int:G^2}
    There exists $C>0$ such that for all
    $0<a<b$ and $x \in [0,1]$,
    \begin{equation*}
        \int_a^b
        \int_{0}^1 G^2_s(x,y) \, dy ds \le C\frac{b-a}{
        \sqrt{b} + \sqrt{a}}.
    \end{equation*}
\end{lem}

\begin{proof}
Using Lemma \ref{(A.1)} and the change of variables
$z=\frac{x-y}{\sqrt{s}}$, we see that
\begin{equation*}\begin{split}
\int_a^b \int_{0}^1 G^2_s(x,y) \, dy ds & \le C \int_a^b
\int_{-\infty}^{\infty} \frac{1}{\sqrt{s}} e^{-z^2} \, dz ds \\
&=\tilde{C} \int_a^b \frac{1}{\sqrt{s}} \, ds = 2 \tilde{C}
(\sqrt{b}-\sqrt{a}),
\end{split}\end{equation*}
which concludes the proof.
\end{proof}
\vskip 12pt

The next result is a straightforward extension to $d\ge 1$ of \textsc{Morien} \cite[Lemma 4.2]{Morien:99} for $d=1$.
\begin{lem} \label{morien}
Assume ${\bf P1}$.
    For all $q\ge 1$, $T>0$ there exists $C>0$
    such that for all $T \ge t\ge s\ge \e>0$ and $0\le y\le 1$,
    \begin{equation*}
        \sum_{k=1}^d \E\left[ \left(
        \int_{s-\e}^s dr \int_0^1 dv\, \left|
        \sum_{i=1}^d D_{r,v}^{(k)} \left( u_i(t,y) \right) \right|^2
        \right)^q \right] \le C \e^{q/2}.
    \end{equation*}
\end{lem}

The next result is Burkholder's inequality for Hilbert-space-valued martingales.
\begin{lem} \label{valuedm} \textnormal{\cite[eq.(4.18)]{Bally:98}}
Let $H_{s,t}$ be a predictable $L^2(([0,t] \times [0,1])^m)$-valued process, $m \geq 1$. Then, for any $p>1$,
there exists $C>0$ such that
\begin{equation*}\begin{split}
&\E \biggl[ \bigg\vert \int_{([0,t] \times [0,1])^m} \biggl( \int_0^t \int_0^1 H_{s,y}(\alpha) W(dy, ds)\biggr)^2
d \alpha \bigg\vert^p \biggr] \\
&\qquad \qquad\leq C
 \E \biggl[ \bigg\vert   \int_0^t \int_0^1 \biggl( \int_{([0,t] \times [0,1])^m} H_{s,y}^2(\alpha) d \alpha \biggr)
dy ds \bigg\vert^p \biggr].
\end{split}\end{equation*}
\end{lem}

\vskip 12pt \noindent {\bf Acknowledgement}. The authors thank V.~Bally for several stimulating discussions.

\end{document}